\documentclass[reqno, 11pt]{jams-l}

\usepackage{amssymb}
\usepackage{amsthm}
\usepackage[colorlinks,citecolor=red,pagebackref,hypertexnames=false]{hyperref}
\usepackage{color}
\usepackage{esint}
\usepackage{bbm}
\usepackage{graphicx}
\usepackage{mathrsfs}
\usepackage[left=3.3cm, right=3.3cm, top=2.7cm, bottom=2.7cm]{geometry}
\usepackage{enumitem}
\usepackage[hang,flushmargin]{footmisc}

\def\dB{{\mathcal{B}}}

\def\dF{{\mathcal{F}}}

\def\dN{{\mathcal{N}}}
\def\dS{{\mathcal{S}}}

\def\bR{{\mathbb{R}}}

\def\bZ{{\mathbb{Z}}}

\def\cC{{\mathscr{C}}}
\def\cD{{\mathscr{D}}}

\def\cF{{\mathscr{F}}}
\def\cG{{\mathscr{G}}}

\def\ve{\epsilon} 
\def\vp{\varphi}

\renewcommand{\d}{{\partial}}
\def\lec{\lesssim}

\DeclareMathOperator{\diam}{diam}
\def\dim{\mathop\mathrm{dim}} 					
\def\dist{\mathop\mathrm{dist}} 						
\def\Int{\mathop\mathrm{Int}} 						

\newcommand{\ps}[1]{\left( #1 \right)}

\newcommand{\ck}[1]{\left\{#1 \right\}}

\newcommand{\floor}[1]{\left\lfloor #1 \right\rfloor}

\def\XXint#1#2#3{{\setbox0=\hbox{$#1{#2#3}{\int}$ }
\vcenter{\hbox{$#2#3$ }}\kern-.58\wd0}}

\newtheorem{theorem}{Theorem}[section]
\newtheorem{lemma}[theorem]{Lemma}

\theoremstyle{definition}
\newtheorem{definition}[theorem]{Definition}

\theoremstyle{remark}
\newtheorem{remark}[theorem]{Remark}

\numberwithin{equation}{section}

\newcommand{\R}{\mathbb{R}}
\newcommand{\N}{\mathbb{N}}
\newcommand{\Z}{\mathbb{Z}}

\newcommand{\hd}{\mathcal{H}^d}

\newcommand{\hdc}{\mathcal{H}^d_\infty}

\newcommand{\spt}{\mathrm{spt}}

\newcommand{\wt}{\widetilde}

\newcommand{\eqt}[1]{\ensuremath{\stackrel{#1}{=}}}
\newcommand{\leqt}[1]{\ensuremath{\stackrel{#1}{\leq}}}

\newcommand\blfootnote[1]{%
  \begingroup
  \renewcommand\thefootnote{}\footnote{#1}%
  \addtocounter{footnote}{-1}%
  \endgroup
}

\newcommand{\dH}{\mathcal{H}}
\numberwithin{equation}{section}
\theoremstyle{plain}

\newtheorem{proposition}[theorem]{Proposition}
\newtheorem{corollary}[theorem]{Corollary}
\newtheorem{notation}[theorem]{Notation}
\newtheorem{claim}[theorem]{Claim}
\newtheorem*{lemman}{Lemma}

\newcommand{\st}{\mbox{ s.t. }}
\newcommand{\Adj}{\mathrm{Adj}}

\newcommand{\Child}{\mathrm{Child}}

\newcommand{\dD}{\mathcal{D}}

\newcommand{\cc}{\mathsf{c}}

\newcommand{\aA}{\mathsf{A}}
\newcommand{\Cc}{\mathsf{C}}

\newcommand{\geqt}[1]{\ensuremath{\stackrel{#1}{\geq}}}

\newcommand{\Forest}{\text{Forest}}

\newcommand{\Tree}{\mathsf{Tree}}
\newcommand{\Stop}{\mathsf{Stop}}
\newcommand{\Top}{\mathsf{Top}}
\newcommand{\Bad}{\mathsf{Bad}}
\newcommand{\Nextt}{\mathsf{Next}}
\def\Next{{\mathsf Next}}
\newcommand{\Ss}{\mathsf{S}}

\newcommand{\BWGL}{\mathsf{BWGL}}

\newcommand{\1}{\mathbf{1}}
\newcommand{\Vv}{\mathsf{V}}
\newcommand{\Gg}{\mathsf{G}}

\newcommand{\jones}{\textrm{BETA}}

\newcommand{\doi}[1]{\textsc{doi}: \href{http://dx.doi.org/#1}{\nolinkurl{#1}}}

\makeatletter
\def\@makefnmark{%
  \leavevmode
  \raise.9ex\hbox{\fontsize\sf@size\z@\normalfont\tiny\@thefnmark}}
\makeatother

\newcommand{\RomanNumeralCaps}[1]
    {\MakeUppercase{\romannumeral #1}}

\title[Higher dimensional Jordan curves]{Higher dimensional Jordan curves}

\author{Michele Villa}
\address{School of Mathematics, University of Edinburgh, JCMB, Kings Buildings,
Mayfield Road, Edinburgh,
EH9 3JZ, Scotland.}
\email{m.villa-2 ``at" sms.ed.ac.uk}

\date{}

\dedicatory{}

\begin{document}

\maketitle
\begin{center}
\begin{minipage}[c][][r]{400pt}
\begin{small}
\textsc{Abstract.} 
We address the question of what is the correct higher dimensional analogue of Jordan curves from the point of view of quantitative rectifiability. More precisely, we show that `topologically stable' sets can be used as covering objects in Analyst's Travelling Salesman Theorem-type theorems: if $E$ is lower $d$-regular (in a certain suitable sense), then we show that there exists a topologically stable surface $\Gamma$ so that $E \subset \Gamma$ and
\begin{align*}
    \diam(E)^d + \beta^d(E) \approx \hd(\Gamma),
\end{align*}
where $\beta^d$ is a term quantifying the curvature of $E$. A corollary of the main result of this paper and a construction by Hyde \cite{hyde2020TST}, is a higher dimensional analogue of Peter Jones TST, valid for \textit{any} subset of Euclidean space. 
\end{small}
\end{minipage}
\end{center}   

\blfootnote{\textup{2010} \textit{Mathematics Subject Classification}: \textup{28A75}, \textup{28A12} \textup{28A78}.


M. V. was supported by The Maxwell Institute Graduate School in Analysis and its
Applications, a Centre for Doctoral Training funded by the UK Engineering and Physical
Sciences Research Council (grant EP/L016508/01), the Scottish Funding Council, Heriot-Watt
University and the University of Edinburgh. M.V. was also supported by the Academy of Finland via the project Incidences on Fractals, grant No. 321896.
}
\vspace{-.5cm}
\tableofcontents

\section{Introduction}
A set $E \subset \R^n$ is $d$-rectifiable if it can be covered, up to a set of $\hd$-measure zero, by a countable number of images of $\R^d$ under Lipschitz maps $f: \R^d \to \R^n$. Here $\hd$ denotes the $d$-dimensional Hausdorff measure which can be seen as a surface measure of $E$; recall also that $f:\R^d \to \R^n$ is $L$-Lipschitz if $|f(x)-f(y)| \leq L |x-y|$ for some $L\geq 1$. Rectifiable sets can be thought of as the measure theoretic analogues of smooth manifolds and they are central in many branches of analysis.

In the past thirty years several theorems appeared which quantify the \textbf{multiscale behaviour} of rectifiable sets. These results were motivated by, and essential to, the solution of long-standing and fundamental questions in analysis, notably the proof of Vitushkin's conjecture by David \cite{david1998unrectictifiable} or the solution to Painlev\'{e}'s problem by Tolsa \cite{tolsa03}. This loose family of theorems and techniques came to be known as \textbf{quantitative rectifiability}. The most well-known theorem pertaining to quantitative rectifiability is the so-called \textbf{Analyst's Travelling Salesman Theorem} (TST), originally proved by Jones in the \textit{Inventiones} paper \cite{jones90}: let $E \subset \R^n$ be compact. If $A \subset \R^n$ ($A$ will be a cube or a ball) and $\diam(A)$ denotes the diameter of $A$, we define $\beta_E(A)$ by setting $\beta_E(A) \diam(A)$ to be equal to the radius of the smallest cylinder containing $E \cap A$. This is a scale invariant measure of how flat, that is, how close to a line, $E$ is in $A$. Jones' TST says that $E$ may be covered by a rectifiable curve $\Gamma$ if and only if
$
\beta(E) := \sum \beta_E(3Q)^2 \ell(Q) <+ \infty,
$
where the sum is over all dyadic cubes and $\ell(Q)$ is the side length of $Q$. Moreover, if $\Gamma$ is the shortest curve containing $E$, we have
\begin{align}\label{e:TST}
 \diam(E) + \beta(E) \approx \mathcal{H}^1(\Gamma).
\end{align}
Jones' result quantifies in a precise manner how quickly a rectifiable set should become flat as we zoom in through the scales. With time, it found many applications (harmonic measure \cite{bishop1994harmonic}, Kleinian groups \cite{bishop1997hausdorff}, analytic capacity \cite{tolsa-book}, Brownian motion \cite{bishop1997dimension}, and graph networks \cite{azzam2012take}).

Jones' results holds for curves (1-dimensional objects) in space. The extension of quantitative results to $d$-rectifiable sets (with $d>1$) has mostly been developed under the \textbf{Ahlfors regularity} assumption, and with the aim of characterising uniformly rectifiable sets. We give some definitions: a set $E$ is \textbf{Ahlfors} $d$-\textbf{regular} (\textbf{$d$-AR}) if there exists a constant $C$ so that 
\begin{align*}
r^d/C \leq \hd(E \cap B(x,r)) \leq C r^d \quad \mbox{ for }  \quad x \in E \quad \mbox{ and } \quad 0< r< \diam(E). 
\end{align*}
(Throughout the article $B(x,r)$ is the open ball centered at $x$ and with radius $r>0$). A set $E$ is called \textbf{uniformly $d$-rectifiable} (or \textbf{UR}) if we can find constants $\theta>0$ and $L\geq 1$ so that for each $x \in E$ and $0<r< \diam(E)$, there exists an $L$-Lipschitz map $f:\R^d \to \R^n$ so that 
\begin{align*}
\hd(E \cap B(x,r)\cap f(B)) \geq \theta r^d,
\end{align*}
for some ball $B \subset \R^d$ also with radius $r$.
This is a quantitative notion of rectifiability - it is also stronger: if a set is UR it is also rectifiable.
Uniformly rectifiable sets were introduced by David and Semmes in \cite{david-semmes91} and \cite{david-semmes93}. While their motivation was in connection with singular integral operators (SIOs), 
uniform rectifiability is now an important notion in geometric measure theory and characterising UR sets in terms of their multiscale behaviour has become an active field of research. Many recent important  breakthroughs in the theory of singular integrals (for example \cite{ntv}) and in harmonic measure (see \cite{azzam2020harmonic}) fundamentally rely on these characterisations.

Currently, there is a \textbf{move toward results about UR sets beyond the Ahlfors regular} \textbf{setting}. Some examples are \cite{david2012reifenberg}, where David and Toro construct Lipschitz parameterisations for Reifenberg flat sets (with holes), or \cite{girela2018riesz}, which was vital in the solution of Bishop's Conjecture on harmonic measure \cite{azzam2016rectifiability}. With J. Azzam, we embarked on a program aimed at developing a quantitative theory of rectifiability for \textbf{lower content $d$-regular sets} (or \textbf{$d$-LCR}), building on a previous work by Azzam and Schul \cite{azzam2018analyst}. A set $E \subset \R^n$ is lower content $d$-regular\footnote{A more precise definition is in \eqref{e:LCR}.} if for each ball $B$ centered on $E$ with radius $r_B>0$ we have
\begin{align*}
\mathcal{H}_\infty^d (E \cap B) \gtrsim r_B^d.
\end{align*}
Roughly speaking this means that it cannot concentrate around objects with dimension smaller than $d$. 
 In the joint work with Azzam \cite{azzam2019quantitative} we show that many of the characterisations for UR sets still have a meaning in this more general setting. Some of the results proved here were later used by Azzam (\cite{azzam2019harmonic}) to give several applications to harmonic measure. 
 
In this paper, building on techniques developed in \cite{azzam2019quantitative}, we give a solution to the problem of `\textbf{higher dimensional Jordan curves}', which we presently describe. Jones' TST has been extended\footnote{See Section \ref{s:lit} for more relevant literature.} to curves lying in many different spaces (e.g. $\R^n$ in \cite{oki92}, Hilbert space in \cite{schul2007}, Heisenberg group in \cite{li2016traveling}). However, an Analyst's TST for \textit{higher dimensional objects} was not available until recently: Azzam and Schul in \cite{azzam2018analyst} proved a TST for lower content $d$-regular sets which however looked like
\begin{align}\label{e:AStst}
  \diam(E)^d + \beta^d(E) \approx \mathcal{H}^d (E) + \mathrm{Error}.
\end{align}
Here $\beta^d(E):=\sum \beta_E(3Q)^2 \ell(Q)^d$ and the $\beta$ coefficients are the correct variation of Jones' original ones (and they measure distance to a $d$-dimensional plane).
In contrast to \eqref{e:TST}, observe the \textbf{appearance of the error term}; note also that there is \textit{no mention of a covering surface}. To prove a TST \textit{\`{a} la} Jones \eqref{e:TST}, one needs to face a fundamental issue: \textbf{what is a `higher dimensional curve'}? (Topological balls will not work, for example). Here, we give an answer to this question by considering sets whose $d$-measure is stable under Lipschitz deformations: for each $x$ in the surface $\Sigma$ and each $0<r< \diam(\Sigma)$, we have\footnote{See Definition \ref{d:TC} for the precise definition.}
\begin{align}\label{e:stable-surfaces}
	\hd(B(x,r)\cap \vp_t(\Sigma)) \gtrsim  r^d,
\end{align}
whenever $\{\vp_t\}_{0\leq t \leq 1}$ is a family of (Lipschitz) deformations homotopic to the identity and so that $\vp_t(y)=y$ if $y \notin B(x,r)$. In English, $\Sigma$ cannot be continuously distorted into a lower dimensional set (just like Jordan curves). Similar topological conditions appeared previously in \cite{david2000uniform} and \cite{david2004hausdorff}, which also inspired parts of our proof.
For this class of sets, which we will call \textbf{topologically stable} $d$-surfaces, we are able to retrieve a TST \textit{a l\'{a}} Jones, \eqref{e:TST}, for lower content regular sets (see Theorems \ref{t:topo-main} and \ref{c:TST-lcr-stable}, Section \ref{s:actual-theorems}, for precise statements).

\begin{theorem}\label{t:main-sloppy}
Let $F \subset \R^n$ be lower content $d$-regular. There exists a topologically stable $d$-surface such that $F \subset \Sigma$ and
\begin{align*}
\diam(F)^d + \beta^d(F) \approx \hd(\Sigma). 
\end{align*}
If $\beta^d(F)< + \infty$ then $F$ and $\Sigma$ are $d$-rectifiable.
\end{theorem}
\noindent
Recently Hyde (\cite{hyde2020TST}) proved that for \textit{any} set $E \subset \R^n$, if there is a $d$ so that $\beta^d(E)< +\infty$ (here, again, the coefficients are the appropriate variant for the problem), then one can construct a lower content $d$-regular set $F$ so that $E \subset F$. Moreover it holds that $\beta(E) \approx \beta(F)$. Hence, a corollary of Theorem \ref{t:main-sloppy} together with Hyde's construction\footnote{This will be more thoroughly explained in Section \ref{s:actual-theorems}.} is a \textbf{full higher dimensional Travelling Salesman Theorem} for general sets. 
\begin{corollary}\label{hyde-villa}
Let $E \subset \R^n$. Then if there exists a $d$ so that $\beta^d(E) < \infty$, there exists $d$-LCR set $F$ and a \textup{(}rectifiable\textup{)} topologically stable $d$-surface $\Sigma$ so that $E \subset F \subset \Sigma$ and
\begin{align*}
 \diam(E)^d + \beta^d(E) \approx \hd(\Sigma).   
\end{align*}
\end{corollary}
\noindent
We conclude that topologically stable surfaces can be conceived as higher dimensional analogues to `Jordan curves' (from a TST perspective).

\subsection{Relevance and more applications}
The reader may understandably wonder whether Theorem \ref{t:main-sloppy} is at all interesting or useful. 
The author thinks it is interesting in the sense that, behind it, there is a natural question: what should a `higher dimensional curve' be? Although as stated such a question is imprecise and open to very many answer, if we add `\textit{from the point of view of the Analyst's TST}', it takes a decidedly more precise meaning. 
\noindent
As for the usefulness of the result, we provide a few more applications (with the hope that more will come). We state the results rather imprecisely, to convey some ideas. We postpone details to Section \ref{s:applications}. 

\begin{itemize}[leftmargin=0.5cm]
    \item We call a set uniformly non $d$-flat if the $d$-dimensional $\beta$-coefficients (as defined in \eqref{e:beta-p-cont}) are large at all scales\footnote{We postopone the precise definition to Section \ref{s:actual-theorems}.}. As first application of our main result, we show that \textit{if $E$ is $d$-stable and uniformly non $d$-flat, then it must have Hausdorff dimension larger than $d$.} This result was proven by David in \cite{david2004hausdorff} under stronger topological assumptions. We provide a quantitative strengthening of David's result (under a weaker topological condition) which gives a precise dependence between  $\beta$-coefficients and dimension. See Theorem \ref{t:non-flat}, Section \ref{s:actual-theorems}. 
    \item In \cite{semmes1995finding}, Semmes stated the following guiding principle to understand the relation between the topology of some set, and its `mass' distribution: \textit{`Suitable topological conditions on a space in combination with upper bounds on the mass often implies serious restrictions on the geometric complexity of the space.'} 
As a second application, we give a result that makes this principle precise: a topologically stable $d$-surface which is also upper regular is uniformly rectifiable. This result is similar to the main result in \cite{david2000uniform}. See Corollary \ref{c:UR-TC}, Section \ref{s:actual-theorems}.
\end{itemize}

\subsection{Sketch of the proof}
We give a brief idea of the proof. As a first step, we apply a construction from \cite{azzam2019quantitative} which gives a \textit{coronisation} of a $d$-stable surface $E$ by Ahlfors $d$-regular sets. One should imagine such a construction as a dyadic approximation of $E$ at certain scales. Next, following ideas coming from the proofs in \cite{david2004hausdorff}, we show that, first, the topological condition transfers from $E$ to the approximating sets. The fundamental observation is this: because the approximating sets are Ahlforse regular and satisfy the topological condition, they turn out to have large intersection with \textit{quasiminimal sets}. By the main result in \cite{david2000uniform}, sets of this type are (locally) \textit{uniformly rectifiable}. This will give certain estimates on the $\beta$-coefficients, which will then be transferred back to $E$, to obtain one direction of Theorem \ref{t:main-sloppy}. The other direction follows from the main result in \cite{azzam2018analyst}.

\subsection{Structure of the paper}
In Section \ref{s:prel} we gather some notation which will be used throughout the paper. In Section \ref{s:actual-theorems} we state precisely the main result (Theorem \ref{t:topo-main}) and the definition of the topological condition (Definition \ref{d:TC}). We give the precise statements of the two applications mentioned  above (Theorem \ref{t:non-flat} and Corollary \ref{c:UR-TC}) together with the relevant definitions in Section \ref{s:applications}.  In
Section \ref{s:rem} we give some remarks on the choice of the topological condition and show that some perhaps better known types of surfaces falls withing this category. In Section \ref{s:first-red} we start the proof in earnest with some preliminary reductions and by applying the corona construction of \cite{azzam2019quantitative}. In Section \ref{s:STC} we define a variant of the topological condition and show that it is `inherited' from $E$ by the approximating sets coming from \cite{azzam2019quantitative}. Section \ref{s:FFproj} introduces a fundamental tool, Federer-Fleming projections. In Section \ref{s:TC-STC} we show that the above mentioned `inheritance' effectively happens. Section \ref{s:E-rho-STC} presents a further family of approximating sets. 
In Sections \ref{s:functional} and \ref{s:quasiminimality} we show that the approximating sets lie close to a uniformly rectifiable set. Sections \ref{s:beta} and \ref{s:end} concludes the proof via some $\beta$-coefficients estimates. In Section \ref{s:sigma-construction} we construct a stable surface, given any lower content regular set. In Section \ref{s:non-flat} we give a proof of the application to uniformly non-flat sets mentioned above. 

\subsection{Acknowledgements}
The bulk of this work was done while I was a PhD student at Edinburgh under Jonas Azzam. I thank him for suggesting the problem, for his help, support and patience. I would also like to thank PCMI/IAS and the organisers of the graduate school of 2018 on Harmonic Analysis: it was here that I learnt many tools used in this paper, as explained by Guy David (whom I also thank for the clear explanations). Some revision was done while a postdoc under Tuomas Orponen, whom I thank for his support. I finally thank Matt Hyde for explaining to me his results. 
\section{Preliminaries}\label{s:prel}

We gather here some notation and some results which will be used later on.
We write $a \lesssim b$ if there exists a constant $C$ such that $a \leq Cb$. By $a \approx b$ we mean that there exists a constant $C >0$ so that $a \lesssim_C b \lesssim_C a$. If the constant $C$ depends on some parameter, say $\ve$, than we will write $a \lesssim_\ve b$ or $a \approx_\ve b$. If we have no interest in keeping track of parameter and or constant, we will simply write $a \lesssim b$ or $a \approx b$. If $A$ is a subset of some set $B$, then $A^c = B \setminus A$.
For sets $A,B \subset \R^n$, we let
\begin{align*}
	\dist(A,B) := \inf_{a\in A, b \in B} |a-b|.
\end{align*}
The diameter of a set $A$ is defined as $\diam(A):= \sup_{x,y \in A} |x-y|$.
For a point $x \in \R^n$ and a subset $A \subset \R^n$, 
$
	\dist(x, A):= \dist(\{x\}, A)= \inf_{a\in A} |x-a|.
$
We write $
	B(x, r) := \{y \in \R^n \, |\,|x-y|<r\}.$
Sometimes we will denote a ball simply by $B$; in this case its radius will be $r(B)$ or $r_B$ and its center $x_B$ or $z_B$. 

\subsection{Measure, content, Choquet integrals}
For a set $A\subset \R^n$, $d \geq 0$ and $\delta>0$, write
\begin{align*}
    \mathcal{H}_\delta^d(A) = \inf \left\{ \sum_{i} \diam(A_i)^d \quad |\quad  A \subset \bigcup_i A_i \quad {\rm and} \quad \diam(A_i) < \delta \right\}.
\end{align*}
The the $d$-dimensional Hausdorff \textit{measure} of $A$ is given by $\hd(A) = \lim_{\delta \to 0} \hd_\delta(A)$. The $d$-dimensional Hausdorff \textit{content}, on the other hand, is given by 
\begin{align} 
    \hd_\infty(A) =  \inf \left\{ \sum_{i} \diam(A_i)^d \quad |\quad  A \subset \bigcup_i A_i \right\}.\label{e:hdc}
\end{align}
For a function $f:\R^n \to [0, \infty)$, a set $A\subset \R^n$ and $0<p<\infty$,  the $p$-Choquet integral with respect to Hausdorff content $\hdc$ is defined as
\begin{align}\label{e:choquet}
    \int_A f(x)^p \, d \hdc(x) := \int_0^\infty \hdc \left( \left\{ x \in A \, |\,  f(x)> t \right\}\right) \, t^{p-1}\, dt.
\end{align}
Choquet integrals satisfy all the standard linearity properties of integral with respect to measures. We will highlight throughout the proofs when certain properties of Choquet integrals are needed. We (will) refer the reader to Section 2.1 of \cite{azzam2018analyst} (`Hausdorff measure and content') for the relevant lemmas. Following \eqref{e:choquet}, we will write
\begin{align*}
    \beta_E^{d,p}(B) = \inf_{L} \left(\int_B \left( \frac{\dist(y, L)}{r(B)}\right)^p \, d \hdc(y) \right)^{\frac{1}{p}}, 
\end{align*}
where $B$ is a ball. We will often write 
\begin{align}\label{e:beta-x-r}
    \beta_{E}^{d,p} (x,r) := \beta_{E}^{d,p} (B(x,r)). 
\end{align}

\subsection{Cubes, stopping times} 
\subsubsection{Dyadic cubes and skeleta}
For $j \in \Z$, we will denote by $\Delta_j$ the family of dyadic cubes with side length $2^{-j}$. We also set 
\begin{align*}
    \Delta := \bigcup_{j \in \Z} \Delta_j. 
\end{align*}
For a set $A \subset \R^n$, we put $\Delta(A):= \left\{ I \in \Delta \, |\, I \cap A \neq \varnothing\right\}.$
For a cube $I \in \Delta$ and an integer $1\leq k< d$, we write
\begin{align} \label{e:skeleta}
    \partial_k I
\end{align}
to denote the $k$-dimensional skeleton\footnote{This is precisely what the reader thinks.} of $I$. We also set
\begin{align} \label{e:Sjd}
    \dS_{j,k} := \bigcup_{I \in \Delta_j} \partial_k I.
\end{align}
Let us remark that for a set $V$, we write $\partial I$ to mean the standard boundary of $V$; so in particular $\partial I = \partial_{n-1} I$.

\subsubsection{Christ-David cubes}
In the (sloppy) statement of Theorem \ref{t:main-sloppy}, the summation was over a family of cubes $\dD$. These \textit{are not} the standard dyadic cubes; however the reader may think of them as dyadic `intrinsic' dyadic cubes of a set $E$. We may refer to these cubes as \textit{Christ-David cubes}, to distinguish them from the dyadic cubes. As it would be cumbersome to use write each time `Christ-David cube', in some sections we will simply call them cubes. However, there will always be a \textit{notational difference}: 
\begin{itemize}[leftmargin=0.5cm]
    \item Dyadic cubes will always be denoted by $I, J$.
    \item On the other hand, we will denote Christ-David cubes with $Q, P, R$. 
\end{itemize}

The following theorem gives the existence Christ-David cubes. A first version of the theorem is by David in \cite{David1988}; subsequent versions appeared in \cite{christ1990b} and \cite{hytonen2012non}. Before stating the theorem, recall that a metric space is  \textit{doubling} if there exists a number $N$ so that any ball can be covered by $N$ balls of half the radius. We will apply Theorem \ref{t:Christ} below with $X=E$, where $E$ is a subset lying in the Euclidean space $\R^n$; in particular $E$ (as metric space) will always be doubling. 

\begin{theorem} \label{t:Christ}
Let $X$ be a doubling metric space. Let $X_{k}$ be a nested sequence of maximal $\lambda^{k}$-nets\footnote{Recall that for a constant $c>0$ a $c$-net $A$ is a set so that $|x-y|>c$ for all $x,y \in A$.} for $X$ where\footnote{The exact value of these constants in unimportant.} $\lambda<1/1000$ and let $c_{5}=1/500$. For each $n\in\bZ$ there is a collection $\dD_{k}$ of `cubes', which are Borel subsets of $X$ such that the following hold.
\begin{enumerate}
\item For every integer $k$, $X=\bigcup_{Q\in \dD_{k}}Q$.
\item If $Q,Q'\in \dD=\bigcup \dD_{k}$ and $Q\cap Q'\neq\emptyset$, then $Q\subseteq Q'$ or $Q'\subseteq Q$.
\item For $Q\in \dD$, let $k(Q)$ be the unique integer so that $Q\in \dD_{k}$ and set $\ell(Q)=5\lambda^{k(Q)}$. Then there is $z_{Q}\in X_{k}$ so that
\begin{equation}\label{e:containment}
B(z_{Q},c_{5}\ell(Q) )\subseteq Q\subseteq B(z_{Q},\ell(Q))
\end{equation}
and $ X_{k}=\{z_{Q}\; |\; Q\in \dD_{k}\}$.
\end{enumerate}
\end{theorem}

We now set some notation. For $Q \in \dD$, we denote by $B_Q$ the ball $B(z_Q, \ell(Q))$. Also, some jargon: for a cube $Q \in \dD$, we denote $\dD(Q)$ the family of cubes in $\dD$ which are also subsets of $Q$. If $Q \in \dD_k$, then the cubes $Q' \in \dD_{k+1}(Q)$ are called the \textit{children} of $Q$, while $Q$ is the (unique) \textit{parent} cube of any $Q' \in \dD_{k+1}(Q)$, or the \textit{ancestor} of any $Q' \in \dD(Q)$. For a Christ-David cube $Q$, the collection of its children is denoted by $\Child(Q)$. Two cubes having the same parent are surprisingly called \textit{siblings}.

\subsubsection{Stopping times}
Later on we will need to partition $\dD$ into subfamilies of cubes. These subfamilies will be often called \textit{stopping-time} regions and are defined as follows. 
\begin{definition}\label{d:stopping}
A collection of Christ-David cubes $\Ss \subseteq \dD$ is a {\it stopping-time region} or {\it tree} if the following hold:
\begin{enumerate}
\item There is a cube $Q(\Ss)\in \Ss$ that contains every cube in $\Ss$.
\item If $Q\in \Ss$, $R\in \Ss$, and $Q\subseteq R\subseteq Q(\Ss)$, then $R\in \Ss$.
 \item If $Q\in \Ss$ and there is $Q'\in \Child(Q)\backslash \Ss$, then $\Child(Q)\subset \Ss^{c}$.
\end{enumerate}
\end{definition}

\begin{definition}\label{d:maxfam}
Given a family $\cF$ of cubes (either dyadic or Christ-David), the subcollection of \textit{maximal cubes} from $\cF$ is the subset of $\cF$ composed of cubes which have the largest sidelength and such that any other cube in $\cF$ is contained in one of the cubes in this subcollection. 
\end{definition}

\subsection{Constants}
We collect here all the various different constants that will be used throghout the paper. 
\begin{enumerate}[leftmargin=0.5cm]
    \item \label{i:dim} $n, d$: the dimension of the ambient space and of the set under consideration, respectively.
    \item $\cc_1$: the lower content regularity constant. It first appears in \ref{e:LCR}. \label{i:lcr}
    \item $\aA$: it determines how much we are inflating the ball where we are measuring the $\beta$ number. It first appears in Theorem \ref{t:AS-real}. \label{i:Aa} 
    \item $\Cc_0$: it determines the expansion of the ball where we are measuring the BWGL. It first appears in Theorem \ref{t:AS-real}. \label{i:Cc0}
    \item $\epsilon$: tolerance parameter in BWGL (see \eqref{e:BWGL-1}), and in the definition of Reifenberg-flatness. 
    \item $\cc_2$ and $\cc_3$: constants in the definition of the modified content, first appearing in \eqref{e:mod-cont1}.
    \item $c_1$: constant appearing in the definition of Semmes surfaces, see Definition \ref{def:semmes-surface}.
    \item $C_1$: expansion factor of top cubes in Lemma \ref{l:corona}.
    \item $\tau$: smoothing parameter in Lemma \ref{l:corona}.
    \item $k_0$: generation parameter in Lemma \ref{l:corona}.
    \item $M$: constant for the stopping time in the construction of Lemma \ref{l:corona}.
    \item $\lambda$: nets parameter in Theorem \ref{t:Christ}.
    \item $c_5$: containment parameter in Theorem \ref{t:Christ}.
    \item $r_0, \alpha_0, \eta_0, \delta_0, \gamma_0$: parameters of the topological condition \eqref{e:TC}.
    \item $r_1, \alpha_1, \eta_1, \delta_1$: parameters for the skeletal topological condition (see \eqref{e:TND-constants}).
    \item $C_2$: constant of the skeletal topological condition. See Definition \ref{d:TND}.
    \item $\Cc_3$: Ahlfors regularity constant of the approximating set $E_R$ (and of $E_\rho)$. See the paragraph above Lemma \ref{l:TC-STC-Erho}.
    \item $\rho$: scale parameter of the approximating set $E_\rho$. See \eqref{e:rho}.
    \item $\sigma$: scale parameter for the construction of the domain of the fucntional $J$. See \eqref{e:sigma}.
    \item $M$: large constant in the functional $J$ (not the same $M$ as above!). See \ref{e:M}. 
    \item $c_2$: small constant in the definition of $M$ (see \eqref{e:M}).
    \item $k$: quasiminimality: Hausdorff measure constant. See \eqref{e:QMconst} and \eqref{e:QMmain}.
    \item $\delta$: quasiminimality: locality constant. See \eqref{e:QMconst} and the properties of the Lipschitz deformation $\phi$ \eqref{e:QM-W}-\eqref{e:QM-homo}.
    \item $c_3$: small constant in the choice of $\delta$. See \eqref{e:delta-choice}.
    \item $C_4$: inflation constant for the $\beta$ numbers on $Z_Q$.
\end{enumerate}

\section{Actual theorems}\label{s:actual-theorems}
In this section we state our main results and those of Azzam and Schul and of Hyde in a precise manner. Note that all the statements below are \textit{local}, that is, they concern an arbitrary Christ-David cube of the relevant set.

\subsection{Azzam and Schul's TST}
\subsubsection{Issues in higher dimensions}
It has been mentioned in the introduction that it took quite a while to arrive at a higher dimensional version of the TST.  There are two reasons for this.

\begin{enumerate}[leftmargin=.5cm]
    \item The first is that which motivates the current paper: what to use as a covering surface in higher dimensions?  One could legitimately think about, for example, topological spheres; see Figure \ref{f:octopus} for why this would not be a good candidate.
    \item The second issue is technical in nature: in any situation where one is dealing with sets of dimension larger than one, Jones' $\beta_{E, \infty}$ coefficients (as defined in the Introduction) become rather useless: in his PhD thesis, X. Fang constructed a codimension one Lipschitz graph $G$ in $\R^3$ with $\beta_{\infty}^2(G)=\infty$, see \cite{fang1990} or Example 1.16 in \cite{azzam2018analyst}.
    A further version of Jones' coefficients was introduced by David and Semmes in \cite{david-semmes91}. We will define them later on; for the moment it will suffice to say this: David and Semmes' coefficients make sense for `higher dimensional' sets, as long as such sets are \textit{a priori} \textit{Ahlfors regular}. Ahlfors regularity is a size condition\footnote{It is not a regularity condition: the four corner Cantor set is Ahlfors $1$-regular and purely unrectifiable.}  which impose a certain relationship between the diameter of a portion of the set and its mass.
 However, Jones's theorem does not assume this; indeed, $E$ need not to have \textit{a priori} finite $d$-dimensional measure, let alone be Ahlfors regular. Thus, a new type of coefficient measuring flatness is needed.
 \end{enumerate}
 
 \begin{figure}[hbt!]\label{f:octopus}
 \centering
 \includegraphics[scale=0.5]{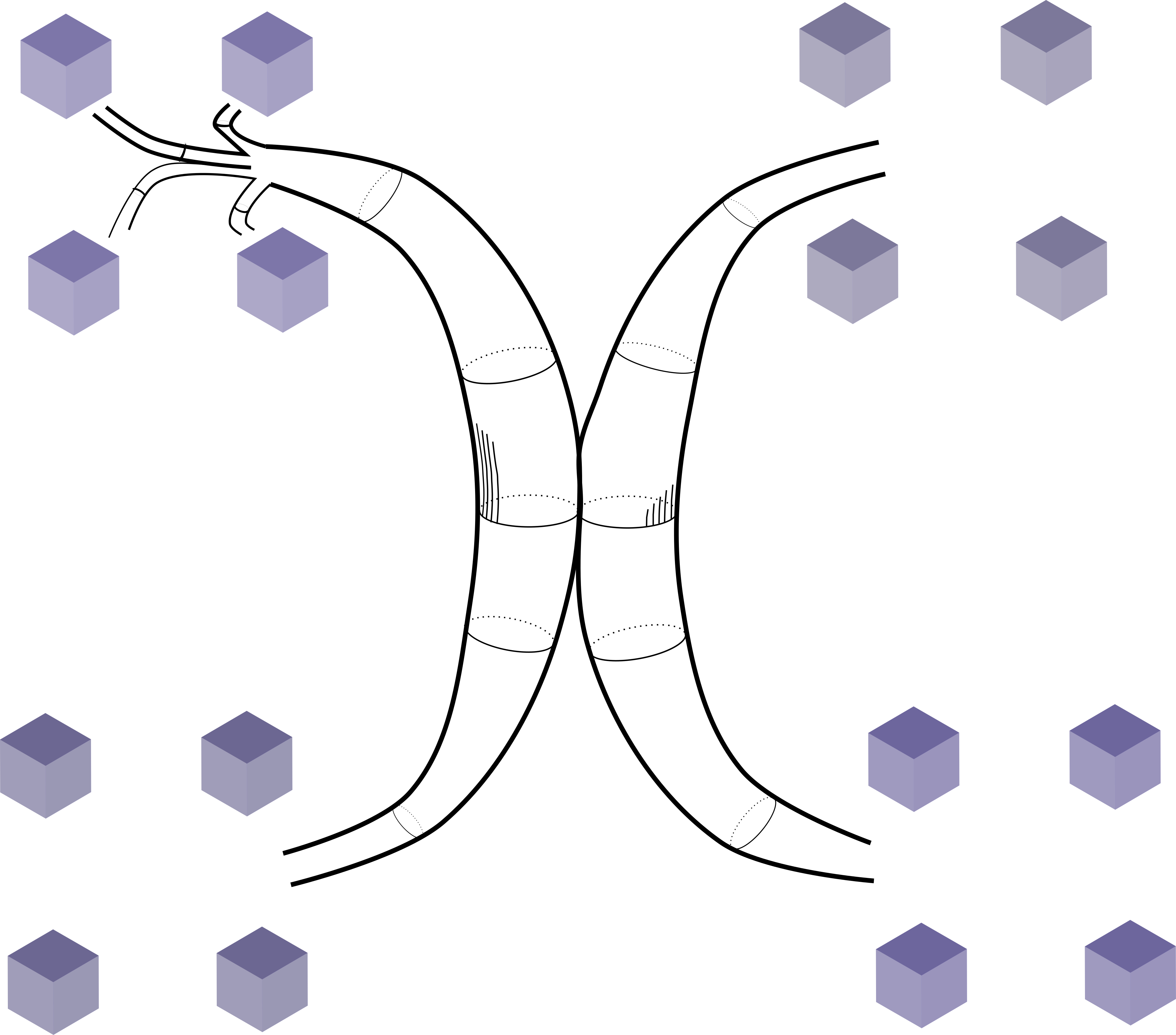}
 \caption{Given the $2$-dimensional  8-corner Cantor set in $\R^3$, one can construct a $2$-dimensional surface with finite measure, so that the closure of this surface will contain the Cantor set and will be homeomorphic to the $2$-sphere.}
\end{figure}

\noindent
 Notwithstanding these difficulties Azzam and Schul proved in \cite{azzam2018analyst} a version of Jones' theorem for sets of dimension larger than one in Euclidean space. 
\begin{enumerate}[leftmargin=0.5cm]
    \item To deal with the first difficulty,
Azzam and Schul decided to focus on obtaining a quantitative result of type \eqref{e:TST} directly for a set $E$ lying in $\R^n$, without trying to find a `covering object'. To be able to make some progress, they imposed on $E$ a size condition called \textit{lower content regularity}. 
\begin{definition}
We say that a set $E \subset \R^n$ is \textit{lower content} $d$-\textit{regular} with constant $\cc_1 \leq1$, or lower content $(d, \cc_1)$-regular, if 
\begin{align} \label{e:LCR}
    \hd_\infty(B(x,r) \cap E) \geq \cc_1 r^d
\end{align}
for all $(x, r) \in E \times (0, \diam(E))$.
\end{definition}
\noindent
The symbol $\hdc$ represents Hausdorff content, which has been defined in \eqref{e:hdc}.
Roughly speaking, lower content regularity prevents a set $E$ to concentrate around objects of dimension smaller than $d$. It is a natural assumption to impose on $E$, and it appears, in a form or another, in works on harmonic measure for example (see \cite{akman2019absolute}, \cite{akman2019rectifiability}, \cite{azzam2019harmonic})\footnote{Lack of lower bounds on the mass of a set (whether measured with a measure or a content) is highly problematic, as it becomes hard to say anything about its geometry.}. 

\item To deal with the second difficulty, Azzam and Schul introduced the following variant of the $L^p$-type $\beta$ coefficients\footnote{The $L^p$ $\beta$-coefficients are the ones that David and Semmes came up with.}. For a set\footnote{This set will usually be a ball or a cube. In this case $\diam(B)$ will be substituted by the radius of the ball or the side length of the cube.} $B \subset \R^n$
\begin{align}
    \beta_E^{d,p}(B) = \inf_{L} \left(\int_B \left( \frac{\dist(y, L)}{r(B)}\right)^p \, d \hdc(y) \right)^{\frac{1}{p}}, 
\end{align}

\begin{align} \label{e:beta-p-cont}
\beta_E^{d,p} (B) := \inf_L \left( \frac{1}{r^d} \int_0^1 \hdc (\{y \in B\cap E \,|\, \dist(y,L)> t \, \diam(B)\}) t^{p-1} \, dt \right)^{\frac{1}{p}},
\end{align}
where the infimum is taken over all affine $d$-planes $L$ in $\R^n$. The integral on the right hand side of \eqref{e:beta-p-cont} is a Choquet integral. Azzam and Schul chose $\hdc$ to define their $\beta$ coefficients because $\hdc$ \textit{does not blowup}. Indeed, for any set $E$, we always have that $\hdc(E)\lesssim \diam(E)^d$, even when the Hausdorff dimension of $E$ is larger than $d$. We refer the reader to Examples 1.14, 1.15 and 1.16 in \cite{azzam2018analyst} for some examples which motivate their choice of $\beta$-coefficients.
\end{enumerate}

\subsubsection{What Azzam and Schul proved}
We need to introduce a little more notation. Given  two closed sets $E$ and $F$, and $B$ a set we denote
\begin{align}\label{e:loc-haus}
d_{B}(E,F)=\frac{2}{\diam B}\max\left\{\sup_{y\in E\cap B}\dist(y,F), \sup_{y\in F\cap  B}\dist(y,E)\right\}
\end{align}
Recall that, by Theorem \ref{t:Christ}, for each Christ-David cube $Q\in \dD$, there is a ball $B_{Q}$ centered on, containing, and of comparable size to, $Q$. For $\Cc_0>0$, and $\ve>0$, let
\begin{align}\label{e:BWGL-1}
{\mathsf{ BWGL}}(\Cc_0,\ve)=\{Q\in \dD| \;\; d_{\Cc_0\, B_{Q}}(E,P)\geq \ve \;\; \mbox{ for all $d$-planes } P\}. 
\end{align}
We can now state the result from \cite{azzam2018analyst} (Theorems I and II there). This is the estimate which appeared in the Introduction as \eqref{e:AStst}. To be sure, our phrasing is slightly different to the original one, but the interested reader can find a justification of this reformulation in the Appendix of \cite{azzam2019quantitative}.
\begin{theorem}[\cite{azzam2018analyst}, Theorems I and II]\label{t:AS-real}
Let $1\leq d<n$ and $E\subseteq \bR^{n}$ be a closed set. Suppose that $E$ is lower content $d$-regular with constant $\cc_1$ and let $\dD$ denote the Christ-David cubes for $E$.  Let $\Cc_0>1$. Then there is $\ve>0$ small enough so that the following holds. 
Let\footnote{The range of $p$ is always the same for all subsequent theorems. This is the choice of $p$  which we alluded to in the introduction.} $1\leq p<p(d)$ where
\begin{equation}
\label{e:pd}
p(d):= \left\{ \begin{array}{cl} \frac{2d}{d-2} & \mbox{if } d>2 \\
 \infty & \mbox{if } d\leq 2\end{array}\right. . 
 \end{equation}
For $R\in \dD$, let
    \begin{align}\label{e:BWGL}
  {\rm BWGL}(R):=  {\rm BWGL}(R,\ve,\Cc_0)=\sum_{ Q\in {\mathsf{BWGL}}(\ve,\Cc_0) \atop Q\subseteq R } \ell(Q)^{d},
      \end{align}
    and
    \begin{align} \label{e:jones-function-2}
    \textup{\jones}_{E,d, p, \aA}(R) := \sum_{Q \in \dD(R)} \beta_{E}^{d,p}(\aA \, B_Q)^{2}\ell(Q)^{d}.
    \end{align}
    Then for $R\in \dD$,
\begin{equation}
\label{e:tst}
\hd(R)+  {\rm BWGL}(R,\ve,\Cc_0)
 \approx_{\aA,n,\cc_1,p, \Cc_0,\ve} \textup{\jones}_{E,\aA,p}(R) + \ell(R)^d.
 \end{equation}
\label{t:AS}
\end{theorem}
\noindent
Let us recall that the coefficients $\beta_E^{p,d}$ are those defined in \eqref{e:beta-p-cont}. These coefficients enjoy some standard\footnote{Standard in the sense that also the coefficients introduced by Jones and David and Semmes enjoy these properties. It is rather important for the content $\beta$'s to satisfy these properties - this allowed Azzam and Schul to use familiar techniques from the theory of uniformly rectifiable sets.} properties such as a sort of monotonicity with respect to balls' radii and an  `ability' to jump between two close-by sets without changing their value too dramatically. We will highlight when we use these properties and cite the relevant lemmas from \cite{azzam2018analyst}. 

\begin{remark}\label{rem:BWGL}
The presence of $\rm BWGL(E)$ in \eqref{e:tst} is natural: in Jones's theorem the right-hand-side only contains the length of the covering curve; now, a curve has no holes\footnote{Sherlock, personal communication.}. However, $E$ may very well be quite broken (even while being lower content regular). Thus, if we imagine our set $E$ being covered by `a higher dimensional curve' $\Gamma$, we would have $\hd(\Gamma) \approx \hd(E) + \text{BWGL}(E)$, where $\text{BWGL}(E) \approx \hd( \Gamma \setminus E)$ (see Figure \ref{f:bwgl}).

\begin{figure}[h]
\begin{minipage}[l]{0.45\linewidth}
        \centering
        \includegraphics[scale=0.4]{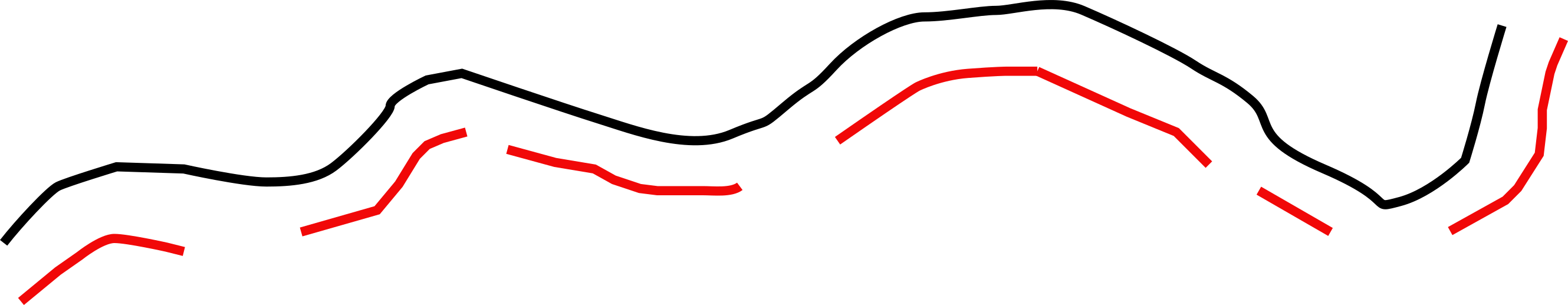}
\end{minipage}
\begin{minipage}[l]{0.45\linewidth}
\centering 
\includegraphics[scale=0.4]{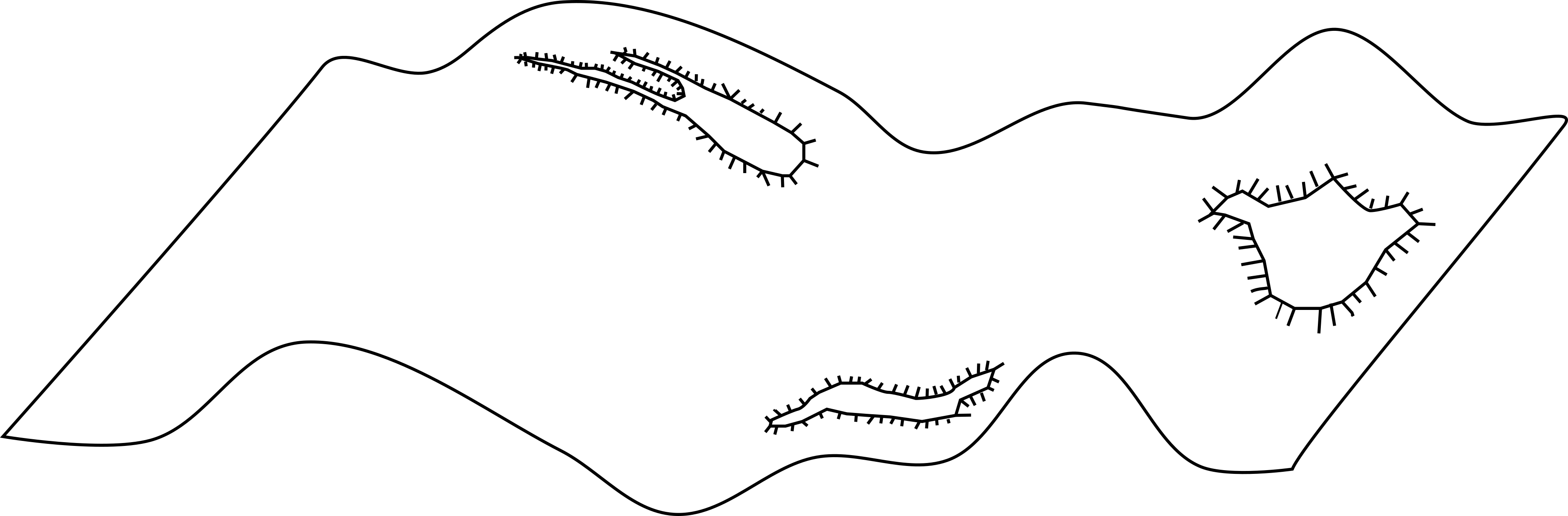}
\end{minipage}
\caption{A covering curve picks up the holes in the set (left). In the absence of a `higher dimensional curve', the term BWGL accounts for the holes (right).} \label{f:bwgl}
\end{figure}
\end{remark}

\subsection{Stable $d$-surfaces} \label{ss:TC}\

\noindent
We now define precisely the topological condition mentioned in \eqref{e:stable-surfaces}. Our definition is a weaker version of that of David in \cite{david2004hausdorff}.
Let $E$ be a \textit{closed} subset of $\R^n$. 
\begin{definition}[Allowed Lipschitz deformations with parameter $\alpha_0$] \label{d:ALD} Fix a constant $0< \alpha_0 <1$.
Consider a one parameter family of Lipschitz maps $\{\varphi_t\}$, $0 \leq t \leq 1$, defined on $\R^n$. For $x \in E$ and $r>0$, we say that $\{\vp_t\}_{0\leq t \leq 1}$ is an \textit{allowed Lipschitz deformation} with parameter $\alpha_0$, or an $\alpha_0$-ALD, if it satisfies the following four conditions: 
\begin{align}
    & \vp_t(\overline{B}(x,r) )\subset \overline{B}(x,r) \mbox{ for each } t \in [0,1]; \label{e:vp1}\\
    & \mbox{ for each } y \in \R^n,\, t \mapsto \vp_t(y) \mbox{ is a continuous function on } [0,1]; \label{e:vp2}\\
    & \vp_0 (y) = y \mbox{ for all } y \in \R^n \mbox{ and } \vp_{t}(y) = y \mbox{ for } t \in [0,1] \mbox{ whenever }  y \in \R^n \setminus B(x,r); \label{e:vp3}\\
    & \dist(\vp_t(y), E) \leq \alpha_0 r \mbox{ for } t \in [0,1] \mbox{ and } y \in E \cap B(x,r). \label{e:vp4}
\end{align}
\end{definition}
\noindent
Thus if we perturb a set $E$ in a ball $B(x,r)$ with an $\alpha_0$-ALD, what we do is effectively to move in a thin tube around $E \cap B(x,r)$. 
The topological condition that we impose on $E$ is the following. 
\begin{definition}[Stable $d$-surface] \label{d:TC}
Fix five parameters:
\begin{align}
    & r_0, \mbox{ the scale parameter;} \label{e:r_0}\\
    & \alpha_0, \mbox{ the distance parameter;} \label{e:alpha_0}\\
    & \delta_0, \mbox{ the lower regularity parameter; } \label{e:delta_0} \\
    & \eta_0, \mbox{ the boundary parameter;} \label{e:eta_0}\\
    & \gamma_0, \mbox{ the `other ball' parameter.} \label{e:gamma_0}
\end{align}
We say that a \textit{closed} set $E$ is a \textit{\textup{(}topologically\textup{)} stable $d$-surface with parameters} $r_0, \alpha_0, \delta_0$, $\eta_0$ \textit{and}  $\gamma_0$, if for each $x_0 \in E$ and $0<r<r_0$ we can find a ball $B$ such that\footnote{Recall that for a ball $B$ we denote its center by $x_B$ and its radius by $r_B$.}
\begin{center}
\begin{gather}
     B \mbox{ is centered on } E;\label{e:Bcent}\\
     B \subset B(x_0,r); \label{e:Bin}\\
     \gamma_0 r \leq r(B) \leq r;\label{e:r_B}\\
    \hd\ps{B(x_B, (1-\eta_0)r(B)) \cap \vp_1(E) } \geq \delta_0 r^d, \label{e:TC} \tag{TC}
\end{gather}
  \end{center}
whenever $\vp_t$ is an $\alpha_0$-ALD relative to $(x_B,r(B))$. 
We will refer to the above condition on $E$ as the `\textit{topological condition'}, or TC for short. 
\end{definition}

\noindent
This definition looks rather technical. However there are plenty of stable surface.
An obvious example is that of a curve. Also a square in the plane is a stable $1$-surface, simply because the left hand side of \eqref{e:TC} is always infinite. We will present more interesting examples in Section \ref{s:rem}, namely sets satisfying the so called \textit{Condition B} and \textit{Reifenberg flat} sets. On the other hand, the surface depicted in Figure \ref{f:octopus} (the many-tentacled octopus) is not a stable $2$-surface: at any scale one can shrink the diameter of the tubes as much as one wants (note that there is no uniform bound on the Lipschitz constant of an ALD). By doing so, one can make the Hausdorff $2$-measure of the surface of the tubes arbitrarily small, hence breaking \eqref{e:TC}\footnote{
This example suggests that a stable $d$-surface should really look $d$-dimensional: the many-tentacled octopus in Figure \ref{f:octopus} is effectively made of tubes, which are, in a way, $2$-dimensional approximations of ($1$-dimensional) pieces of curves. The topological condition excludes such instances by requiring the set to be spread out (in a ball) like a $d$-plane. }


\subsection{Statement of the main result}
We now state Theorem \ref{t:main-sloppy} as the following (more precise) statement.
\begin{theorem}\label{t:topo-main}
Assume that $1 \leq d < n$. Let $E \subset \R^n$ be a stable $d$-surface with parameters $r_0, \alpha_0, \delta_0$ and $\eta_0$ and let $\dD$ be its Christ-David cubes. Let $Q_0 \in \dD$ be an arbitrary cube satisfying $\ell(Q_0) \leq r_0$ and let $\aA$ be a sufficiently large positive constant. Take $1 \leq p < p(d)$ with 
\begin{align*}
p(d):= \left\{ \begin{array}{cl} \frac{2d}{d-2} & \mbox{if } d>2 \\
 \infty & \mbox{if } d\leq 2\end{array}\right. 
 \end{align*}
 Then it holds that
\begin{align}\label{e:main}
    {\rm diameter}(Q_0)^d + {\rm BETA }(Q_0) \approx \hd(Q_0),
\end{align}
where 
\begin{align}\label{e:BETA}
    {\rm BETA}(Q_0) = {\rm BETA}_{E, \aA, p, d} (Q_0): = \sum_{Q \in \dD(Q_0)} \beta_E^{d,p} (\aA B_Q)^2 \ell(Q)^d,
\end{align}
and where the constant behind the symbol $\approx$ depends on the parameters of $\eqref{e:TC}$ \textup{(}$\alpha_0, \delta_0, r_0, \eta_0$, $\gamma_0$\textup{)}, on $\aA$, on $p$, and on the dimensional parameters \textup{(}$d$, $n$\textup{)}.
\end{theorem}

\noindent
The assumption that $\ell(Q_0) \leq r_0$ is a natural one and cannot be avoided. Assuming the topological condition from a certain scale, i.e. from $r_0$, means that at larger scale there could be holes. This would make the term $\text{BWGL}(Q_0)$ come back. 
The assumption that $E$ is closed is relevant, but not restrictive; in fact we have that $\hd(E) = \hd(\overline{E})$. 
A corollary of Theorem \ref{t:topo-main} is the following. 
\begin{corollary}\label{c:TST-lcr-stable}
Fix $\ve>0$ sufficienlty small and $\Cc_0>1$. Let $E \subset \R^n$ be a lower content $(d, \cc_1)$-regular set (see \eqref{e:LCR}). Let $Q_0 \in \dD$. There exists a set $\Sigma= \Sigma(Q_0, \ve, \Cc_0)$ such that
\begin{enumerate}[leftmargin=1cm, label=\textup{(}\textup{\arabic*}\textup{)}] 
    \item $Q_0 \subset \Sigma$. 
    \item $\Sigma$ is a stable  $d$-surface with parameters  $r_0= \ell(Q_0)$,  $\eta_0, \gamma_0, \alpha_0$   sufficiently small with respect to $\ve>0$  and $\delta_0$ sufficiently small depending on $\eta_0, \alpha_0, \gamma_0$ and $d$.
    \item We have the estimate 
    \begin{align}
        {\rm diameter}(Q_0)^d + {\rm BETA}_{E, d, p, \aA}(Q_0)  \approx \hd(\Sigma),
    \end{align}
    where ${\rm BETA}_{E, d, p, \aA}(Q_0)$ is as in \eqref{e:BETA} and the constant behind the symbol $\approx$ depend on $n,d, \cc_1, \ve, \Cc_0, \aA$ and $(\alpha_0, \eta_0, \gamma_0, \delta_0, \gamma_0)$. 
\end{enumerate}
\end{corollary}
\noindent The corollary above gives a TST for lower content regular sets, and uses as `curves' stable $d$-surfaces. Glaringly, one would like to drop the lower content regularity assumption. The following construction, due to Hyde, does precisely this. 

\subsection{Hyde's construction}
In the recent paper \cite{hyde2020TST}, Hyde shows that for any set $F \subset \R^n$, one can construct a lower content $d$-regular set $E$ with some parameter $\cc_1$ such that $F \subset E$ and, importantly, one has comparability between the ${\rm BETA}$ term of $F$ and that $E$ (as was mentioned in the introduction). The removal of the lower content regularity assumption presents an important issue here - we spell it out for clarity: how to pick up the geometry of a set which has potentially very little mass? The ($d$-dimensional) $\beta$ coefficients of Azzam and Schul\footnote{This is the tool through which we `picked up' the geometry in the previous theorems.} (see \eqref{e:beta-p-cont}) will always return a small number if computed on a set with small $d$-dimensional content, even if the set is extremely non-flat\footnote{In this case, being flat means lying close to a $d$-dimensional affine plane.}. 
It is clear then that one needs to devise a different `geometry detector' to say something useful about sets which have no lower bounds on their mass. Hyde's idea is simple yet ingenious: he defines a different content\footnote{As opposed to the Hausdorff one.}, denoted by ${\wt \hdc}$, with respect to which \textit{any set  $E$} \textit{becomes lower content $d$-regular}. Then he proceeds to define a $\beta$ coefficient via Choquet integration against $\wt \hdc$ - these $\beta$ coefficient will be able to detect non-flatness even when a set has very little (Hausdorff-type) mass. Let us be more precise (see Definition 2.9 in \cite{hyde2020TST}).
Fix two constants $\cc_2$, which should be thought of as small, and $\cc_3$, which should be thought of as large. For a set $E \subset \R^n$ and a ball $B$, we say that a collection of balls $\dB$ is \textit{good} if for all $B' \in \dB$, $x \in E \cap B$ and $0< r< r(B)$, we have
\begin{align}\label{e:mod-cont1}
    & \sum_{\substack{B' \in \dB\\ B' \cap B(x,r) \cap B \cap E \neq \varnothing}} r(B')^d \geq \cc_2 r^d \quad  \mbox{ and } \quad
    \sum_{\substack{B' \in \dB\\ B' \cap B(x,r) \cap B \cap E \neq \varnothing \\ r(B') \leq r}} r(B')^d \leq \cc_3 r^d.
\end{align}
Then the above mentioned content of a set $A \subset E \cap B$ is given by
\begin{align*}
    {\wt \hdc} (A) = {\wt \hdc}^E_B (A) := \inf \left\{ \sum_{\substack{ B' \in \dB \\ B' \cap A \neq \varnothing}} r(B')^d \quad | \quad \dB \mbox{ is good for } E \cap B \right\}. 
\end{align*}
\noindent We refer the reader to \cite{hyde2020TST}, Section 2.4 for details on and properties of this modified content.
The $\beta$ coefficients are then defined as in \eqref{e:beta-p-cont} (see \cite{hyde2020TST}, Definition 2.20\footnote{In \cite{hyde2020TST}, the notation is somewhat different: Azzam and Schul's coefficients are denoted by $\check{\beta}$, while the coefficients in \eqref{e:H-beta} are simply denoted by $\beta$.}): for $1\leq p < \infty$, $E \subset \R^n$, $B$ a ball centered on $E$ and $L$ a $d$-dimensional affine plane, set
\begin{align}
    {\wt \beta}_E^{p,d}(B)&  :=\inf_{L} \left( \frac{1}{r(B)^d} \int_{E \cap B} \left( \frac{\dist(y, L)}{r(B)}\right)^p \, d {\wt \hdc} (x) \right)^{\frac{1}{p}} \nonumber\\
    & =\inf_{L}\left( \frac{1}{r(B)^d} \int_0^1 {\wt \hdc} (\{y \in B\cap E \,|\, \dist(y,L)> t r(B)\}) t^{p-1} \, dt \right)^{\frac{1}{p}}. \label{e:H-beta}
\end{align}
We now state the main results in \cite{hyde2020TST}.
\begin{theorem}[{\cite{hyde2020TST}, Theorem 1.13}]\label{t:TST-H1}
Let $1 \leq d < n$, $\aA > 1$ and $1< p < p(d)$ \textup{(}$p(d)$ as in \eqref{e:pd}\textup{)}. There exists a constant $\cc_4$ such that the following holds. Suppose $E\subset F \subset \R^n$, where $F$ is lower content $d$-regular with some constant $c \geq \cc_4$. Let $\dD(E)$ and $\dD(F)$ be the Christ-David cubes for $E$ and $F$, respectively. Let $Q_0^E \in \dD(E)$ and let $Q_0^F$ be the cube in $\dD(F)$ with the same center and side length as $Q_0^E$. Then 
\begin{align} \label{e:TST-H1}
    \diam(Q_0^E)^d + \sum_{Q \in \dD(Q_0^E)} {\wt \beta}_{E}^{d,p}(\aA B_Q)^2 \ell(Q)^d \lesssim_{\aA, c, n, p} \diam(Q_0^F)^d + \sum_{Q \in \dD(Q_0^F)} {\wt \beta}_F^{d,p}(\aA B_Q)^2 \ell(Q)^d. 
\end{align}
\end{theorem}
\begin{theorem}[{\cite{hyde2020TST}, Theorem 1.14}]\label{t:TST-H2}
Let $n,d, \aA$ and $p$ be as above. Let $E \subset \R^n$, and let $Q_0^E \in \dD(E)$ be so that $\diam(Q_0^E) \geq \theta \ell(Q_0^E)$ for some $0< \theta \leq 1$. Then there exists a lower content $d$-regular set $F$ with constant\footnote{The constant $\cc_4$ is as in Theorem \ref{t:TST-H1}.} $\cc_4$ such that the following holds. Let $\dD(F)$ be the Christ-David cubes for $F$ and let $Q_0^F$ be the cube in $\dD(F)$ with the same center and side length as $Q_0^E$. Then 
\begin{align}
    \diam(Q_0^F)^d + \sum_{Q \in \dD(Q_0^F)} \beta_F^{d,p}(\aA B_Q)^2 \ell(Q)^d  \lesssim_{\aA, \cc_4, n, p, \theta} \diam(Q_0^E)^d + \sum_{Q \in \dD(Q_0^E)} \beta_{E}^{d,p} (\aA B_Q)^2 \ell(Q)^d.
\end{align}

\end{theorem}

\subsection{A TST for general sets}
We now state Corollary \ref{hyde-villa} precisely. See also Corollary 1.16 in \cite{hyde2020TST}. As mentioned above, this corollary represents a full higher dimensional counterpart to the theorem of Jones from 1990 \eqref{e:TST}. 
\begin{corollary}
Let $E \subset \R^n$ and take $1 \leq d <n$, $\aA >1$ and $1 \leq p < p(d)$, with $p(d)$ as in \eqref{e:pd}. Then there exists a stable $d$-surface $\Sigma$ with parameters $(r_0, \alpha_0, \delta_0, \eta_0)$ such that $E \subset \Sigma$ and, if $Q_0^E \in \dD(E)$,  
\begin{align}
    {\rm diameter}(Q_0^E)^d + {\wt {\rm BETA}}(Q_0^E) \approx \hd(\Sigma),
\end{align}
where
\begin{align*}
    {\wt {\rm BETA}}(Q_0^E) = {\wt {\rm BETA}}_{d,p,\aA} (Q_0^E) = \sum_{Q \in \dD(Q_0^E)} {\wt \beta}_E^{d,p} (\aA B_Q)^2 \ell(Q)^d
\end{align*}
\end{corollary}

\section{Applications}\label{s:applications}
In this section we properly introduce the two applications of the main result Theorem \ref{t:topo-main} which were briefly mentioned in the introduction. 

\subsection{Uniformly non-flat sets}
In \cite{david2004hausdorff}, David proved that if $E$ is a stable $d$-surface (as in Definition \ref{d:TC}) and it is uniformly non-flat, then it must have dimension strictly larger that $d$. David's result was in the spirit of a previous result by 
 Bishop and Jones about uniformly wiggly, or uniformly non-flat, sets in the plane. 
 A set $E \subset \R^2$ is called \textit{uniformly wiggly} or \textit{uniformly non-flat} with parameter $\beta_0$ if for all dyadic cubes $I \in \Delta(E)$, we have that
 \begin{align*}
     \beta_{E,\infty}(I) > \beta_0 >0.
 \end{align*}
 \noindent
 Let us now recall the result of Bishop and Jones. It was later generalised to metric spaces by Azzam in \cite{azzam2015hausdorff}.
 \begin{theorem}[\cite{bishop1997wiggly}, Theorem 1.1] \label{t:nonflat}
 Let $E \subset \R^2$ be a compact, connected subset which is uniformly wiggly with parameter $\beta_0$. Then $\dim(E) > 1+ C \beta_0^2$, where $C$ is an absolute constant\footnote{We remark once more that by dimension we mean Hausdorff dimension. See \cite{mattila}, Definition 4.8 for a definition.}.
 \end{theorem}
 \noindent Let us go back to David's result. His is, in a sense, a generalisation of Bishop and Jones's theorem. However, it is of qualitative nature, and the dependence of the lower bound on the parameter $\beta_0$ is not explicit. See Theorem 1.8 in \cite{david2004hausdorff}. We give a quantitative strengthening of David's result where this dependence is made explicit. This result is a fairly immediate application of Theorem \ref{t:topo-main} and of the scheme of proof from \cite{bishop1997wiggly}. 
 
\begin{theorem} \label{t:non-flat}
Let $E \subset \R^n$ be a stable $d$-surface, where $1\leq d <n$ and let $\aA> 1$ and $1 \leq p < p(d)$, with $p(d)$ as in \eqref{e:pd}. Suppose that $Q_0 \in \dD$ is such that, for any $Q \in \dD(Q_0)$, we have that
\begin{align}
    \beta_E^{p,d}(\aA B_Q)^2 > \beta_0 >0.
\end{align}
Then
\begin{align}
    \dim(Q_0) > d + c \beta_0^2. 
\end{align}
\end{theorem}
\noindent
A quick corollary is the following.
\begin{corollary} \label{c:non-flat}
Let $E \subset \R^n$ be a stable $d$-surface, where $1 \leq d < n$. Then if for all dyadic cubes $I \in \bigtriangleup (I)$ with $\ell(I) \leq \diam(E)$, $\beta_{E, \infty}(I) > \beta_0$, then $\dim(E) > d + c \beta_0^2$.
\end{corollary}
\noindent
Both statements will be proven in Section \ref{s:non-flat}.

\subsection{Connection to uniform rectifiability}
In this subsection, we state a corollary which makes Semmes' principle, as mentioned in the introduction, precise. 
One of the most famous results from UR theory is the following characterisation of uniform rectifiability (see \cite{david-semmes91}, (C3)).
\begin{theorem}\label{t:DS-beta}
An Ahlfors regular set $E$ is uniformly $d$-rectifiable if and only if there exists a constant $C$ such that for all $x \in E$ and $0< r< \diam(E)$, we have that
\begin{align} \label{e:DS-beta}
\int_{E \cap B} \int_0^{r(B)} \beta_{E, p}^d(x,r)^2 \, \frac{dr}{r} \, d\hd(x) \leq C\, r(B)^d. 
\end{align}
\end{theorem}
\noindent
An immediate corollary of this and Theorem \ref{t:topo-main} is the following.
\begin{corollary}\label{c:UR-TC}
    Let $E \subset \R^n$ be a stable $d$-surface (as in Definition \ref{d:TC}) which also satisfies $\hd(E \cap B) \lesssim r(B)^d$. Then $E$ is a uniformly $d$-rectifiable set. 
\end{corollary}
\noindent
We remark that, most likely, the upper regularity hypothesis is not needed. A possible conjecture is the following: if $E \subset \R^n$ satisfies a condition as in Definition \ref{d:TC}, except that instead of \eqref{e:TC} we assume the existence of a constant $k \geq 1$ so that
\begin{align*}
    \hd(E \cap B(x,r)) \leq k \hd(\vp_t(E) \cap B(x,r)),
\end{align*}
then $E$ is Ahlfors $d$-regular and $d$-UR. This is hinted by the work of David and Semmes on quasiminimal sets \cite{david2000uniform}. Note that a condition as above would be a priori weaker than that assumed in \cite{david2000uniform}.
\section{Some remarks on the topological condition} \label{s:rem}
We would like to motivate a little bit our choices: why would one use the topological condition as in Definition \ref{d:TC}? The quantitative bound \eqref{e:main} was already known for surfaces satisfying the so called Condition B and for Reifenberg flat sets; as mentioned above, both of them imply the topological condition \eqref{e:TC}. 
\noindent As Condition B applies only to subsets of codimension one, let us consider instead a more general property which make sense in any codimension. Subsets satisfying this property are called Semmes surfaces. They were first introduced by G. David in \cite{David1988}. 
\begin{definition}\label{def:semmes-surface}
Let $n, d$ be two integers with $0 \leq d\leq n-1$. A (local\footnote{The definition is local, since it holds up to a maximum radius $r_0$.}) \textit{Semmes $d$-surface} is a closed subset $E \subset \R^n$ satisfying the following: we can find a constant  $c_1>0$ and a radius $r_0>0$ so that for all points $x_0 \in E$ and radii $0<r<r_0$, there exists  an affine subspace $W$ of dimension $n-d$ and a sphere $S$ of dimension $n-d-1$ so that
\begin{align}
    & S \subset B(x_0, r/2) \cap W, \\
    & \dist( S, E) \geq 2 c_1 r, \mbox{ and } \\
    & S \mbox{ links } E \mbox{ in } B(x_0, r). 
\end{align}
Let us explain what we mean by $S$ \textit{links} $E$ in a ball $B(x_0, r)$; we say that $S$ and $E$ are linked if it is \textit{not possible to find} an homotopy $f_t(y)$ defined and continuous for all $(y,t) \in \R^n \times [0,1]$  such that
\begin{align}
    & f_t(\overline{B}(x_0,r)) \subset \overline{B}(x_0,r) \mbox{ for each } t \in [0,1];\\
    & \mbox{for each } y \in \R^n, f_t(y) \mbox{ is a  continuous function of } t \in [0,1];\\
    & f_t(y) = y \mbox{ for } (y,t) \in \R^n \times \{0\}, \mbox{ and } f_t(y)=y \mbox{ for }  (y,t)  \in \R^n \setminus B(x_0,r) \times [0,1]; \\
    & \dist(f_t(y), S) \geq c_1 r \mbox{ for } t \in [0,1] \mbox{ and } y \in E \cap B(x_0,r);
\end{align}
and 
\begin{align}
    f_1(E \cap B(x_0,r)) \subset \partial B(x_0,r).
\end{align}
\end{definition}
\noindent In other words it is not possible to deform $E$ into the boundary of the ball $B(x_0,r)$ while at the same time keeping it away from the sphere $S$. 
Note that a set satisfying Condition B is just a $d$-dimensional Semmes surface with $d=n-1$. 
David shows the following.
\begin{lemma}[{\cite{david2004hausdorff}, Lemma 2.16}]\label{l:ss-top}
A Semmes $d$-surface is a stable $d$-surface with parameters $\alpha_0, \delta_0, \eta_0$, $\gamma_0$, depending on $c_1$ and with the same $r_0$.
\end{lemma}
\noindent
In fact, David shows that a Semmes $d$-surface satisfies the condition he introduces in \cite{david2004hausdorff}, which implies our condition in Definition \ref{d:TC}.

\noindent We now turn to Reifenberg flat sets. 
\begin{definition}\label{d:reif}
Let $n,d$ as above, and fix two positive constants $\epsilon$ and $r_0$. A closed subset $E \subset \R^n$ is called a (local) $d$-dimensional \textit{Reifenberg $\epsilon$-flat} set if for all $(x,r) \in E \times (0, r_0)$, there exists a $d$-dimensional affine plane $L$ so that
\begin{align}
    d_{x,r} (E, L) < \epsilon, 
\end{align}
where $d_{x,r}= d_{B(x,r)}$ is as in \eqref{e:loc-haus}.
\end{definition}
\noindent
Perhaps the most typical example of a Reifenberg $\ve$-flat set is the so called Koch snowflake. We recall Reifenberg's well-known topological disk theorem - we state it rather informally. For a more precise statement, we refer the reader to \cite{david2012reifenberg}, Theorem 1.1.
\begin{theorem}[Reifenberg Topological Disk Theorem, \cite{reifenberg1960solution}]\label{t:reif}
For all choice of integers $0<d<n$, and for $\ve>0$ sufficiently small, if a closed set $E\subset \R^n$ is locally a $d$-dimensional Reifenberg $\ve$, then there exists a bi-H\"{o}lder map (with H\"{o}lder parameter depending on $\ve$) $g: \R^n \to \R^n$ and a $d$-dimensional affine plane $P$ so that $g(P) = E$ locally.
\end{theorem}
\noindent It is then immediate that, for $\ve>0$ sufficiently small, if $E$ is a $d$-dimensional Reifenberg $\ve$-flat set (with local radius $r_0$), then it is a stable $d$-surface with parameters depending on $\ve$ and $r_0$.

\section{Remarks on literature}\label{s:lit}
To keep the introduction at a manageable length, we refrained from mentioning much of the relevant literature. After feeling rather guilty about this, we added this short section. However, the literature being rather plentiful, we have certainly omitted something. We apologise in advance for this. 

\subsection{Works in connection to the TST and the $\beta$ numbers. } 

The generalisation to $\R^n$  by Okikiolu \cite{oki92} and to Hilbert spaces by Schul \cite{schul2007} were mentioned in the introduction. For other TST-type results in Euclidean space connected to the $\beta$-coefficients, see \cite{pajot1996sous} and \cite{lerman2003quantifying}.
For Banach space results, see \cite{david2019sharp} and \cite{badger2020subsets}. 
TST-type results appeared also in the context of Heisenberg and Carnot geometry, see \cite{ferrari2007geometric}, \cite{li2014upper}, \cite{li2014upper}, \cite{chousionis2019traveling}.
Some results of TST-type in metric spaces are available in \cite{hahlomaa2005menger, hahlomaa2007curvature}, albeit also concerning Menger curvature. See also the H\"{o}lder parameterisation result \cite{badger2019holder}.

\subsection{Pointwise characterisations.}
The $\beta$ coefficients has been used to characterise rectifiable sets and measures.
The first result of this type is \cite{bishop1994harmonic}. This has been extended to the higher dimensional setting under a variety of hypothesis: see the series by Azzam and Tolsa \cite{tolsa2015characterization, azzam2015characterization}; the works of Badger and Schul \cite{badger2015multiscale} and \cite{badger2017multiscale}; the works of Edelen, Naber and Valtorta \cite{edelen2016quantitative}, \cite{edelen2019effective}.
See also
\cite{villa2020tangent}.
Characterisations of rectifiable sets inspired by works on the $\beta$-coefficients are given for example with respect to Menger curvature, see the works of Leger  \cite{leger1999menger} and Goering \cite{goering2018characterizations}, or Tolsa's $\alpha$ numbers (\cite{tolsa2009uniform}, \cite{dkabrowski2019necessary}, \cite{dkabrowski2019sufficient}), or center of mass \cite{mayboroda2009boundedness} and \cite{villa2019square},
or in terms of behaviour of SIOs, see \cite{tolsa2008principal}, \cite{villa2019omega}.

There has been a lot of work done to find sufficient conditions in terms of $\beta$ coefficients (or similar) to parameterise sets. An example, \cite{david2012reifenberg} will be heavily used in Section \ref{s:sigma-construction} (see also \cite{david2008generalization}  and \cite{toro1995geometric}). For higher order results in this direction see \cite{ghinassi2017sufficient} (and the related \cite{del2019geometric}).
\section{First reductions and the construction of approximating skeleta}\label{s:first-red}
Sections \ref{s:first-red} to \ref{s:end} will be devoted to the proof of Theorem \ref{t:topo-main}. We will first prove the following inequality (the converse being an immediate consequence of Azzam and Schul's result Theorem \ref{t:AS-real}, as we will see in Section \ref{s:end}).   
\begin{proposition} \label{p:main-one-dir}
Assume that $1 \leq d < n$. Let $E \subset \R^n$ be a stable $d$-surface with parameters $r_0, \alpha_0, \delta_0$ and $\eta_0$ and let $\dD$ be its Christ-David cubes. Let $Q_0 \in \dD$ be an arbitrary cube satisfying $\ell(Q_0)< r_0$ and let $\aA>1$. Take $1 \leq p < p(d)$ with 
\begin{align*}
p(d):= \left\{ \begin{array}{cl} \frac{2d}{d-2} & \mbox{if } d>2 \\
 \infty & \mbox{if } d\leq 2\end{array}\right. 
 \end{align*}
 Then it holds that
\begin{align}\label{e:main6}
    {\rm diameter}(Q_0)^d + {\rm BETA }(Q_0) \lesssim \hd(Q_0),
\end{align}
where 
\begin{align}\label{e:BETA6}
    {\rm BETA}(Q_0) = {\rm BETA}_{E, \aA, p, d} (Q_0): = \sum_{Q \in \dD(Q_0)} \beta_E^{d,p} (\aA B_Q)^2 \ell(Q)^d,
\end{align}
and where the constant behind the symbol $\lesssim$ depends on the parameters of $\eqref{e:TC}$ \textup{(}$\alpha_0, \delta_0, r_0, \eta_0, \gamma_0$\textup{)}, on $\aA$, on $p$, and on the dimensional parameters \textup{(}$d$, $n$\textup{)}.
\end{proposition}

\subsection{First reductions and lower content regularity of stable surfaces}
 \noindent Fix a Christ-David cube $Q_0 \in \dD$, as in the statement of Proposition \ref{p:main-one-dir}. First, we see that if $\hd(Q_0) = \infty$, then Proposition \ref{p:main-one-dir} is trivially true and hence there is nothing to prove.  Thus, we may (and will) assume that
\begin{align} \label{e:hdE}
    \hd(Q_0)< + \infty
\end{align}
We can also take $E$ to be compact, since Theorem \ref{t:topo-main} is local.
To obtain the estimates on $\beta$ coefficients that we want, we would like to apply a coronisation of lower content regular sets by Ahlfors regular sets proved in \cite{azzam2019quantitative} (we will state in later on). To do so, we first need to show that any topologically stable $d$-surface is lower content $d$-regular. 
\begin{lemma} \label{l:TP-LCR}
Let $E \subset \R^n$ be compact stable $d$-surface with parameters $r_0, \alpha_0, \delta_0$, $\eta_0$ and $\gamma_0$. Then $E$ satisfies
\begin{align*}
    \hdc(E \cap B(x,r)) \geq \cc_1 r^d
\end{align*}
for all $x \in E$ and $r< r_0$; the lower regularity constant $\cc_1$ will depend on $\delta_0$, $\eta_0$ and $\gamma_0$. 
\end{lemma}
\noindent This fact is essentially present in Chapter 12 of \cite{david2000uniform}, although in a somewhat different form. We give a proof for this reason. 
We will first prove the following further Lemma, which will imply Lemma \ref{l:TP-LCR}.
\begin{lemma} \label{sl:hdc-sl}
Let $E$ be a compact subset of $\R^n$; fix two positive and sufficiently small constants $\delta_0$ and $\eta_0$, and let $B$ be a ball centered on $E$ satisfying
\begin{align} \label{e:small-content}
    \hdc(B \cap E) < \mu \nu \, \delta_0 r(B)^d,
\end{align}
for a parameter $\nu$ \textup{(}sufficiently small depending on $\eta_0$\textup{)} and a number $\mu>0$ which depend  only on $\eta_0$ and $\delta_0$. Then there exists a one parameter family of Lipschitz mappings $\{\vp_t\}$ which satisfies \eqref{e:vp1}-\eqref{e:vp4} and so that $\vp_1$ maps $B(x_B,(1-\eta_0)r(B)) \cap E$ into the $(d-1)$-dimensional skeleton of cubes from $\Delta_{j}$, where $j=j(\rho) \in \N$ is such that $2^{-j} \approx \rho$, and  $\rho = (\nu \delta_0)^{1/d} r(B)$. 
\end{lemma}
\noindent
The proof of this lemma will follow quickly if we use the following proposition from \cite{david2000uniform}.
\begin{proposition}[{\cite{david2000uniform}, Proposition 12.61}]\label{p:content} 
Let $A$ be a union of dyadic cubes from $\Delta_j$, where $j$ is some integer. There is a possibly small constant $c>0$ so that if $\theta \approx c\, 2^{-j}$, the following is true. Let $F$ be a compact subset of $A$ such that
\begin{align} \label{e:theta-small}
    \hdc(F \cap I) < \theta \mbox{ for all } I \in \Delta_{j}.
\end{align}
Then there is a Lipschitz mapping $\phi: F \to A$ so that $\phi(F) \subset \dS_{j, d-1}$ and $\phi(F \cap I) \subset I$ for all $I \in \Delta_j$. Also, $\phi$ is homotopic to the identity through mappings from $F$ to $A$.  
\end{proposition}
\begin{proof}[Proof of Lemma \ref{sl:hdc-sl}]
To ease notation, set $x_B = x$ and $r(B)=r$. Let $\rho>0$ and $j(\rho) \in \N$ be as in the statement of the lemma, and let $\mu>0$, $\nu>0$ two possibly small parameters to be fixed soon. Set
\begin{align} 
    & A^1:= \bigcup \ck{ I \in \Delta_{j(\rho)} \, |\, I \cap B\ps{x,(1-\mu)r} \neq \emptyset}, \label{e:A1}\\
    & A^2 := \bigcup \ck{ I \in \Delta_{j(\rho)} \, |\, I \cap A^1 \neq \emptyset}. \label{e:A2}
\end{align}
We want $\mu$ and $\nu$ to be so that
\begin{align}\label{e:mu-nu}
    \eta_0 > 10 \, \mu > 30(\nu \,  \delta_0)^{1/d}.
\end{align}
This choice then implies that
\begin{align}\label{e:A1subA2}
    E \cap B(x,(1-\eta_0) r) \subset E\cap B(x,(1-\mu)r) \subset E \cap A^1 \subset E \cap A^2 \subset B(x,r) \cap E,
\end{align}
Now, by the hypothesis \eqref{e:small-content}, we see that for any $I \in \Delta_{j(\rho)}$ which is also contained in $A^2$ we have 
\begin{align*}
    \hdc\ps{ I \cap (E \cap A^1) } < \mu\, \nu\, \delta_0 \,r^d = \mu\, \rho^d.
\end{align*}
Adjusting the choice of $\mu$ and $\nu$ if needed, we see that this implies \eqref{e:theta-small} to hold for all $I \in \Delta_{j(\rho)}$ which also lie in $A^2$ with $F= E \cap A^1$. Moreover, with this $F$, \eqref{e:theta-small} holds trivially for any other $I \in \Delta_{j(\rho)}$. Hence we apply Proposition \ref{p:content} with $j= j(\rho)$ (i.e. so that $2^{-j} \approx \rho$), $A=A^2$ as defined in \eqref{e:A2} and $F = A^1 \cap E$, as defined in \eqref{e:A1}. We obtain a Lipschitz mapping $\phi$ which sends $E \cap A^1$ into $\dS_{j(\rho), d-1}$ and all the properties listed in the proposition. Note in particular that with the choice \eqref{e:mu-nu} of $\mu$ and $\nu$ and the fact that $\phi(E\cap I) \subset I$ for any $I \in \Delta_{j(\rho)}$, we have that
\begin{align} \label{e:containment1}
    B\ps{x, (1-\eta_0)r} \cap \phi(E) = B\ps{ x, (1-\eta_0)r} \cap \phi(E \cap A^1) \subset \phi(E \cap A^1).
\end{align}
\noindent Now we can extend $\phi$ to be the identity outside of $A^2$. Setting 
\begin{align*}
    \vp_t(y)= t\phi(y) + (1-t) y \mbox{ for } t \in [0,1],
\end{align*}
it is easy to check that $\vp_t$ satisfies \eqref{e:vp1}-\eqref{e:vp4}.
\end{proof}

\begin{proof}[Proof of Lemma \ref{l:TP-LCR}]
By definition, there exists a ball $B$ centered on $E$ satisfying $B \subset B(x,r)$, $\gamma_0 r < r_B$ and \eqref{e:TC}. It suffices to show that $\hd(B \cap E) \geq \wt{\cc_1} r(B)^d \geq \cc_1 r^d$, for some $\wt{\cc_1}$ (depending only on $\eta_0, \delta_0, d,n$) and with $\cc_1 = \wt{\cc_1} \gamma_0^d$.   
For the sake of contradiction, suppose that for the inequality \eqref{e:small-content} holds for $B$. Then, using the definition of topological condition \eqref{e:TC} (which can be applied since $r<r_0$), we obtain
\begin{align*}
    \delta_0 r^d &
    < \hd\ps{B(x_B,(1-\eta_0)r_B) \cap \vp_1(E)}\\
    & = \hd\ps{B(x_B,(1-\eta_0)r_B) \cap \phi(E)} \\
    & \leqt{\eqref{e:containment1}} \hd\ps{\phi(E \cap A^1)} = 0.
\end{align*}
Thus we must have that for any such a ball $B$, \eqref{e:small-content} cannot hold. This implies the lower content $d$-regularity of $E$ (for scales smaller than $r_0$), with constant $\cc_1$ depending only on $\delta_0,\eta_0$ and $\gamma_0$ (since recall that $\mu$ and $\nu$ in \eqref{e:small-content} only depend on $\delta_0$ and $\eta_0$). 
\end{proof}
\begin{remark}
Because all our statements are local, we will be ignoring the fact that our set is lower regular only for (possibly) small scales. In fact, we could assume without loss of generality that $r_0=1$. 
\end{remark}
\noindent
Let $E$ be stable $d$ surface as in the statement of Theorem \ref{t:topo-main} (or Proposition \ref{p:main-one-dir}). Since by Lemma \ref{l:TP-LCR} we know that $E$ is lower content $d$-regular (with constant $\cc_1$), we can apply Theorem \ref{t:AS}. In particular we have that for any cube $Q_0 \in \dD$ with $\ell(Q_0) < 1$, 
\begin{align*}
    {\rm BETA}_{E, 3, 2}(Q_0) \approx_{\aA, n, p, \cc_1, \Cc_0} {\rm BETA}_{E, \aA, p }(Q_0). 
\end{align*}
This was also shown in the Appendix of \cite{azzam2019quantitative}. Hence we have:
\begin{lemma}\label{l:redux-Ap}
It suffices to prove Theorem \ref{t:topo-main} \textup{(}and thus Proposition \ref{p:main-one-dir}\textup{)} with $p=2$ and $\aA=3$. 
\end{lemma}

\subsection{Construction of the approximating skeleta $E_R$} \label{s:ER}
In this subsection we recall the corona construction from \cite{azzam2019quantitative}. One should think of such a coronisation as a collection of `top cubes' $\{R\} \subset \dD$; so that each of these top cubes one can construct an Ahlfors regular set $E_R$ which approximates $R \subset E$.  
\begin{lemma}[\cite{azzam2019quantitative}, Main Lemma] \label{l:corona}
Let $k_0>0$, $\tau>0$, $d>0$ and $E \subset$ be a closed subset that is lower content $(d, \cc_1)$-regular. Let $\dD_k$ denote the Christ-David cubes on $E$ of scale $k$ and $\dD=\bigcup_{k\in\bZ} \dD_{k}$. Let $Q_{0}\in \dD_{0}$ and $\dD(Q_0, k_0)=\dD(k_0)=\bigcup_{k=0}^{k_0}\{Q\in \dD_{k}|Q\subseteq Q_0\}$. Then, for\footnote{$\Top(k_0)$ is a sub-collection of Christ-David cubes; what sort of Christ-David cubes belong to $\Top(k_0)$ will become clear in the sketch of the proof we give below. } $R \in \Top(k_{0})\subseteq \dD(k_{0})$, we may partition $\dD(k_{0})$ into stopping-time regions which we call $\Tree(R)$; this partition has the following properties:
\begin{enumerate}[leftmargin=0.8cm]
\item We have 
\begin{equation}
\label{e:ADR-packing}
\sum_{R \in \Top(k_{0})} \ell(R)^{d} \lec_{\cc_1,d} \dH^{d}(Q_0).
\end{equation}

\item Given $R\in \Top(k_{0})$ and a stopping-time region\footnote{Recall Definition \ref{d:stopping}.} $\Ss \subseteq \Tree(R)$ with maximal cube $T=T(\Ss)$, let  $\dF=\dF(S)$ denote the minimal cubes of $\Ss$ and set
\begin{align}\label{e:d_F}
    d_{\Ss, T, \dF}(x) = d_{T}(x) := \inf_{Q \in \dF} \ps{ \ell(Q) + \dist(x,Q)}.
\end{align}
For $C_{0}>4$ and $\tau>0$, there is a collection  $\cC(T(\Ss))=\cC_T \subset \Delta$ of disjoint dyadic cubes covering $C_{0}B_{T}\cap E$ so that 
if 
\[
E(T(\Ss))=E_T:=\bigcup_{I\in \cC_T} \d_{d} I,\]
where $\d_{d}I$ denotes the $d$-dimensional skeleton of $I$, then the following hold:
\begin{enumerate}[label=\textup{(}\alph*\textup{)}, leftmargin=0.8cm]
\item $E_T$ is Ahlfors $d$-regular with constants depending on $C_{0},\tau,d,$ and $\cc_1$.
\item We have the containment
\begin{equation}
\label{e:contains}
C_{0}B_{T}\cap E \subseteq \bigcup_{I\in \cC} I\subseteq 2C_{0}B_{T}.
\end{equation}

\item $E$ is close to $E_T$ in $C_{0}B_{T}$ in the sense that
\begin{equation}
\label{e:adr-corona}
\dist(x,E_T)\lec  \tau d_{T}(x) \;\; \mbox{ for all }x\in E\cap C_{0}B_{T}.
\end{equation}
\item The cubes in $\cC$ satisfy
\begin{equation}
\label{e:whitney-like}
\ell(I)\approx \tau \inf_{x\in I} d_{T}(x) \mbox{ for all }I\in \cC_T.
\end{equation}
\end{enumerate}
\end{enumerate}
\end{lemma}
\noindent
For readability purposes (and for setting some notation which will be used later on), we give a short sketch of the proof of Lemma \ref{l:corona}. The reader can find the details in Section 3 of \cite{azzam2019quantitative}. 
The idea is akin to Frostmann's Lemma and its proof (both can be found in \cite{mattila}, Chapter 8, from page 112). 
Assume without loss of generality that $Q_0 \subset [0,1]^n$. Let us fix some notation:
\begin{align*}
    & \Delta_{j}(Q_0) := \ck{I \in \Delta_{j} \, |\, Q_0 \cap I \neq \emptyset };\\
    & \Delta(Q_0) := \bigcup_{j \geq 0} \Delta_{j}(Q_0).
\end{align*}
We also set
\begin{align*}
    V_{j}(Q_0):= \bigcup_{I \in \Delta_{j}(Q_0) } I.
\end{align*}
We are going to iteratively define a measure on the approximating set $V_{j}$; we then put all those dyadic cubes where this measure growa too large in a family called $\Bad$, to then perform a stopping time algorithm on the David-Christ cubes of $Q_0\subset E$, stopping whenever a David-Christ cubes hits a dyadic cubes in $\Bad$ with comparable side length, and restarting after skipping one generation of cubes. Through this stopping time procedure we obtain the decomposition of $\dD(k_0)$ mentioned in Lemma \ref{l:corona}.
\\
\\
\noindent
Let us define the measure mentioned above. For $m \in \N$  
\begin{align*}
    \mu_{m}^m := \mathcal{H}^n|_{V_{m}} 2^{(n-d)m}.
\end{align*}
 Note that then, if $I \in \Delta_{m}(Q_0)$, 
\begin{align*}
    \mu_{m}^m (I) = \ell(I)^d.
\end{align*}
We now define a family of cubes $\Bad(m)$ as follows. First, we immediately impose that
\begin{align*}
    \Delta_{m}(Q_0) \subset \Bad(m).
\end{align*}
Next, we look at the cubes one level up, that is, at the cubes in $\Delta_{m-1}(Q_0)$. If for one such cube $J$, we have
\begin{align*}
    \mu_{m}^m(J) > 2 \ell(J)^d,
\end{align*}
the we put $J$ in $\Bad(m)$ and define
\begin{align*}
    \mu_{m}^{m-1}|_J := \ell(J)^d \frac{\mu_{m}^m|_J}{\mu_{m}^m(J)} < \frac{1}{2} \mu_{m}^m|_J.
\end{align*}
Otherwise, we set
\begin{align*}
    \mu_{m}^{m-1}|_J := \mu_{m}^m|_{J}.
\end{align*}
Note that in this way, it is always true that, for a cube $I \in \Delta_{m-1}(Q_0)$, $\mu_{m}^{m-1}(I) \leq 2 \ell(I)^d$. Continuing inductively in this fashion, we define $\mu_{m}^{m-2}$, $\mu_{m}^{m-3}$ and so on; suppose we defined $\mu_{m}^{k}$, for $1\leq k\leq m$. We consider the cubes $I \in \Delta_{k-1}$: if 
\begin{align*}
    \mu_{m}^{k}(I) > 2 \ell(I)^d, 
\end{align*}
then we put $I \in \Bad{}(m)$ and set 
\begin{align*}
    \mu_{m}^{k-1}|_I := \ell(I)^d \frac{\mu_{m}^{k}|_I}{\mu_{m}^{k}(I)} < \frac{1}{2}\mu_{m}^k.
\end{align*}
Otherwise, we set
\begin{align*}
    \mu_{m}^{k-1}|_I:= \mu_{m}^{k}|_I.
\end{align*}
We stop when we reach $k=1$ (and so $\mu_{m}^0$ is defined). 
One can then show the packing condition
\begin{align} \label{e:Bad-est}
    \sum_{I \in \Bad(m) } \ell(I)^d \lesssim_{n,d} \hd(Q_0),    
\end{align}
which is independent of $m \in \N$.
For a proof of this, see \cite{azzam2019quantitative}, in particular equation (3.5). 

Let now $k_0>0$ be an arbitrary integer number, $M>1$ a constant to be fixed later and $\aA>1$ the inflation constant for the $\beta_E^{d,p}$ coefficients (see Constant \eqref{i:Aa}). Here's the stopping time algorithm on the Christ-David cubes which will output the above mentioned partition of $\dD(Q_0)$.
For a Christ-David cube $R\in \dD(k_0)$ contained in $Q_0$, we let $\Stop(R)$ denote the set of maximal Christ-David cubes in $R$ from $\dD(k_0)$ that are either in $\dD_{k_{0}}$ or have a child $Q$ for which there is a dyadic cube $I\in \Bad(m)$ such that 
\begin{align} \label{e:ST-cond}
    & MB_Q \cap I \neq \emptyset \enskip \enskip \enskip \mbox{ and} \nonumber \\
    & \lambda \ell(I) \leq \ell(Q) \leq \ell(I),
\end{align}
where $\lambda$ is as in Theorem \ref{t:Christ}. Observe that if $R\in \dD_{k_{0}}$, then $\Stop(R)=\{R\}$. We then let $\Tree(R)$ be those Christ-David cubes contained in $R$ that are not properly contained in any cube from $\Stop(R)$, so in particular, $\Stop(R)\subseteq \Tree(R)$. Let $\Nextt(R)$ be the children of cubes in $\Stop(R)$ that are also in $\dD(k_{0})$ (so this could be empty).
Note that the process will eventually terminate because we stopped at all cubes, or because we reached the bottom of $\dD(k_0)$; in particular, the cardinality of $\Tree(R)$ is finite (perhaps depending on $k_0$). Furthermore, we consider all cubes $Q$ of the same generation of $Q_0$, so that
\begin{align*}
    2\aA\, Q_0 \cap Q \neq \emptyset,
\end{align*}
where $\aA$ is the constant appearing in \eqref{i:Aa}.
We denote this family by $\dN(Q_0)$. 
On each of these cubes, we perform the same stopping time, so to construct the relative $\Tree(Q)$.
\begin{figure}
    \centering
    \includegraphics[scale=0.7]{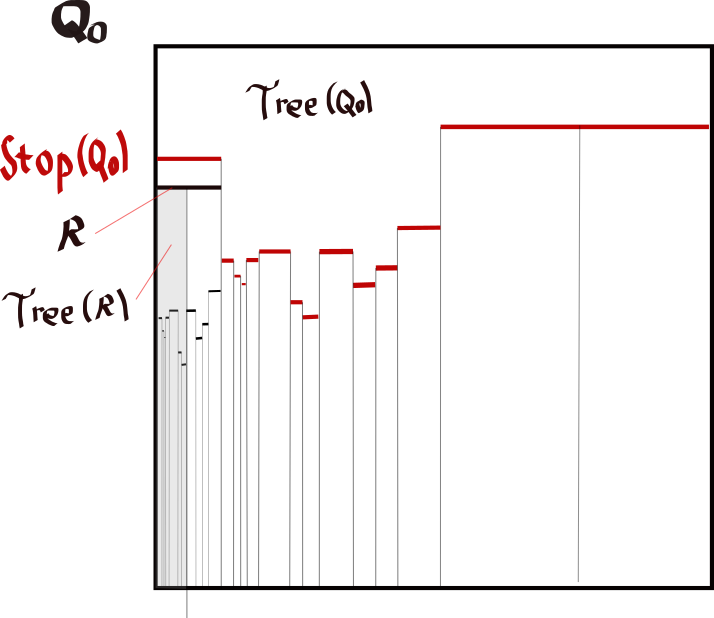}
    \caption{How one could imagine the partition structure. }
    \label{fig:stopping-time}
\end{figure}
Finally we put 
\begin{align} \label{e:forest}
    \Forest(Q_0) := \bigcup_{Q \in \dN (Q_0)} \Tree(Q),
\end{align}
and also
\begin{align*}
    \Stop(Q_0) := \ck{ Q \in \dD(k_0) \, |\, Q \mbox{ is minimal in } \Forest(Q_0)}.
\end{align*}
Next, we put
\begin{align*}
    \Nextt(Q_0) := \bigcup_{Q \in \Stop(Q_0)} \Child(Q).
\end{align*}
We now repeat the stopping time on each $R \in \Nextt(Q_0)$. Thus, if we set $\Top_0(k_0) :=\{Q_0\}$, then $\Top_1(k_0) := \Nextt(Q_0)$; proceeding inductively, supposed that $\Top_m(k_0)$ has been defined, for $m \in \N$: we put
\begin{align*}
    \Top_{m+1}(k_0) := \bigcup_{R \in \Top_m(k_0)} \Nextt(R).
\end{align*}
Finally, we set
\begin{align*}
    \Top(k_0) = \bigcup_{k\geq 0} \Top_k(k_0).
\end{align*}
Hence, to each element $R \in \Top(k_0)$, there correspond a forest $\Forest(R)$ and a family of minimal cubes $\Stop(R)$. Now, for each $R \in \Top(k_0)$ and for $x \in \R^N$, define
\begin{align*}
    & d_{R}(x) := \inf_{Q \in \Stop(R)} \ps{ \ell(Q) + \dist(x, Q)}, \mbox{ and } \\
    & d_R(I) := \inf_{x \in I} d_{R}(x), \mbox{ whenever } I \in \Delta.
\end{align*}
\begin{remark}\label{d_R(x)bounded}
Recall that $\Tree(R)$ has finite cardinality. This implies, in particular, that $d_R(x)$ and $d_R(I)$ are always bounded from below by the side length of the smallest Christ-David cube in $\Tree(R)$. 
\end{remark}
\noindent
These auxiliary functions are used to `smoothen out' the collection of dyadic cubes, so that if the lie close by, they will also have comparable side length. This is a trick that goes back to David and Semmes' \cite{david-semmes91}. 
For a parameter $\tau>0$, we put
\begin{align} \label{e:CR}
    \cC_R := \ck{ \mbox{ maximal } I \in \Delta \, |\, I\cap 2\aA R \neq \emptyset \mbox{ and } \ell(I) < \tau d_R(I)}. 
\end{align}
\begin{remark}\label{rem:CR-finite}
Because of Remark \ref{d_R(x)bounded}, the family $\cC_R$ will always have finite cardinality, too.
\end{remark}
\noindent
Finally we set
\begin{align}\label{e:wtER}
    \wt E_R := \bigcup_{I \in \cC_R} \partial_d I.
\end{align}
Thus $\wt E_R$ is the union of $d$-dimensional skeleta (see \eqref{e:skeleta}) of cubes belonging to $\cC_R$. 
\begin{lemma}[\cite{azzam2019quantitative}, Lemma 3.6]
The set $\wt E_R$ is Ahlfors $d$-regular with constant $\cc_0$.
\end{lemma}
\noindent
The constant $M>1$ is fixed here: it has to be sufficiently large (depending on $\tau$). See the proof of Lemma 3.6 in \cite{azzam2019quantitative}.
The following lemmas summarise some of the properties of the cubes in $\cC_R$. Their proof is standard, but we included them in the appendix. 
\begin{lemma}\label{l:whitey-cC-R}
The cubes $I \in \cC_R$ have disjoint interior and satisfy the following properties. 
\begin{enumerate}
    \item If $x \in 15\,I$, for some $I \in \cC_R$, then $\ell(I) \approx \tau d(x)$.
    \item There is a constant depending on $\tau$, such that if $15 I \cap 15J \neq \emptyset$ for $I, J \in \cC_R$, then $\ell(I) \approx_\tau \ell(J)$. 
\end{enumerate}
\end{lemma}
\begin{lemma} \label{l:meta1} Let $S$ be a cube in $\Stop(Q)$ for some $Q \in \Nextt(R)$, $R \in \Top(k_0)$. Then there exists a dyadic cube $I_S := I \in \cC_{Q}$ so that $I_S \subset \frac{1}{2}B_S$ and $\ell(I_S) \sim \tau \ell(S)$.
\end{lemma}

\begin{lemma}\label{l:I-QI}
Let $I \in \cC_Q$ for $Q \in \Nextt(R)$, $R \in \Top(k_0)$. Then there exists a cube $Q_I \in \Tree(Q)$ so that 
\begin{align*}
    & \ell(I) \leq \ell(Q_I) \leq c \tau^{-1} \ell(I); \\
    & \dist(I, Q_I) \leq c \tau^{-1} \ell(I). 
\end{align*}
\end{lemma}

\subsection{Modification of $\wt E_R$}
In this subsection we modify slightly the construction of $\wt E_R$; we need to do so to construct a coherent Federer-Fleming projection in the next section. 
\\
\\
\noindent
Fix $R \in \Top$; recall the definition of $\cC_R$ in \eqref{e:CR}.
 Take a cube $I \in \cC_R$. Consider one of its $(n-1)$-dimensional faces, and denote it by $T_I$. Set (see Figure \ref{fig:Adj})
\begin{align*}
    & \Adj^{n-1} (T_I) \\
    & := \ck{ J \in \cC_R \, |\, \ell(J) \leq \ell(I), \, J \cap T_I \mbox{ is an } (n-1) \mbox{-face of } J \mbox{ and } J \cap T_I \subset \Int(T_I) }.
\end{align*}
\noindent
\begin{figure}
    \centering
    \includegraphics[scale=0.4]{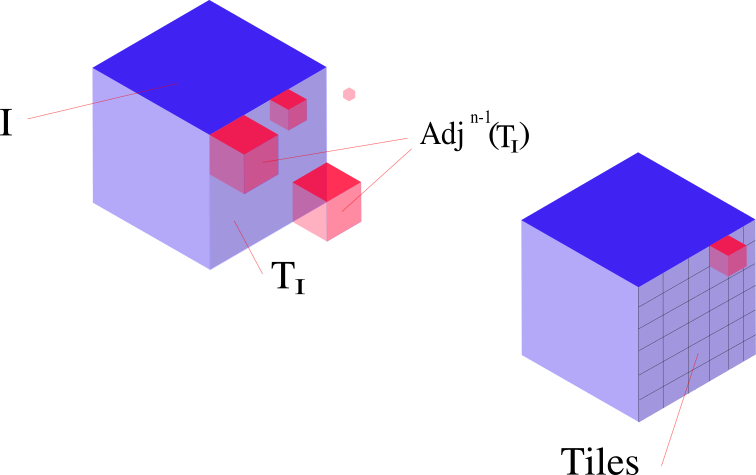}
    \caption{The red dyadic cubes constitute the family $\Adj^{n-1}.$}
    \label{fig:Adj}
\end{figure}
\noindent
We order the cubes in $\cC_R$ from the largest to the smallest one, and we label them as $I_0,....,I_N$, for some $N \in \N$. This is true because the cardinality of $\cC_R$ is finite (depending on $k_0$ - see Remarks \ref{d_R(x)bounded} and \ref{rem:CR-finite}). 
Let us start our construction with $I_0 \in \cC_R$ (thus $I_0$ is the largest cube in $\cC_R$). We look at one of its $(n-1)$-dimensional faces, let us denote it by $T_{I_0}$. 
Now, let $I$ be a cube of minimal side length contained in $\Adj^{n-1}(T_{I_0})$; let $n(I) \in \N$ be such that $\ell(I) = 2^{-n(I)}$. We consider the family of cubes in $\Delta_{n(I)}$ such that they have an $(n-1)$-dimensional face contained in $T_{I_0}$. We call this family  $\Delta_{n(I)}(T_{I_0})$. Let us denote by
\begin{align} \label{e:tiles}
    \cD^{n-1}(T_{I_0})
\end{align}
the family of $(n-1)$-dimensional faces of the same side length of $I$, such that they are both an $(n-1)$-dimensional face of a cube $J \in \Delta_{n(I)}(T_{I_0})$ and also they are contained in $T_{I_0}$. We may refer to this family as the \textit{tiles} of $T_{I_0}$. 
We repeat the same procedure for $I_1,...,I_N$; we don't do anything if $\Adj^{n-1}(T_{I_j})= \emptyset$ for some face $T_{I_j}$ of $I_j$, $1 \leq j \leq N$. Note that the definition of $\Adj^{n-1}(T_I)$ imposes the following: if two cubes $I$ and $I'$ are so that, say, $\ell(I)> \ell(I')$ and $I' \in \Adj^{n-1}(T_I)$, then the tiles constructed on $T_I$ will be the same one that we have on the face $T_{I'} \subset T_I$. The construction of tiles on the other $(n-1)$-faces of $I'$ will not change the ones already present in $T_{I'}$. 
This procedure terminates since $\cC_R$ is finite. 
\\
\\
\noindent
Once we constructed $(n-1)$-dimensional tiles on all the $(n-1)$-dimensional faces of all cubes in $\cC_R$, we rest. After, we proceed as follows. 
Denote by 
\begin{align} \label{e:cF}
    \cF^{n-1}
\end{align}
the family of $(n-1)$-dimensional faces belonging to some cube in $\cC_R$. If $T \in \cF^{n-1}$ and $\cD^{n-1}(T) \neq \emptyset$, the put the elements of $\cD^{n-1}(T)$ in $\cF^{n-1}$ and take $T$ away. If $\cD^{n-1}(T) = \emptyset$, then leave $T$ in $\cF^{n-1}$.
Next, we repeat the previous construction: order the elements of $\cF^{n-1}$ in decreasing order with respect to side length and consider $T_0$ (the largest face in $\cF^{n-1}$). For each $(n-2)$-dimensional face $F_{T_0}$ of $T_0$ we set
\begin{align*}
    & \Adj^{n-2}(F_{T_0})\\
    & := \ck{ T \in \cF^{n-1} \, |\, \ell(T) < \ell(T_0), \, T \cap F_{T_0} \mbox{ is an } (n-2)-\mbox{face of } T \mbox{ and } T \cap F_{T_0} \subset \Int (F_{T_0}) }. 
\end{align*}
We now look for the minimal element of $\Adj^{n-2}(F_{T_0})$, and call it $T$. Let $n(T) \in \Z$ so that $\ell(T) = 2^{n(T)}$; we now tessellate $F_{T_0}$ with tiles of side length $2^{n(T)}$; by tessellate here we mean the obvious thing, i.e. we substitute $F_{T_0}$ with its children of size $2^{n(T)}$. 
Let us denote the tiles so constructed by
\begin{align*}
    \cD^{n-2}(F_{T_0}). 
\end{align*}
We repeat the same procedure for $T_1,...,T_{N'} \in \cF^{n-1}$. Again, the construction of $(n-2)$-dimensional tiles for smaller $(n-2)$-dimensional faces does not affect the previously constructed tiles for larger faces. 
This procedure terminates since $\cF^{n-1}$ is finite, which follows trivially from $\cC_R$ being finite. 
Next, we set 
\begin{align*}
    \cF^{n-2} 
\end{align*}
to be the family of $(n-2)$-dimensional faces coming from elements of $\cF^{n-1}$, and we immediately modify it as above: if $\cD^{n-2}(F_T) \neq \emptyset $, for $T \in \cF^{n-1}$, we substitute $F_T$ with the corresponding family of tiles. 
\\
\\
\noindent
We continue this construction: we obtain $\cF^{n-3}$ from $\cF^{n-2}$, $\cF^{n-4}$ from $\cF^{n-3}$, and so on, until we construct $\cF^d$. We stop at this point and we set (see Figure \ref{fig:ER}). 
\begin{align} \label{e:newER}
   E_R := \wt E_R \cup \ps{\bigcup_{F \in \cF^d} F}.
\end{align}
\begin{lemma}
The set $E_R$ is Ahlfors $d$-regular. 
\end{lemma}
\begin{proof}
Lower regularity follows immediately from the definition and the lower regularity of $\wt E_R$. On the other hand, note that for any cube $I \in \cC_R$, any smaller neighbouring cube $I' \in \cC_R$ will satisfy $\ell(I')> \tau \ell(I)$ (using Lemma \ref{l:whitey-cC-R}, (2)). If we envelope $I$ in cubes of side length $\ell(I) \tau$ and we consider the $d$-dimensional skeleton of this family of cubes, we see that the overall additional mass will not exceed a constant times $\ell(I)^d$, where such a constant depends on $n$, $d$ and $\tau$. Thus upper regularity is also preserved.
\end{proof}
\begin{figure}[h]
    \centering
    \includegraphics{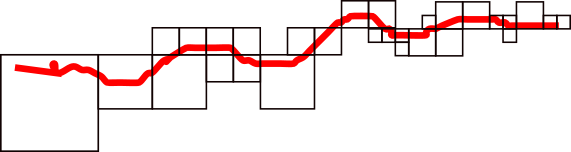}
    \caption{The approximating set $E_R$.}
    \label{fig:ER}
\end{figure}

\begin{notation} \label{n:C3}
From now on, we fix the notation for the regularity constant of $E_R$: it will be denoted by $\Cc_3$ and depends on $n,d,\tau$ anf the regularity constant of $\wt E_R$.
\end{notation}

\section{A topological condition on approximating skeleta}\label{s:STC}
We now introduce a condition on $E_R$ which will imply the existence of a uniformly rectifiable sets lying close to it. This is akin to the condition David calls TND (topological nondegeneracy condition) in \cite{david2004hausdorff}. We modified it somewhat to adhere to our trees structure. Let $R \in \Top(k_0)$ and $E_R$ be the set constructed in Section \ref{s:ER}, i.e. the set given in \eqref{e:newER}.

\begin{definition}[STC]\label{d:TND}
Let $C_2$ be an arbitrary big constant and let $k_0 \in \N$ be as in the statement of Lemma \ref{l:corona}. Let $E$ be a lower content ($\cc_1, d)$-regular set. Then we say that the family of subsets $\{E_R\}_{R \in \Top(k_0)}$ satisfies the \textit{skeletal topological condition} with parameter $C_2$, or $C_2$-(STC), if we can find positive constants 
\begin{align} \label{e:TND-constants}
   0< \alpha_1, \, \eta_1, \, \delta_1< 1 \mbox{ and } r_1>0
\end{align}
such that 
\begin{align} 
    & \mbox{for all } x_1 \in E,\label{e:TND-x1} \\
    & \mbox{for all } R \in \Top(k_0) \st x_1 \in R \mbox{ and } \frac{\ell(R)}{4} \leq r_1, \label{e:TND-R}\\
    & \mbox{for all } Q \in \Tree(R) \mbox{ s.t. } x_1 \in Q, \label{e:TND-Q}
\end{align}
for which 
\begin{align} \label{e:TND-upper}
    \hd\ps{E_R \cap B(x_1, \ell(Q))} \leq C_2 \ell(Q)^d
\end{align}
holds, 
there is a ball $B(x_2, r_2)$ centered on $E$ and contained in $B(x_1, \ell(Q))$ such that, for each one-parameter family $\{\vp_t\}_{0\leq t \leq 1}$ of Lipschitz functions on $\R^n$ that satisfy \eqref{e:vp1}, \eqref{e:vp2}, \eqref{e:vp3} and
\begin{align} \label{e:TND3}
    \dist(\vp_t(y), E) \leq \alpha_1 \ell(Q) \mbox{ for } t\in[0,1] \mbox{ and } y \in E_R \cap B(x_2,r_2),
\end{align}
we have that 
\begin{align}\label{e:TND4}
    \hd\ps{\vp_1(E_R \cap B(x_2,r_2))} \geq \delta_1 \ell(Q)^d + \hd(E_R \cap A_{\eta_1 \ell(Q)}(x_2,r_2)), 
\end{align}
where 
\begin{align} \label{e:TND5}
    A_{\eta_1 \ell(Q)}(x_2,r_2):= B(x_2,r_2) \setminus B(x_2, r_2-\eta_1 \ell(Q)).
\end{align}
\end{definition}
\begin{remark}
Note that
\begin{align}\label{e:r-R}
    r_2> \eta_1 \ell(Q);
\end{align}
if $r_2\leq \eta_1 \ell(Q)$, then $A_{\eta_1 \ell(Q)}= B(x_2, r_2)$.
Thus if we apply \eqref{e:TND4} with $\vp_t(y)=y$, then we would obtain
that $\hd(E_R \cap B(x_2, r_2)) > \hd(E_R\cap B(x_2, r_2))$, a contradiction.
\end{remark}

\section{Federer-Fleming projections}\label{s:FFproj}
In this section we will construct a Federer-Fleming projection of $E$ onto a subset of $E_R$; we will use these projections in the next section to prove that the topological condition \eqref{e:TC} on $E$ implies the condition STC on the approximating skeleta.  Our construction will mimic the one in \cite{david2004hausdorff}, which in turn comes from \cite{david2000uniform}. The difference here is that we are dealing with a skeleton of faces coming from cubes of different sizes.

\subsection{Preliminaries and notation}
\noindent
\subsubsection{The story so far.}
Let $E$ be the stable $d$-surface that we started with. From Lemma \ref{l:TP-LCR} we know that it is lower content ($d, \cc_1$)-regular. Hence we apply the coronisation as described in Subsection \ref{s:ER} (Lemma \ref{l:corona}), which is a collection of `top cubes' $\{R\}_{R \in \Top(k_0)}$. For each such an $R$ we construct an approximating set which we called $\wt E_R$ (see its definition in \eqref{e:wtER}, and see recall the definition of $\cC_R$ in \eqref{e:CR}). We subsequently modify it as in \eqref{e:newER}. 
\\
\\ 
\noindent
\subsubsection{Some notation}
Now fix any (Christ-David) `top cube' $R \subset E$ and let $B(x,r)$ be a ball centered on it (the construction below will be applied to the ball $B(x_2,r_2)$ as in the definition of STC, Definition \ref{d:TND}). Set
\begin{align}
    & \cC_R(x,r) := \left\{ I \in \cC_R \, |\, I \cap B(x,r) \neq \emptyset\right\}; \label{e:C_R(x,t)}\\
    & \cF^{m}(x,r) := \ck{ T \in \cF^{m} \, |\, T \cap B(x,r) \neq \emptyset } \mbox{ for } d \leq m \leq n-1; \label{e:F(x,r)} \\
    & D_R(x,r) := \bigcup_{I \in \cC_R(x,r)} I. \label{e:D-R}
\end{align}
We define a further family, which we call $\cC_R^2(x,r)$, as follows: 
\begin{itemize}[leftmargin=0.5cm]
    \item First we add to $\cC_R^2(x,r)$ all the dyadic cubes belonging to $\cC_R(x,r)$ (so in particular $\cC_R(x,r)$ will be a subfamily of $\cC^2_R(x,r)$. 
    \item Next, consider the collection (which we call) $\cF(x,r)$ of dyadic cubes $J \in \Delta \setminus \cC_R$  so that
\begin{align}
    & \Int(J) \subset \ps{ \bigcup_{I \in \cC_R(x,r)} I}^c; \label{e:C2a}\\
    & \mbox{ there exists a dyadic cube } I \in \cC_R(x,r) \st I \cap J \neq \emptyset; \label{e:C2b}\\
    & \min\ck{\ell(I) \, |\, I \in  \dN (J)} \leq \ell(J) \leq \max\ck{\ell(I) \, |\, I \in  \dN(J)}, \label{e:C2c}
\end{align}
where $\dN(J)$ is the family of cubes in $\cC_R(x,r)$ which intersects $J$. Then we add to $\cC_R^2(x,r)$ the subcollection of maximal dyadic cubes (see Definition \ref{d:maxfam}) of $\cF(x,r)$.
\end{itemize} 
\noindent
The family $\cC^2_R(x,r) \setminus \cC_R(x,r)$ forms a sheath for $\cC_R(x,r)$ (imagine the plastic covering of some Minecraft electrical wires). Finally we define
\begin{align} \label{e:D-R-2}
    D_R^2(x,r):= \bigcup_{I \in \cC_R^2(x,r)} I.
\end{align} 
\subsection{The construction}
Once again, recall the definition of $E_R$ as in \eqref{e:newER}. The following lemma is similar to Proposition 3.1 in \cite{david2000uniform}, and so is the proof. The only difference is that we are working with a non-uniform grid of cubes. 
\begin{lemma}\label{l:FF}
Let $E \subset \R^n$ be a lower content \textup{(}$d,\cc_1$\textup{)}-regular subset with $\hd(E)<+ \infty$.   Given $(x,r) \in E\times \R_+$, there exists a Lipschitz map $\pi : \R^n \to \R^n$ such that 
\begin{align}
    & \pi(y)=y \mbox{ whenever } x \in \R^n \setminus D_R^2(x,r); \label{e:pi1}\\
    & \pi(I) \subset I \mbox{ if } I \in \cC_R^2(x,r); \label{e:pi2}\\
    & \pi(E) \cap I \subset E_R \cap I \mbox{ for any } I \in \cC_R(x,r). \label{e:pi3}
\end{align}
\end{lemma}
\noindent
We will obtain our Federer-Fleming projection as the composition of a finite number of maps which we will define inductively. 
We start by defining a map, let us call it $\pi_1$, that will send points in $D_R^2(x,r) \cap E$ into $(n-1)$-dimensional faces. We define $\pi_1$ on each individual cube $I \in \cC_R^2(x,r)$ as follows. Pick a point $c_I \in I$ such that $c_I \notin E$. This is possible since $\hd(E)<\infty$ (recall \eqref{e:hdE}) and thus, in particular, $\dim_H(E) < d+1$; a standard argument then shows that $E$ is porous, and thus such a point $c_I$ must exist. Then for $y \in E \cap \Int(I)$, we set 
\begin{align}\label{e:pi1-a}
    \pi_1(y) \mbox{ to be the point where the line passing through } y  \mbox{ and } c_I \mbox{ meets } \partial I;
\end{align} 
note that, then, $\pi_1(y)$ belong to some $(n-1)$ dimensional face of $I$. On the other hand, if $y \in E\cap \partial I$, we set 
\begin{align} \label{e:pi1-b}
    \pi_1(y)=y.
\end{align}
We then 
\begin{align} \label{e:pi1-c}
    \mbox{extend } \pi_1 \mbox{ on the whole of } I  \mbox{ such that } \pi_1(I) \subset I \mbox{ and } \pi_1 \mbox{ is Lipschitz on } I.
\end{align} 
(This can be done via standard extension results, see for example \cite[Theorem 2.5]{heinonen2005lectures}).
Note that this definition is coherent, in the sense that one can glue together the definition of $\pi_1$ on each $I \in \cC_R^2(x,r)$  into a unique map $\pi_1$ defined on the whole of $D^2_R(x,r)$. Indeed, if $I, I' \in \cC_R^2(x,r)$ are so that $I \cap I' \neq \emptyset$, then the definition of $\pi_1$ on $I'\cap I$ must agree, since $I \cap I'$ is contained in $\partial I$ and $\partial I'$. Furthermore, we extend the definition of $\pi_1$ to $\R^n \setminus D_R^2(x,r)$ by setting \begin{align}\label{e:pi1-d}
    \pi_1(y)=y
\end{align}
there. Thus \eqref{e:pi1-a}-\eqref{e:pi1-d} give a coherent definition of $\pi_1$ on the whole of $\R^n$. 
\begin{figure}
        \centering
        \includegraphics[scale=.6]{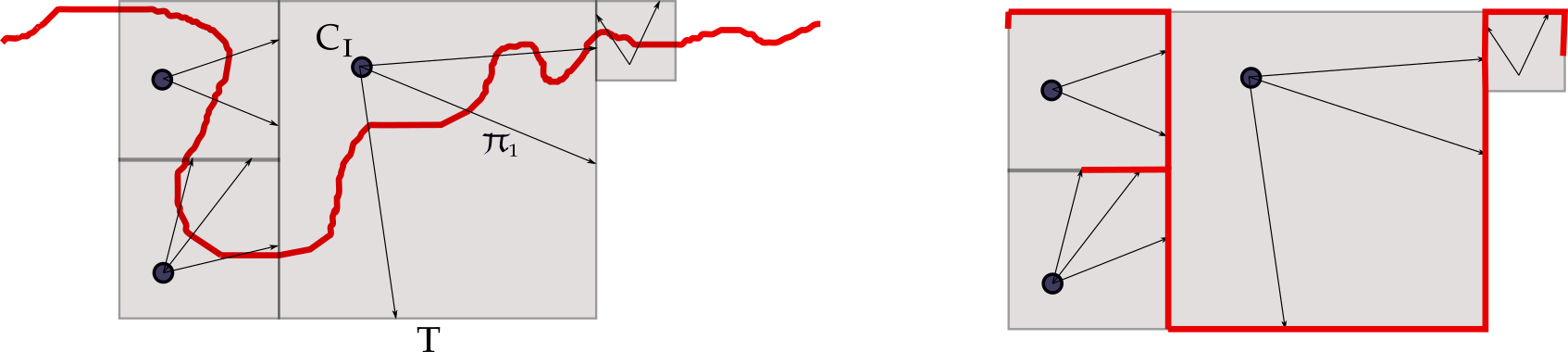}
        \caption{An image to give an heuristic idea of Federer-Fleming projections onto skeleta.}
    \end{figure}
\\
\\
\noindent
Now, if $d=n-1$, we stop here and we set $\vp := \pi_1$. Otherwise, we continue as follows. We want to send points on the $(n-1)$-dimensional faces of cubes in $\cC_R(x,r)$ to the boundaries of these faces, which are, in turn, $(n-2)$-dimensional faces. To do this, we proceed, as above, by defining the map we need on each individual face.
Recall the definition of $\cF^{n-1}$ in \eqref{e:cF}. Let us start by defining $\pi_2$ on each $\partial T \cup (\pi_1(E) \cap T)$, where $T \in \cF^{n-1}$: we repeat the construction above. Namely, 
\begin{align} \label{e:pi2a}
    \mbox{we find a point } c_T \in \Int(T) \setminus \pi_1(E) \mbox{ and then project radially } \pi_1(E)\cap T \mbox{ onto } \partial T;
\end{align} 
once again, this definition leave unchanged those points which already belong to $\partial T$. Moreover, we can extend $\pi_2$ as a Lipschitz map from $T$ to $T$ (for each $T \in \cF^{n-1}$), again by standard extension results. In this way, we obtain a coherently defined map on $\partial I$, for each $I \in \cC_R(x,r)$. Next, 
\begin{align} \label{e:pi2b}
    & \mbox{we extend } \pi_2 \mbox{ to the whole of } D_R(x,r)  \mbox{ by requiring that } \nonumber \\
    & \pi_2(I) \subset I \mbox{ for any } I \in \cC_R(x,r).
\end{align}
 To do so, we want to extend $\pi_2$ from $\partial I$, to the whole of $I$, with the requirement that $\pi_2(I) \subset I$. Let $c_I$ be the center of $I$. We set $\pi_2(c_I):= x^*$, where $x^*$ is any point in $\pi_2(\partial I)$. Then for any point $y \in \partial I$, and a point $x = t c_I + (1-t) y$, $t \in [0,1]$, (so that $x$ belongs to the line segment from $c_I$ to $y$), we set
\begin{align*}
    \pi_2 (x) = t \pi_2(c_I) + (1-t) \pi_2(y).
\end{align*}
Note that, because both $\pi_2(c_I)$ and $\pi_2(y)$ belong to $\partial I$, and $I$ is convex, then $\pi_2(x) \in I$. 
Let us check that $\pi_2$ so defined is Lipschitz on $I$. Take any two points $x_1, x_2 \in I$ and write them as
\begin{align}
    & x_1= t c_I + (1-t) y_1, \, \, t \in [0,1] \mbox{ and } y_1 \in \partial I;\label{e:x1}\\
    & x_2= s c_I + (1-s) y_2,\, \, s \in [0,1] \mbox{ and } y_2 \in \partial I. \label{e:x2}
\end{align}
Assume first that $t=s$. We can assume that $t=s<1$, for otherwise $x_1=x_2$. In this case, we have that
\begin{align*}
    |\pi_2(x_1) - \pi_2(x_2)|  = |(1-t)(\pi_2(y_1)-\pi_2(y_2))| & \leq |1-t|C |y_1-y_2| \\
    & = C \, |(1-t)(y_1-y_2)| \\
    & = |x_1-x_2|.
\end{align*}
Here the constant $C$ is the Lipschitz constant of $\pi_2$ as function defined on $\partial I$.
Next, let us suppose that for $x_1$ and $x_2$ as in \eqref{e:x1} and \eqref{e:x2}, we have that $y_1=y_2$, hence they lie on the same line segment from $c_I$ to $\partial I$. We first note that (assuming without loss of generality that $t>s$), 
\begin{align*}
|x_1 - x_2| = |(t-s) (c_I - y_1)| \geq (t-s) \ell(I).  
\end{align*}
On the other hand, simply beacuse both $\pi_2(c_I)$ and $\pi_2(y_1)$ belong to $\partial I$, we have that
\begin{align*}
    |\pi_2(x_1) - \pi_2(x_2)|= |(t-s)(\pi_2(c_I) - \pi_2(y_1))| \leq \sqrt{n}\, (t-s) \ell(I).
\end{align*}
Thus $|\pi_2(x_1) - \pi_2(x_2)| \leq \sqrt{n}|x_1-x_2|$. Finally, for  any two points $x_1, x_2 \in I$ as in \eqref{e:x1} and  \eqref{e:x2}, put
\begin{align*}
    x_2':= t\, c_I + (1-t)y_2.
\end{align*}
Note that there exists a constant, depending only on $n$, so that
\begin{align}
|x_2 - x_2'| \leq C \, |x_1 - x_2|,    
\end{align}
But then, by the triangle inequality, we also have that
\begin{align*}
    |x_1-x_2'| \leq (C+1) |x_1-x_2|. 
\end{align*}
This give us the following:
\begin{align*}
    |\pi_2(x_1) - \pi_2(x_2)| &  \leq |\pi_2(x_1) - \pi_2(x_2')| + |\pi_2(x_2') - \pi_2(x_2)| \\
    & \leq C \ps{ |x_1- x_2'| + |x_2'-x_2|} \\
    & \leq C' |x_1-x_2|. 
\end{align*}
This proves that the extension of $\pi_2$ to the whole of $I$ is indeed Lipschitz, with a Lipschitz constant comparable to that of $\pi_2$ as defined on $\partial I$. Now we let $\pi_2$ on $D(x,r)$ to be piecewise defined on each $I$ of $\cC_R(x,r)$. 
\\
\\
\noindent
Let us see why this definition is coherent. If $T, T' \in \cF^{n-1}$, $T \cap T' \neq \emptyset$ and let us assume without loss of generality that $\ell(T') < \ell(T)$, then either 
\begin{align}\label{e:pi2-case1}
    T' \subset T,
\end{align}
or
\begin{align} \label{e:pi2-case2}
    T \cap T' \subset \ps{\partial T} \cup \ps{\partial T'}.
\end{align}
If \eqref{e:pi2-case2} holds, than we immediately see that the definition of $\pi_2$ is coherent, since we defined to be the identity on both $\partial T$ and $\partial T'$. 
We divert a moment from the main construction to show that the former case does not happen.
\begin{lemma}
The case \eqref{e:pi2-case1} does not occur.
\end{lemma}
\begin{proof}
Let $T \in \cF^{n-1}(x,r)$, and assume first that $T$ is an $(n-1)$-dimensional face (as opposed to a tile) of a cube $I \in \cC_R(x,r)$. Suppose that there exists an element $T'$ of $\cF^{n-1}(x,r)$ such that $T' \subset T$. If $T'$ is an $(n-1)$-dimensional face of a cube $I' \in \cC_R(x,r)$, then, by construction of $\cF^{n-1}(x,r)$, we must have that
$I' \in \Adj^{n-1}(T)$. But then $F$ cannot possibly belong to $\cF^{n-1}$. On the other hand, if $T'$ is a tile, then also in this case $T$ cannot be in $\cF^{n-1}$, since it should have been tessellated into tiles of the same size of $T'$. 

Suppose now that $T$ is a tile itself. But by construction, we cannot have two tiles of different sizes lying on the same $(n-1)$-dimensional face. Thus $T' \subset T$ has to really be $T' = T$, which contradicts the fact that $\ell(T')< \ell(T)$.
\end{proof}
\noindent
Thus the definition of $\pi_2$ is coherent.
Let us now define $\pi_2$ on those $(n-1)$-dimensional faces $T'$  of cubes in $\cC_R^2(x,r)$ such that $\Int(T') \nsubseteq \Int(D_R(x,r))$ (recall the definition of $D_R(x,r)$, \eqref{e:D-R}). 
These are the faces which form the external boundary of the sheath $\cC_R^2(x,r) \setminus \cC_R(x,r)$. For these faces we leave everything unchanged, i.e. we let
\begin{align}\label{e:pi2c}
    \pi_2(y) = y  &\mbox{ for any } y \in T, \\
    & \mbox{ where } T \mbox{ is a } (n-1)\mbox{-dimensional face } T \mbox{ with } T \nsubseteq D_R(x,r).
\end{align}\label{e:pi2d}
Finally, we extend $\pi_2$ to the whole of $D^2_R(x,r) \setminus D_R(x,r)$ by requiring that 
\begin{align}
    & \pi_2(I) \subset I \mbox{ for } I \in \cC^2_R(x,r)\setminus \cC_R(x,r) \nonumber \\
    & \pi_2(y) = y \mbox{ whenever } y \in \partial D_R^2(x,r).
\end{align}
(This can be done in the same fashion as for \eqref{e:pi2b}).
We finally set 
\begin{align}\label{e:pi2e}
    \pi_2(y)=y \mbox{ whenever } y \in \R^n \setminus D_R^2(x,r).
\end{align}
Hence \eqref{e:pi2a}-\eqref{e:pi2e} give us a Lipschitz map $\pi_2$ defined on the whole of $\R^n$.
Now, if $d=n-2$, then we can set $\vp = \pi_2 \circ \pi_1$, otherwise we continue projecting. To do so, we define a third map $\pi_3$. We follow the procedure above: first, if $F$ is an $(n-2)$-dimensional element of $\cF^{n-2}(x,r)$, then we set $\pi_3$ to be the radial projection from some point $c_F \in \Int(F) \setminus \pi_2 \circ \pi_1(E)$ defined on $\partial F \cup \ps{\pi_1 \circ \pi_2 (E) \cap F}$. In particular, $\pi_3(y)=y$ if $y \in \partial F$. Next, we extend $\pi_3$ to the whole of $T$, by requiring that $\pi_3(T) \subset T$; if there is an element $F$ of $\cF^{n-2}$ such that $(\pi_2 \circ \pi_1) (E) \cap F = \emptyset$, we set $\pi_3(y)=y$ on such an element. Note that this definition is coherent by construction of $\cF^{n-2}(x,r)$, as in the definition of $\pi_2$. Next, we extend the definition of $\pi_3$ to the faces $T$ of dimension $(n-1)$, requiring that for any such a face, we have $\pi_3(T) \subset T$ and $\pi_3(\Int(T)) \subset \Int(T)$; we also require that $\pi_3(y) = y$ on those faces $T$ such that $T \cap D_R(x,r) = \emptyset$. Finally, we extend $\pi_3$ to the whole cubes $I$, requiring again that $\pi_3(I) \subset I$. At this point, note that for $y \in E$, we have 
\begin{itemize}
    \item  
    either $\pi_3 \circ \pi_2 \circ \pi_1 (y)\in \R^n \setminus D^2_R(x,r)$ if $ y \in E \setminus D^2_R(x,r)$;
    \item 
    or $\pi_3 \circ \pi_2 \circ \pi_1 (y) \in T$, where $T$ is a $(n-1)$-dimensional face of a cube in  $\cC_R^2(x,r) \st T \nsubseteq D_R(x,r)$; 
    \item
    or  $\pi_3 \circ \pi_2 \circ \pi_1 (y) \in F$,  where $F \in \cF^{n-3}$.
\end{itemize}
\begin{remark}
The second possibility only occurs for those $y \in E$ so that 
\begin{align*}
    y \in \bigcup_{I \in \cC^2(x,r) \setminus \cC(x,r)} I \subset \R^n \setminus B(x,r).
\end{align*}
\end{remark}
We continue constructing projections in this fashion until reaching the $d$-dimensional skeleton.
At each step, we construct $\pi_m$, for $n-d \leq m \leq n$, first on the elements of $\cF^m(x,r)$ as a radial projection, and second we extend this definition to faces (or tiles) of increasing dimension, asking (if $F'$ represents on such face or tile) that $\pi_m (F') \subset F'$. 
We stop once $\pi_{n-d}$ has been defined. If $y \in E $, then, setting
\begin{align} \label{e:def-pi}
    \pi:= \pi_{n-d} \circ \cdots \circ \pi_1,
\end{align}
we see that
\begin{itemize}
    \item either $\pi(y)\in \R^n \setminus D_R^2(x,r)$, if $y \in E \setminus D^2_R(x,r);$
    \item or $\pi(y) \in T$, where $T$ is an $(n-1)$-dimensional face of a cube in $\cC_R^2(x,r)$ such that  $T \nsubseteq D_R(x,r)$;
    \item or $\pi(y) \in F$, where $F \in \cF^{d}(x,r)$.
\end{itemize}
Note that the definition of $\pi$ is coherent for the same reasons that $\pi_3$ and $\pi_2$ were coherent. In particular, $\pi$ is Lipschitz (with possible a very large Lipschitz constant, but we do not mind this). Moreover, it follows from the construction that the properties \eqref{e:pi1}, \eqref{e:pi2} and \eqref{e:pi3} are satisfied; this concludes the proof of Lemma \ref{l:FF}
\section{Topological stability of $E$ implies topological stability of the approximating set} \label{s:TC-STC}
In this section, we will prove that if $E$ is a stable $d$-surface (i.e. it satisfies condition \eqref{e:TC}) and if we approximate it with sets $E_R$ (as in \eqref{e:newER}), then  these sets satisfy the skeletal topological condition STC (see Definition \ref{d:TND}).
\subsubsection{The story so far.}
Let $E$ be the stable $d$-surface that we started with. From Lemma \ref{l:TP-LCR} we know that it is lower content ($d, \cc_1$)-regular. Hence we apply the coronisation as described in Subsection \ref{s:ER} (Lemma \ref{l:corona}), which is a collection of `top cubes' $\{R\}_{R \in \Top(k_0)}$. For each such an $R$ we construct an approximating set which we called $\wt E_R$ (see its definition in \eqref{e:wtER}, and see recall the definition of $\cC_R$ in \eqref{e:CR}). We subsequently modify it as in \eqref{e:newER}. Next we defined coherent Federer-Fleming projections (Lemma \ref{l:FF}), through which one can map $E$ into $E_R$ in a Lipschtitz continuous way (without control on the Lipschitz constant). 
\\
\\
\noindent
The idea behind the proof of the Lemma \ref{l:TC-STC} below is that by composing the Federer-Fleming projection (which maps $E$ into $E_R$) with a deformation $\vp_t$ as in the definition of the skeletal topological condition (Def. \ref{d:TND}), we obtain an allowed Lipschitz deformation \textit{for} $E$ (see Def. \ref{d:ALD}). Consequently we obtain a lower bound for the $d$-measure of the deformation of $E$ (as in \eqref{e:TC}), and, because $E_R$ lies close to $E$, we obtain a lower bound for the measure of the deformation of $E_R$ (as in \eqref{e:TND4}). See Figure \ref{fig:tc-stc}.
\begin{figure}[h]
    \centering
    \includegraphics[scale=0.5]{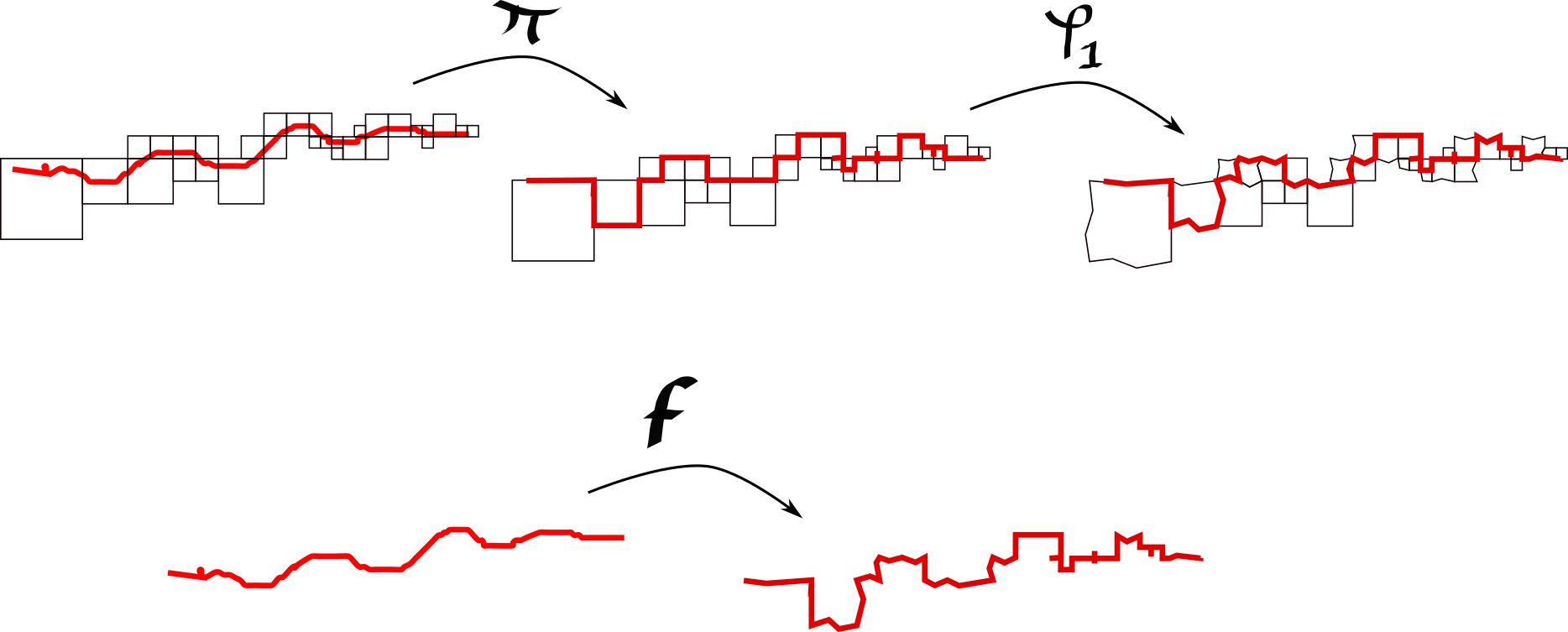}
    \caption{The topological condition on $E$ implies the skeletal topological condition on $E_R$: composing $\pi$ with $\vp_1$ gives an allowed deformation for $E$.}
    \label{fig:tc-stc}
\end{figure}
\noindent

\begin{lemma}\label{l:TC-STC}
Let $E \subset \R^n$ be such that $0<\hd(E)<\infty$. Suppose also that $E$ is a stable $d$-surface with constants $(r_0, \alpha_0, \delta_0, \eta_0, \gamma_0)$; let $Q_0 \in \dD(E)$ be such that $\ell(Q_0)< r_0$. For some $k_0 \in \N$, apply Lemma \ref{l:corona} to $Q_0$ to obtain a corona decomposition $\Top(k_0)=\Top(Q_0, k_0)$ and a family of sets $\{E_R\}_{R \in \Top(k_0)}$ with parameter $\tau$. Then we can find parameters $r_1, \alpha_1, \delta_1$ and $\eta_1$, so that the family $\{E_R\}_{R \in \Top(k_0)}$ satisfies the $C_2$-\textup{(STC)} for $C_2$ sufficiently large. 
\end{lemma}
\noindent
We will prove this lemma through a few lemmas below. 
Set 
\begin{align} \label{e:tau}
    \tau < \frac{1}{1000} \min\{\alpha_0, \eta_0\}.
\end{align}
\noindent
To ease the notation we let $\Top= \Top(k_0)$; recall that for a large constant $C_2$, we want to prove the existence of parameters $r_1>0$ and $0<\alpha_1,\delta_1<1$ (as in \eqref{e:TND-constants}) so that for all $x_1 \in E$, $R \in \Top$ and $Q \in \Tree(R)$ with $x_1 \in Q$, as in \eqref{e:TND-x1}-\eqref{e:TND-R}, for which \eqref{e:TND-upper} holds, we have the lower bound \eqref{e:TND4}. Let us immediately choose the parameters in \eqref{e:TND-constants} (our choice is somewhat similar to that of \cite{david2004hausdorff}, equation 3.10, except that we have one extra parameter to deal with, since our topological condition is weaker).
We set
\begin{align}
    & r_1 = r_0, \mbox{ where } r_0 \mbox{ is the one given by } \eqref{e:TC}; \label{e:r_0b} \\
    & \alpha_1 = C \min(\eta_0, \alpha_0); \label{e:alpha} \\
    & \eta_1 = C \frac{c(\eta_0)\delta_0\gamma_0}{C_2}, \mbox{ where } c(\eta_0):= \frac{\eta_0^2}{1+\eta_0}; \label{e:eta} \\
    & \delta_1 = C \delta_0. \label{e:delta}
\end{align}
We will fix the various absolute constants $C$ as we go along. They will only depend on $n$.
Let $x_1$, $R$ and $Q$ as in \eqref{e:TND-x1}, \eqref{e:TND-R} and \eqref{e:TND-Q}. We now want to find a ball $B(x_2, r_2)$ with the required properties: recall that the topological condition in Definition \ref{d:TC} gives us a ball $B=B(x_B,r_B)$ centered on $E$ (satisfying \eqref{e:r_B} and \eqref{e:Bin}) where the lower bounds \eqref{e:TC} holds. We choose
\begin{align}
    & x_2= x_B \mbox{ and } \label{e:TND-x}\\
    & r_2 \st \frac{r(B)}{1+ \eta_0} \leq r_2 \leq (1-\eta_0)r(B). \label{e:TND-r_2}
\end{align}
\begin{remark}\label{r:annulus}
We would also like the quantity $\hd\ps{E_R \cap A_{\eta_1 \ell(Q)}(x_2,r_2)}$ to be small. Indeed, if for some choice of $r_2$ it held that
\begin{align} \label{e:TP-TNDb}
    \hd\ps{E_R \cap A_{\eta_1 \ell(Q)} (x_2,r_2)} \lesssim \delta_1 \ell(Q)^d,
\end{align}
then in order to verify \eqref{e:TND4} (adjusting the constant in the definition of $\delta_1$), we would only have to check that 
\begin{align}\label{e:TP-TNDa}
\hd\ps{ \vp_1(E_R \cap B(x_2,r_2))} \geq \delta_1 \ell(Q)^d.     
\end{align}
\end{remark}
\begin{claim}\label{claim:annulus}
There exists a radius $r_2$ satisfying  \eqref{e:TND-r_2} and \eqref{e:TP-TNDb}.
\end{claim}
\begin{proof}
Choose $C$ in \eqref{e:eta} so that
\begin{align} \label{e:TND-C-eta}
    C\delta_0/C_2 = 10^{-1}.
\end{align} 
Then we consider ten radii $\{s_k\}_{k=0}^9$ given as 
\begin{align*}
    s_{k} = \frac{r(B)}{1+\eta_0} - \frac{k}{10} c(\eta_0) r(B), \quad 0\leq k\leq 9.
\end{align*}
Note that each $s_k$ satisfies \eqref{e:TND-r_2}. Since
\begin{align*}
    \eta_1 \ell(Q) \stackrel{\eqref{e:eta}}{=} C \frac{c(\eta_0)\gamma_0 \delta_0}{C_2} \ell(Q) \stackrel{\eqref{e:r_B}}{\leq}  \frac{C\delta_0}{C_2}c(\eta_0)r(B) \stackrel{\eqref{e:TND-C-eta}}{\leq} \frac{1}{10} c(\eta_0) r(B), 
\end{align*}
the annuli $\ck{A_{\eta_1 \ell(Q)}(x_2, s_k)}_{k=0}^9$ are pairwise disjoint.
Then we see that
\begin{align*}
    \sum_{k=0}^9 \hd(E_R \cap A_{\eta_1 \ell(Q)} (x_2,s_k)) \leq \hd\ps{E_R \cap B(x_2, r(B))} \leqt{\eqref{e:TND-upper}}  C_2 \ell(Q)^d. 
\end{align*}
Then by the pigeonhole principle and \eqref{e:eta} and \eqref{e:delta}, we must have that for some $0\leq k \leq 9$, 
\begin{align*}
    \hd(E_R \cap A_{\eta_1 \ell(Q)} (x_2, s_k)) \leq 10^{-1} C_2 \ell(Q)^d \stackrel{\eqref{e:TND-C-eta}}{\lesssim} C\delta_0 \ell(Q)^d \stackrel{\eqref{e:delta}}{\approx} 
    \delta_1 \ell(Q)^d. 
\end{align*}
Thus 
\eqref{e:TP-TNDb} holds putting $r_2=s_k$.
\end{proof}
\begin{lemma} \label{l:IandQ}
Let $I \in \cC_R(x_2, r_2)$, where $\cC_R(x,r)$ is defined in \eqref{e:C_R(x,t)}; also, here $x_2=x_1 \in Q$, for some $Q \in \Tree(R)$, and $r_2$ is as in \eqref{e:TND-r_2}. Then 
\begin{align}
    \ell(I) \leq C\, \tau\, \ell(Q). 
\end{align}
\end{lemma}
\begin{proof}
It suffices to prove the lemma for $Q \in \Stop(R)$, since if $Q \in \Tree(R) \setminus \Stop(R)$, then there exists a cube $Q' \in \Stop(R)$ with $x_2=x_1 \in Q'$ and $\ell(Q') < \ell(Q)$. 
If $I \in \cC_R$, then we know that $\ell(I) \lesssim \tau d_R(I)$ (recall that $\tau$ depends on $\alpha_0$ and $\eta_0$ and was fixed in \eqref{e:tau}). Moreover, $d_R(\cdot)$ is 1-Lipschitz. Then if $y \in I \cap B(x_2, r_2)$, we have that
\begin{align*}
    d_R(I) \leq d_R(y) \leq |y-x_2| + d_R(x_2).
\end{align*}
Because $r_2 \approx \ell(Q)$, $|x_2-y| \leq \ell(Q)$. On the other hand, we see that $$d_R(x_2) = \inf_{P \in \Stop(R)} \ps{ \ell(P) + \dist(x_2, P)} \leq \ell(Q).$$
\end{proof}
\begin{remark}
Note that the same holds for any $I \in \cC_R^2(x_2,r_2)$, by definition of $\cC_R^2(x_2,r_2)$ (as defined in \eqref{e:C2a}, \eqref{e:C2b} and \eqref{e:C2c}). 
\end{remark}

\begin{lemma} \label{l:TND}
Let  $(x_2,r_2)$ to be as chosen in \eqref{e:TND-x} and \eqref{e:TND-r_2}. For any one parameter family of Lipschitz deformations $\ck{\vp_t}_{0\leq t \leq 1}$ satisfying \eqref{e:vp1}, \eqref{e:vp2}, \eqref{e:vp3} and \eqref{e:TND3} \textup{(}relative to $(x_2,r_2)$\textup{)}, the property \eqref{e:TP-TNDa} holds, that is, we have
\begin{align*}
    \hd(\vp_1(E_R \cap B(x_2, r_2))) \geq \delta_1 \ell(Q)^d. 
\end{align*}
\end{lemma}
\begin{proof}
We have two ingredients we want to put together to achieve \eqref{e:TP-TNDa}: on one hand, we know that something similar holds for $E$ (i.e. TC); on the other hand, we know that $E$ is locally well approximated by $E_R$, and we have a continuous (actually Lipschitz) way to move from $E$ to $E_R$ (i.e. the Federer-Fleming projection we constructed in the previous section). The idea is therefore the following: pick the one parameter family $\vp_t$ for which we want to show \eqref{e:TP-TNDa}, and pick $\pi$ as in Lemma \ref{l:FF}. We will construct from these a deformation $f$ which satisfies conditions \eqref{e:vp1}-\eqref{e:vp4}; hence from \eqref{e:TC}, we will deduce \eqref{e:TP-TNDa}.
\\
\\
\noindent
Set 
\begin{align*}
    \pi_t (y) := t\, \pi(y) + (1-t)\, y.
\end{align*}
If $y \in D_R^2(x_2,r_2)$, let $I \in \cC_R^2(x_2, r_2)$ be the dyadic cube containing $y$ (recall the definition of $\cC_R^2(x_2, r_2)$ in \eqref{e:C2a}-\eqref{e:C2c}); then \eqref{e:pi2} tells us that
\begin{align*}
    |\pi_t(y) - y |\leq t|\pi(y)-y| \leq n^{1/2} \ell(I)
\end{align*}
if $\pi_t$ actually moves $y$.
Otherwise this quantity is equal to zero (this happens for example if $I \in \cC_R^2(x_2, r_2) \setminus \cC_R(x_2, r_2)$). By Lemma \ref{l:IandQ} we have that
\begin{align*}
    \ell(I) \lesssim \tau d_R(y) \leq \tau \ps{|y-x_2| + d_R(x_2)} \lesssim \tau \ell(Q).
\end{align*}
Hence we obtain
\begin{align} \label{e:TP-TNDg}
    |\pi(y) - y| \lesssim \tau \ell(Q) \mbox{ for all } y \in \R^n.
\end{align} 
Let us now define $\{f_t\}_{0\leq t \leq 1}$. We set
\begin{align}
    f_t (y) := \begin{cases}
    \pi_{2t}(y) & \mbox{ if } 0 \leq t \leq \frac{1}{2};\\
    \vp_{2t-1} (\pi(y)) & \mbox{ if } \frac{1}{2}\leq  t \leq 1.
\end{cases}
\end{align}
We claim that $\{f_t\}$ satisfies the conditions \eqref{e:vp1}-\eqref{e:vp4} applied to the larger ball \begin{align}\label{e:TND-wt-r}
    B(x_2, \wt r) \mbox{ where } \wt r:= (1+ \eta_0)r_2. 
\end{align}
See Figure \ref{fig:stable-top}.
We verify these conditions one by one. It is immediate from the definition that each $f_t$ is Lipschitz. 
\begin{figure}
    \centering
    \includegraphics[scale=0.5]{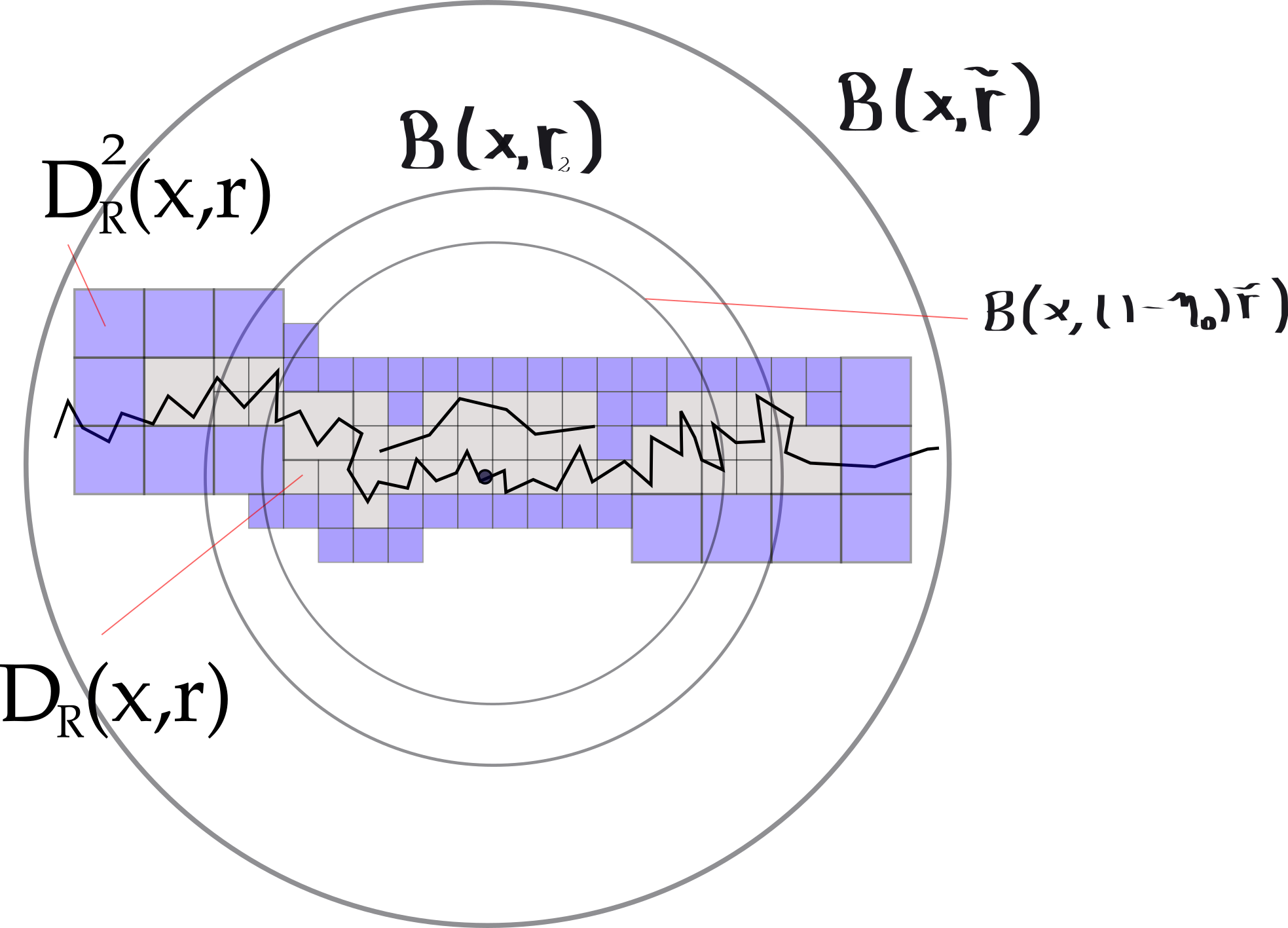}
    \caption{The set-up for the proof of Lemma \ref{l:TND}}.
    \label{fig:stable-top}
\end{figure}

\begin{claim} \label{claim:a}
We have that $f_t(\overline{B}(x,\wt r)) \subset \overline{B}(x, \wt r)$, i.e. \eqref{e:vp1} holds for $f_t$. 
\end{claim}
\begin{proof}[Proof of claim \ref{claim:a}]
Note that
\begin{align} \label{e:TP-TNDd}
    B(x_2,r_2) \cap E \subset D_R(x_2,r_2) \subset D_R^2(x_2,r_2),
\end{align}
where $D_R(x_2,r_2)$ and $D^2_R(x_2,r_2)$ were defined in \eqref{e:D-R} and \eqref{e:D-R-2};  \eqref{e:TP-TNDd} follows immediately from the definitions. Moreover, using Lemma \ref{l:IandQ}, we see that
any cube which was added to $\cC_R^2(x_2,r_2) \setminus \cC_R(x_2,r_2)$, must have side length at most $C \tau \ell(Q)$ (recall that $Q$ satisfies \eqref{e:TND-Q} and $x_2$ and $r_2$ are as in \eqref{e:TND-x} and \eqref{e:TND-r_2}, respectively). Thus $D_R^2(x_2,r_2) \subset B(x_2,r_2+ C \tau \ell(Q))$ and also, since
\begin{align*}
    \tau \ell(Q) \leqt{\eqref{e:tau}} \frac{1}{100} \eta_0 \ell(Q) \leqt{\eqref{e:TND-r_2}} \eta_0 r_2,
\end{align*}
we have that
\begin{align}
    D^2_R(x_2, r_2) \subset B(x_2, r_2+ \tau \ell(Q)) \subset   B(x_2, (1 + \eta_0) r_2) \subset B(x_2, \wt r).
\end{align}
\noindent
Let us consider a few cases separately.
\begin{itemize}[leftmargin=0.5cm]
    \item 
If $y \in D^2_R(x_2,r_2)$, then $\pi_t(y) \in D^2_R(x_2,r_2)$ for all $0\leq t \leq 1$ (by Lemma \ref{l:FF}), and so 
\begin{align*} 
    f_s(y) \in B(x_2, \wt r) \mbox{ for } s \leq \frac{1}{2} \mbox{ whenever } y \in D_R^2(x_2,r_2).
\end{align*} 
Now, recall that \eqref{e:vp1} holds for $\vp_t$ (relative to $B(x_2,r_2)$); hence 
if, together with $y \in D_R^2(x_2,r_2)$, we also have that $\pi(y) \in  \overline{B}(x_2,r_2)$, we then conclude that $f_s(y) \in \overline{B}(x_2,r_2)$ for $\frac{1}{2} \leq s \leq 1$. Also \eqref{e:vp3} holds for $\vp_t$: if $\pi(y) \notin \overline{B}(x_2,r_2)$, then, for $\frac{1}{2} \leq s \leq 1$ and recalling \eqref{e:TP-TNDd}, $f_s(y)= \pi(y) \in D_R^2(x_2,r_2)$. We obtain then that
\begin{align} \label{e:TP-TNDc}
    f_s(y) \in B(x_2, \wt r) \mbox{ whenever } y \in D_R^2(x_2,r_2).
\end{align}
\item Suppose now that $y \in B(x_2,\wt r) \setminus D_R^2(x_2,r_2)$; by construction $\pi(y)=y$ whenever $y \in \R^n \setminus D_R^2(x_2,r_2)$, hence $\pi_t(y) = y$ for $t \in [0,1]$, and thus
\begin{align} \label{e:TP-TNDe}
    f_s(y)=y \mbox{ for } s \leq \frac{1}{2} \mbox{ whenever } y \in B(x_2, \wt r) \setminus D_R^2(x_2,r_2).
\end{align} 
But because $\vp_t$ satisfies \eqref{e:vp3} relative to $B(x_2,r_2)$ (that is, $\vp_t$ is the identity outside $B(x_2,r_2)$, and becuase $B(x_2,r_2) \subset D^2_R(x_2,r_2)$ (i.e. \eqref{e:TP-TNDd}), then $\vp_t(\pi(y))=\pi(y)=y$. Thus we obtain that
\begin{align}\label{e:TP-TNDf}
    f_s(y) \in B(x_2, \wt r) \mbox{ for } \frac{1}{2} \leq s \leq 1 \mbox{ whenever  } y \in B(x_2, \wt r) \setminus D_R^2(x_2,r_2).
\end{align}
Now \eqref{e:TP-TNDc}, \eqref{e:TP-TNDe} and \eqref{e:TP-TNDf} give us the property \eqref{e:vp1} for $\{f_s\}$ relative to $B(x_2,\wt r)$. 
\end{itemize}
\end{proof}
\begin{claim}\label{claim:b}
 The path $s \mapsto f_s(y)$ is continuous, that is, \eqref{e:vp2} holds for $f_t$. Condition \eqref{e:vp3} also holds. 
\end{claim}
\begin{proof}[Proof of claim \ref{claim:b}]
The first conclusion is clear, since $t \mapsto \pi_t(y)$ is continuous; moreover, $\pi_1(y)=\pi(y) = \vp_0(\pi(y))= f_{\frac{1}{2}}(y)$, and $t \mapsto \vp_{2t-1}(\pi(y))$ is also continuous.
As for the second conclusion of the claim,
we see that $f_0(y)= \pi_0(y)=y$; if $y \in \R^n \setminus B(x_2,\wt r)$, we have seen above that $f_t(y)=y$. Thus \eqref{e:vp3} holds for $\{f_s\}$ in $B(x_2, \wt r)$. 
\end{proof}
\begin{claim}\label{claim:c}
Condition \eqref{e:vp4} holds, i.e. we have that $\dist(f_s(y), E) \leq \alpha_0 \wt r$ for $s \in [0,1]$ and $y \in E \cap B(x_2, \wt r)$.
\end{claim}
\begin{proof}[Proof of claim \ref{claim:c}]
First, consider $0 \leq s \leq \frac{1}{2}$; let $y\in E \cap B(x_2,\wt r)$; then we have
\begin{align*}
    & \dist(f_s(y), E) = \dist(\pi_{2s}(y), E) \leq |\pi_{2s}(y)- y|\\
    & \leqt{\eqref{e:TP-TNDg}} n^{\frac{1}{2}} \tau \ell(Q) \leqt{\eqref{e:tau}} \frac{1}{3} \alpha_0 \ell(Q) \leqt{\eqref{e:TND-r_2}, \eqref{e:TND-wt-r}} \alpha_0 \wt r.
\end{align*}
Now suppose that $s> \frac{1}{2}$. If $y \in B(x_2, \wt r) \cap E$, then, 
\begin{align}
    & \mbox{either } y \in D_R^2(x_2,r_2) \label{e:TP-TNDh} \\
    & \mbox{or } y \notin D^2_R(x_2,r_2). \label{e:TP-TNDi}
\end{align}
If $y$ is so that \eqref{e:TP-TNDi} holds, then $\pi(y) =y$, and moreover, from \eqref{e:vp3} for $\vp_t$ relative to $B(x_2,r_2)$, we see that $\vp_{2s-1}(\pi(y))= \vp_{2s-1}(y) = y$. Hence \eqref{e:vp4} holds in this case. 
On the other hand, suppose that \eqref{e:TP-TNDh} holds. Then from the proof of Lemma \ref{l:FF}, we have that
\begin{align}
    & \mbox{either } \pi(y) \in E_R \label{e:TP-TNDl}, \\
    & \mbox{or } \pi(y) \in T, \mbox{ where } T \mbox{ is an } (n-1)\mbox{-face with } T \nsubseteq D_R(x_2,r_2). \label{e:TP-TNDm}
\end{align}
If \eqref{e:TP-TNDl} holds, then $\dist(\vp_{2s-1}(\pi(y),E) \leq \alpha_1 \ell(Q)$ by \eqref{e:TND3} applied to $\vp_t$; this immediately implies $\dist(f_s(y), E) \leq \alpha_0 \wt r$ for $s> \frac{1}{2}$ by the choice of $\alpha_1$ in \eqref{e:alpha}. 
On the other hand, if \eqref{e:TP-TNDm} holds, we must have $\pi(y) \notin B(x_2,r_2)$ (by construction of $\pi$), and therefore $\vp_{2s-1}(\pi(y))= \pi(y)$. Now $\pi(y)$ belongs to a cube in $\cC_R^2(x_2,r_2)$ with side length at most $\tau \ell(Q)$ and touching $E$, hence we retrieve $\dist(f_s(y), E)\leq \alpha_0 \ell(Q) \leq \alpha_0 \wt r$. Together with the previous estimates, we obtain that $\{f_s\}$ satisfies \eqref{e:vp4} for all $s \in [0,1]$. This concludes the proof of claim \ref{claim:c}.
\end{proof}
\noindent
Claims \ref{claim:a}-\ref{claim:c} show that $\{f_t\}$ is an allowed Lipschtiz deformation relative to the ball $B(x_2, \wt r)$, with $\wt r = (1+\eta_0)r_2$. Since $B(x_2, \wt r_2) \subset B$, this in particular implies that $\{f_t\}$ is an allowed Lipschitz deformation for $B$ (the `stable' ball satisfying \eqref{e:Bcent}, \eqref{e:Bin}, \eqref{e:r_B}). Hence since we assumed that $E$ is a stable $d$-surface, we apply \eqref{e:TC}, and obtain
\begin{align}\label{e:TP-TNDo}
    \hd\ps{ B(x_B,(1-\eta_0)r(B)) \cap f_1(E) } \geq \delta_0 \ell(Q)^d.
\end{align}
But recall that $r_2 \geq (1-\eta_0)r(B)$ and thus $B(x_2, r_2) \supset B(x_2, (1-\eta_0) \wt r)$. Then,
\begin{align}\label{e:TP-TNDp}
    \hd\ps{B(x_2,r_2) \cap f_1(E)} \stackrel{\eqref{e:TP-TNDo}}{\geq}  \hd\ps{ B(x_B,(1-\eta_0)r(B)) \cap f_1(E) } \geq \delta_0 \ell(Q)^d.
\end{align}
Recall from above that if $y \in E$ and it is such that $y \notin D_R(x_2,r_2)$, then $f_1(y) \notin B(x_2,r_2)$. Thus, 
\begin{align*}
    B(x_2,r_2) \cap f_1(E) \subset f_1(D_R(x_2,r_2)\cap E).
\end{align*}
Note also that $\pi(E \cap D_R(x_2,r_2)) \subset E_R$. Thus we obtain, using \eqref{e:TP-TNDp}, \eqref{e:TP-TNDb} and \eqref{e:delta},
\begin{align*}
    \hd\ps{\vp_1(E_R) \cap B(x_2,r_2)} \geq 2 \delta_1 \ell(Q)^d. 
\end{align*}
This concludes the proof of Lemma \ref{l:TND}.
\end{proof}

\begin{proof}[Proof of Lemma \ref{l:TC-STC}]
Recall from 
Remark \ref{r:annulus}, that proving the estimates \eqref{e:TP-TNDb} and \eqref{e:TP-TNDa} is enough to prove \eqref{e:TND4}, and thus Lemma \ref{l:TC-STC}. Claim \ref{claim:annulus} gives the first estimate, while Lemma \ref{l:TND} gives the second one. 
\end{proof}

\section{A further approximating set}\label{s:E-rho-STC}
We now construct a dyadic approximation of $E_R$ (this is for technical reasons which will become transparent later on). We will also then show that this approximation satisfies the STC.
\\
\\
\noindent
Let $\rho$ be a small parameter (which we will fix later, and can be assumed to be of the form $2^{-k}, k \in \N$). We write
\begin{align*}
    \Delta_\rho := \Delta_{j(\rho)},
\end{align*}
where $j(\rho)$ is an integer so that $2^{-j(\rho)} = \rho.$
We set
\begin{align}
    & \cC_{R, \rho} := \ck{ I \in \Delta_{\rho} \, |\, I \cap E_R \neq \emptyset}; \label{e:cC-R-rho}\\
    & E_{\rho}= E_{R,\rho} := \bigcup_{I \in \cC_{R, \rho}} \partial_d I. \label{e:E-rho}
\end{align}

\begin{figure}[hbt!]
 \centering
 \includegraphics[scale=1.2]{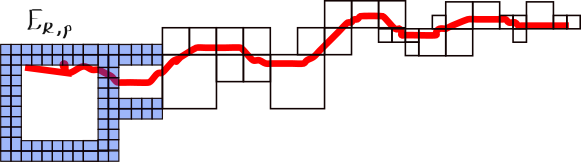}
 \caption{The red set represents $E$; the black squares represent $E_R$, while the blue smaller ones represent  $E_{R, \rho}$.}
\end{figure}

\begin{lemma} \label{l:Erho-ADR}
Let $I_*$ be the smallest cube in $\cC_R$ (which exists since $\cC_R$ is finite). Then for all $\rho< \ell(I_*)$, $ E_{R,\rho}$ is Ahlfors regular, with  regularity constant comparable to that of $E_R$. 
\end{lemma}
\begin{proof}
Let $T$ be a $d$-dimensional face of some cube $J \in \cC_R$. Denote by $\cF^{d}_\rho$ the collection of $d$-dimensional faces from cubes in $\cC_{R, \rho}$. Then we can cover $T$ with a subcollection of $\cF^{d}_\rho$ with pairwise disjoint interiors. If we denote such a collection by $\cF^{d}_\rho(T)$, then it is obvious that
\begin{align*}
    \hd(T) =\sum_{F \in \cF^d_{\rho}} \hd(F). 
\end{align*}
To each such a face $F \in \cF^d_{\rho}$, there corresponds a bounded number of cubes so that $F \subset I \in \cC_{R,\rho}$. This bounded number depends only on $n$ and $d$. Moreover, each of these cubes has a bounded number of other $d$-dimensional faces, and, again, this number depends only on $n$ and $d$. Thus, if we denote by $\Delta_\rho(T)$ the family of cubes in $\Delta_\rho$ which also meet $T$, we see that
\begin{align*}
   \sum_{I \in \Delta_\rho(T)} \hd(I) \leq C(n, d) \sum_{F \in \cF^{d}_\rho} \hd(F) = C(n,d)\, \hd(T).
\end{align*}
Then, we see that
\begin{align*}
    \hd\ps{ E_{R,\rho} \cap B(x,r)} & \leq \sum_{\substack{I \in \cC_{R,\rho}\\ I \cap B(x,r) \neq \emptyset}} \hd(\partial_d I) \\
    & \leq \sum_{J \in \cC_{R}(x,r)} \sum_{T \,\text{face of } J} \sum_{I \in \cF^d_\rho(T)} \hd(\partial_d I) \\
    & \leq C(n,d) \sum_{J \in \cC_{R}(x,r)} \hd(\partial_d J) \\
    & \lesssim \Cc_3 r^d.
\end{align*}
Recall that $\Cc_3$ is the upper regularity constant of $E_R$ (see Notation \ref{n:C3}). By enlarging it if necessary, we will assume that $\Cc_3$ is also the regularity constant of $E_{R, \rho}$.
\\
\\
\noindent It is immediate to conclude lower regularity for those points in $E_{R, \rho} \cap E_R$. Hence let $x \in E_{R, \rho} \setminus E_R$. If $r$ is comparable to $\rho$ or smaller, than clearly $\hd(B(x,r)\cap E_{R, \rho}) \gtrsim r^d$ since $E_{R, \rho}$ is just a $d$-dimensional affine plane at this scale. If $r \gtrsim \rho$, then, by construction, $ B(x,r)$ will contain a ball $B'$ with $r(B') \approx r$ and such that $B'$ is centered on $E_R$. We then use the lower regularity of $E_R$ to conclude that $E_{R, \rho}$ is lower $d$-regular, too.   
\end{proof}
\noindent 
\begin{remark}\label{rem:C3}
By making it even larger if needed, we assume that $\Cc_3$ is the Ahlfors $d$-regularity constant of $E_{R, \rho}$; it will not be useful to distinguish the Ahlfors regularity constants of $E_R$ and $E_{R,\rho}$; we will call them both $\Cc_3$. 
\end{remark}
\begin{lemma}\label{l:TC-STC-Erho}
Lemma \ref{l:TC-STC} holds when we substitute $E_R$ with $E_{R, \rho}$.
\end{lemma}
\noindent In other words, given a stable $d$-surface $E \subset \R^n$ with constants ($r_0, \alpha_0, \delta_0, \eta_0$) and with $0< \hd(E)<+\infty$, let $Q_0 \in \dD(E)$ be so that $r_0< \ell(Q_0)$. Then pick $k_0 \in \N$ be sufficiently large so that we can construct the partition (with parameter $\tau>0$ as in \eqref{e:tau}) $\Top(k_0)$ of $\dD(Q_0, k_0)$ as in Lemma \ref{l:corona}. For each $R \in E_{R, \rho}$ and for $\rho>0$ to be fixed (see \eqref{e:rho} below), we construct $E_{R, \rho}$ as described above. Then for $C_2$ sufficiently large, we can choose the parameters in the definition of STC (Definition \ref{d:TND}) as in\footnote{To be sure, $\eta_1$ will have to be chosen in the same way up to a multiplicative constant $c$, see the proof below. } \eqref{e:r_0b}-\eqref{e:delta} and hence find a ball with center $x_2$ as in \eqref{e:TND-x} and $r_2$ as in \eqref{e:TND-r_2} so that $\hd(E_{R,\rho} \cap B(x_2, r_2)) \geq \delta_1 \ell(Q)^d + \hd(E_{R, \rho} \cap A_{c\eta_1 \ell(Q)}(x_2,r_2))$, where $\vp_t$ is a deformation as in Definition \ref{d:TND}.
\begin{proof}
Let $C_2$ be the same constant as in Lemma \ref{l:TC-STC}. We know from Section \ref{s:TC-STC}, that $E_R$ satisfies the STC with a choice of constants as in \eqref{e:r_0b} - \eqref{e:delta}. Let $(x_2,r_2)$ as in \eqref{e:TND-x} and \eqref{e:TND-r_2}. 
We now add to the constraint on $\rho$ given in the statement of Lemma \ref{l:Erho-ADR}, the following one: we ask that
\begin{align} \label{e:rho}
    \rho < \frac{1}{1000\sqrt{n}} \min \ck{\eta_1, \alpha_1} \ell(Q) \mbox{ for all $Q \in \Tree(R)$}.
\end{align}
Note that because $E_{R,\rho}$ is Ahlfors regular independently of $\rho$, this does not cause any trouble. Moreover, we can always choose $\rho >0$ since $\Tree(Q)$ is a finite family (recall Remark \ref{d_R(x)bounded}). 
\\
\\
\noindent
Let $Q \in \Tree(R)$ and $(x_2, r_2)$ as in the proof\footnote{Recall that the skeletal topological condition (Definition \ref{d:TND}) requires us to consider points in $E$ (not in $E_R$ - this may lead to some confusion). If we then choose $x_2=x_1$ as in \eqref{e:TND-x}, we then do not need to distinguish between points in $E_R$ and $E_{R, \rho} \setminus E_R$, as it was done to show the lower Ahlfors regularity of $E_{R, \rho}$, for example.   } of Lemma \ref{l:TC-STC}, in particular see \eqref{e:TND-x} and \eqref{e:TND-r_2}.
Note that $E_{R,\rho} \cap B(x_2,r_2) \supset E_R\cap B(x_2,r_2)$, simply because $E_R \subset E_{R,\rho}$. Thus, also $\vp_1(E_{R,\rho} \cap B(x_2,r_2)) \supset \vp_1(E_R\cap B(x_2,r_2))$, and therefore
\begin{align*}
    \hd(\vp_1(E_{R,\rho} \cap B(x_2, r_2))) & \geq \hd\ps{\vp_1(E_R \cap B(x_2,r_2))} \\
    &\geq \delta_1 \ell(Q)^d + \hd(E_R \cap A_{\eta_1 \ell(Q)}(x_2,r_2)).
\end{align*}
But now note that if we choose a parameter $c>0$ sufficiently small, and we put $\wt \eta_1 = c \eta_1$, then we see that
\begin{align*}
    \hd(E_R \cap A_{\eta_1 \ell(Q)}(x_2, r_2)) \geq \hd(E_{R,\rho} \cap A_{\wt \eta_1 \ell(Q)}(x_2, r_2)). 
\end{align*}
Note that $c$ only depends on $n$ and $d$. Hence we obtain that 
\begin{align} \label{e:STC-E-rho}
    \hd(\vp_1(E_{R,\rho} \cap B(x_2, r_2))) \geq \delta_1 \ell(Q)^d + \hd(E_{R,\rho} \cap A_{\wt \eta_1 \ell(Q)}(x_2, r_2)),
\end{align}
and the lemma is proven. 
\end{proof}
\begin{remark}\label{rem:eta1-tilde}
To ease the notation, we ignore the fact that we changed $\eta_1$ by a constant $c$, and we will continue denoting $\wt \eta_1$ by $\eta_1$. 
\end{remark}

\section{A functional on skeleta}  \label{s:functional}
In this section, we show that the topological condition STC imposed on $E_R$ (and thus on $E_{R,\rho}$) tells us that $E_{R,\rho}$ has a large intersection (with repsect to the scale of each cube $Q \in \Tree(R)$) with a uniformly $d$-rectifiable set. The idea is to define a functional whose minimizer $F$ has large intersection with $E_{R,\rho}$. In turn $F$, by virtue of being a minimiser of such a functional, will turn out to be a quasiminimiser (in the sense of \cite{david2000uniform}), and thus uniformly rectifiable.

\subsubsection{The story so far.}
Let $E$ be the stable $d$-surface that we started with. From Lemma \ref{l:TP-LCR} we know that it is lower content ($d, \cc_1$)-regular. Hence we apply the coronisation as described in Subsection \ref{s:ER} (Lemma \ref{l:corona}), which is a collection of `top cubes' $\{R\}_{R \in \Top(k_0)}$. For each such an $R$ we construct an approximating set which we called $\wt E_R$ (see its definition in \eqref{e:wtER}, and see recall the definition of $\cC_R$ in \eqref{e:CR}). We subsequently modify it as in \eqref{e:newER}. Next we defined coherent Federer-Fleming projections (Lemma \ref{l:FF}), through which one can map $E$ into $E_R$ in a Lipschitz continuous way (without control on the Lipschitz constant). Via Federer-Fleming projections, we show that the fact that $E$ is a stable $d$-surface (with finite $d$-measure), implies that $E_R$, for each $R \in \Top(k_0)$, is also topologically stable, in the sense of the skeletal topological condition in Definition \ref{d:TND}; see Lemma \ref{l:TC-STC}. In Section \ref{s:E-rho-STC} we construct, for technical reasons, a further approximating set $E_{R, \rho}$ (for each $R \in \Top(k_0)$). 
Once again, the strategy below comes from \cite{david2004hausdorff}. We adapt it to our current situation.

\subsection{Definition of a functional $J$.}

\subsubsection{Preliminaries}

Let $C_2$ be a large constant (depending on $\Cc_3$, the Ahlfors regularity constant of $E_{R,\rho}$ and $E_R$) and $k_0$ a sufficiently large integer (as per Lemma \ref{l:corona}). Then STC gives us constants $r_1<r_0$, and $0<\alpha_1, \eta_1, \delta_1< 1$  (as in \eqref{e:TND-constants}) such that for every choice of $x_1 \in E$, $R \in \Top(k_0)$ and $Q \in \Tree(R)$ as in \eqref{e:TND-x1}, \eqref{e:TND-R} and \eqref{e:TND-Q} respectively, with also $\hd(E_R \cap B(x_1, \ell(Q))) \leq C_2 \ell(Q)^d$, we can find a ball $B(x_2,r_2) \subset B(x_1, \ell(Q))$ centered on $E$ for which, given an appropriate one-parameter family of Lipschitz deformations $\{\vp_t\}$, we have the lower bound \eqref{e:TND4}.
From the previous sections, we see that this holds for both $E_R$ and $E_{R, \rho}$.
\\
\\
\noindent
Fix a cube $Q \in \Tree(R)$ and let $(x_2,r_2)$ be as in \eqref{e:TND-x} and \eqref{e:TND-r_2}. To simplify the notation, we put
\begin{align}\label{e:x-r}
    & r:= r_2;\\
    & x := x_2.
\end{align}
For later use, let us set
\begin{align}\label{e:Bj}
    B_j:= B\ps{x, r - \eta_1 \ell(Q) + \frac{j \eta_1 \ell(Q)}{10}} \mbox{ for } 0 \leq j \leq 10.
\end{align}
Note that $B_{10}= B(x,r)$ and $B_{10}\setminus B_0 =A_{\eta_1 \ell(Q)} (x,r)$. 
\\
\\
\noindent
Recall the constraint on $\rho$, \eqref{e:rho}, and consider a constant $\sigma$ (let it be a power of $2$), so that
\begin{align} \label{e:sigma}
    \frac{1}{200\sqrt{n}} \min\ck{\eta_1, \alpha_1} \ell(Q) \leq \sigma \leq \frac{1}{100\sqrt{n}}  \min\ck{\eta_1,\alpha_1} \ell(Q).
\end{align}
Then we put
\begin{align}
    & \Delta_{\sigma}(B_6):= \ck{ I \in \Delta_{j(\sigma)} \, |\, I \cap B_6 \neq \emptyset};\label{e:delta6}\\
  & \cC_{\sigma}^1(B_6) := \ck{ I  \in \Delta_\sigma(B_6) \, |\, I \cap J \neq \emptyset \mbox{ for some } J \in \cC_R}; \label{e:cCB6}\\
  & \cC_{\sigma}^2(B_6) := \ck{ I \in \Delta_\sigma(B_6) \, |\, \mbox{ there exists a } J \in \cC_\sigma^1(B_6) \mbox{ with } I \cap J \neq \emptyset}.
\end{align}
Finally we put
    \begin{align}
        & V^1_Q := \bigcup_{I \in \cC^1_{\sigma}(B_6)} I  \enskip \enskip \mbox{ and}\label{e:V1}\\
        & V^2_Q := \bigcup_{I \in \cC^2_{\sigma}(B_6)} I. \label{e:V2}
    \end{align}
Note that for any cube $I \in \cC_R(B_6)$, there exists a cube $J \in \cC_{\sigma}(B_6)$ so that $J \supset I$. 
\begin{lemma}
With the notation above, we have that
\begin{align}
    & B_6\cap E \subset V_Q^2 \subset B_7 \mbox{ and } \label{e:Ta}\\
    & \dist(y, E) \leq \alpha_1 \ell(Q) \mbox{ for all } y \in V^1_Q.  \label{e:Tb}
\end{align}
\end{lemma}
\begin{proof}
The first inclusion in \eqref{e:Ta} is immediate\footnote{Indeed, if $y \in E \cap B_6$, then, first there exists a cube $J \in \cC_R$ which contains $y$. But then $J \cap B_6 \neq \varnothing$. Then, because $\Delta_\sigma(B_6)$ covers $B_6$, we find a cube $I \in \Delta_\sigma(B_6)$ which intersects $J$ and $y \in J \cap I$. Clearly $J \in \cC_\sigma^1(B_6)$, and thus $y \in V_Q^2$.}. To see the second one, note that for any point $y \in V_Q^2$, we have that $\dist(y, V^1_Q) \leq \sqrt{n} \sigma$. In turn, any point in $V^1_Q \setminus B_6$ can be at most $\sqrt{n}\sigma$ away from $\partial B_6$. Hence, by the choice of $\sigma$ in \eqref{e:sigma}, we see that if $y \in V_Q^2\setminus B_6$, $\dist(y, \partial B_6) < \frac{1}{10} \eta_1 \ell(Q)$, and so \eqref{e:Ta} will be satisfied. 
\\
\\
\noindent
To prove \eqref{e:Tb}, suppose first that $y \in V^1_Q$ is such that $y \in I$ for some $I \in \cC_R$. Now, by construction, if $y \in V^1_Q$, then $y\in B(x_2, r_2)$, with $x_2 \in Q$ and $r_2 \approx \ell(Q)$. Thus $\dist(I,Q)\lesssim \ell(Q)$. This implies, using Lemma \ref{l:whitey-cC-R} (1), that
\begin{align*}
    \ell(I) \leq \tau d_R(I) \leq \tau (\ell(Q) + \dist(I,Q)) \lesssim \tau \ell(Q).
\end{align*}
But recall also that $\tau \ll \alpha_0, \eta_0$ (as in \eqref{e:tau}), and thus also $\tau \ll \alpha_1$, by the choice of $\alpha_1$ in \eqref{e:alpha}. Since $I \cap E \neq \emptyset$, for otherwise $I$ would not be in $\cC_R$, we obtain that $\dist(y, E) \leq \alpha_1 \ell(Q)$ for these $y$. Note that if $y \in V^1_Q$ but $y \notin I$ for all $I \in \cC_R$, then by the way that $V^1_Q$ was defined, $\dist(y, I) < \sigma$ for some $I \in \cC_R$. Keeping in mind the definition of $\sigma$ in \eqref{e:sigma}, we can repeat the argument above to obtain \eqref{e:Tb} also in this case.
\end{proof}

\subsubsection{The domain of the functional}
\noindent
We are now ready to fix the class of subsets upon which the said functional will be allowed to act.
We set
\begin{align}
    \dF_0 \label{e:F_0}
\end{align}
to be the class of subsets $F$ of $\R^n$ such that
\begin{align}
    & F \mbox{ is closed (in the topology of $\R^n$).} \label{e:F_01}\\
    & F \subset V_Q^2 \label{e:F_02}\\
    & F = F^* \cup L. \label{e:F_03}
\end{align}
Here $L$ denotes any subset of Hausdorff dimension smaller or equal than $d-1$; by $F^*$ we mean a finite union of $d$-dimensional faces of cubes coming from $\Delta_{\rho}$. We will call $F^*$ the \textit{coral} part of $F$. 
In other words the class $\dF_0$ is composed by subsets that are built out of a finite number of $d$-faces coming from cubes in $\Delta_\rho$. 
Let us consider a subclass of $\dF_0$: we set
\begin{align} \label{e:dF}
    \dF := \ck{ F \in \dF_0 \, |\,  F= \vp_1(E_{R,\rho} \cap V_Q^2)},
\end{align}
where $\{\vp_t\}_{0 \leq t \leq 1}$ is a family of Lipschitz mappings on $\R^n$ such that
\begin{align}
    & \vp_t(V_Q^2) \subset V_Q^2 \mbox{ for all } t \in [0,1];  \label{e:vpV2a} \\
    & t \mapsto \vp_t(y) \mbox{ is continuous for all } y \in \R^n; \label{e:vpV2b} \\
    & \vp_t(y) = y \mbox{ for } t=0 \mbox{ and for } y \in \R^n \setminus V_Q^2; \label{e:vpV2c} \\
    & \dist(\vp_t(y), E_R) < \alpha_1 \ell(Q) \mbox{ for } y \in E_{R,\rho} \cap V_Q^2 \mbox{ and all } t \in [0,1]; \label{e:vpV2d}\\
    & \vp_1(y) \in V^1_Q \mbox{ for } y \in E_{R,\rho} \cap V_Q^2.\label{e:vpV2e}
\end{align}
\begin{lemma}\label{l:dF}
We have that $E_{R,\rho} \cap V_Q^2 \in \dF$. In particular, the class $\dF$ is nonempty. 
\end{lemma}
\begin{proof}
We just take the trivial deformation $\vp_t(y)=y$ for all $y$ and $t$, so that \eqref{e:vpV2a}, \eqref{e:vpV2b} and \eqref{e:vpV2c} hold immediately. Moreover, by construction we have that all points in $E_{R,\rho}$ are contained in a cube from $\cC_{R, \rho}$. The side length of these cubes is (much) less than\footnote{Since one such a cube will have side length $\rho$ and $\rho \ll \min\{\eta_1, \alpha_1\}\ell(Q)$ from \eqref{e:sigma}.} $\alpha_1 \ell(Q)$ and they must touch $E_R$. Hence $\dist(y, E_R) < \alpha_1 \ell(Q)$ and so \eqref{e:vpV2d} is satisfied. As far as condition \eqref{e:vpV2e} is concerned, we see that if $y \in E_{R,\rho} \cap V_Q^2$, then by definition of $E_{R,\rho}$ and $\cC_{R, \rho}$ in \eqref{e:E-rho} and \eqref{e:cC-R-rho}, we see that $y$ must lies in a cube which belongs to $\cC_\sigma^1(B_6)$ (from the definition of $\sigma$ in \eqref{e:sigma}), and thus it must be in $V^1_Q$.
\end{proof}
\noindent
We now define the aforementioned functional. For some $c_2<1$ to be chosen later, we put
\begin{align}\label{e:M}
M=\frac{\Cc_3}{c_2\,\delta_1},    
\end{align}
where recall that $\Cc_3$ is the Ahlfors regularity constant of $E_R$ (as fixed in Notation \ref{n:C3}) and of $E_{R, \rho}$.
Then we set
\begin{align}\label{e:J}
    J(F) := \hd(F \cap E_{R,\rho}) + M\, \hd(F \setminus E_{R,\rho}) \mbox{ for } F \in \dF. 
\end{align}
Note that $J(F)= J(F^*)$ (with notation as in \eqref{e:F_03}), and there is only a finite number of sets like $F^*$. Thus there exists a set $\wt F \in \dF$ such that 
\begin{align*}
    J(\wt F) = \min_{F \in \dF} J(F). 
\end{align*}

\subsubsection{A minimiser of $J$ has large intersection with $E_{R, \rho}$}
Note that, for a set $F \in \dF$ trying to keep $J(F)$ small, it will be very expensive to have a large portion which does not intersect $E_{R,\rho}$, as $M$ can be quite large. This is the reason why we expect the minimiser $\wt F$ to have a large intersection with $E_{R,\rho}$. This also implies that a minimiser of $J$ also will lie close to $E_R$.
\begin{lemma} \label{l:large-intersection}
Let $\wt F$ be a minimiser of $J$ \textup{(}as in \eqref{e:J}\textup{)} in $\dF$. Then 
\begin{align} \label{e:L-Int}
    \hd(E_{R,\rho} \cap \wt F) \geq C \,\delta_1\ell(Q)^d.
\end{align}
\end{lemma}
\begin{figure}
    \centering
    \includegraphics[scale=0.7]{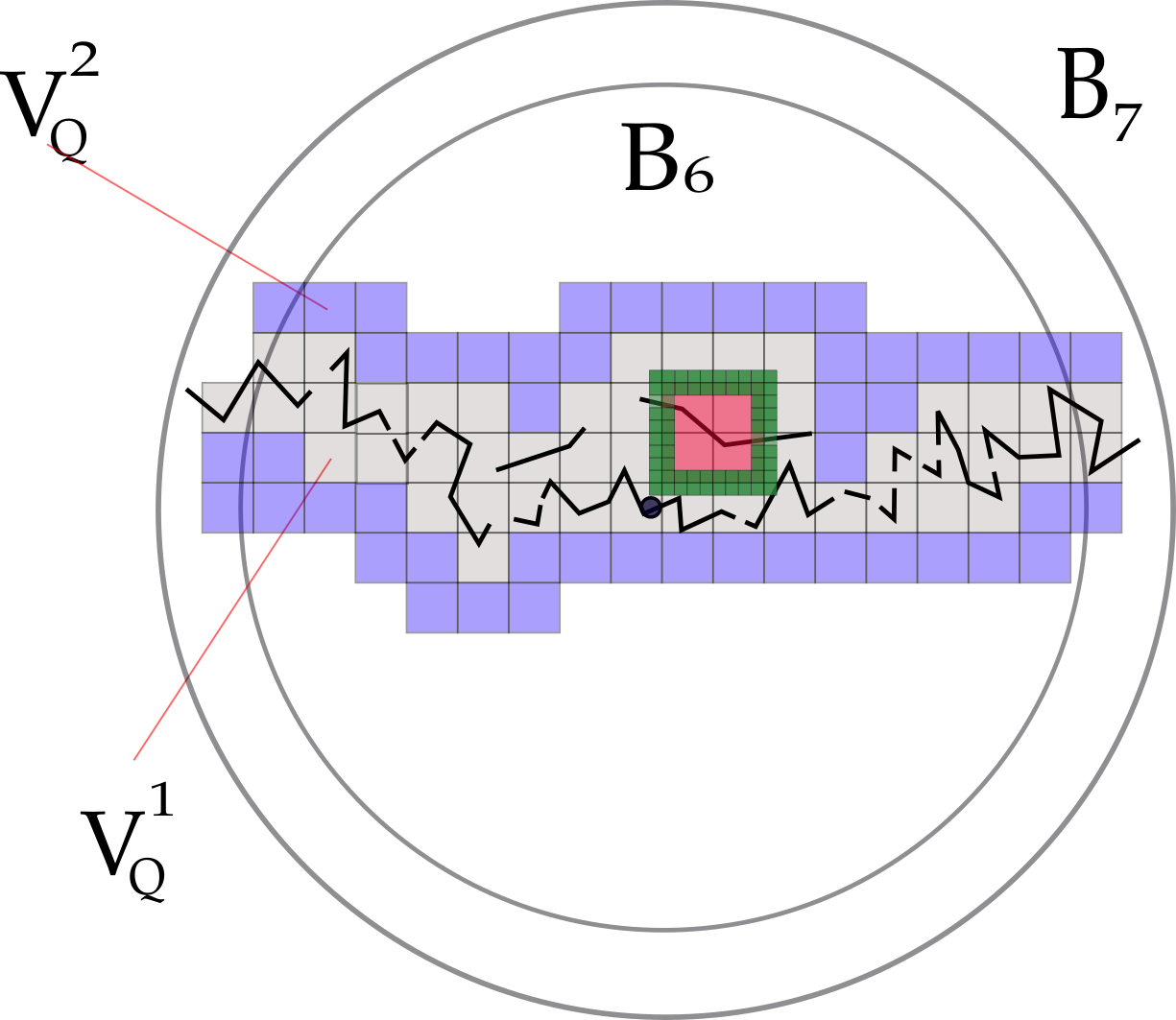}
    \caption{The set up for the proof of Lemma \ref{l:large-intersection}: the grey and blue dyadic square are the ones with side length $\sigma$. The red dyadic square is from $\cC_R$ (i.e. its skeleton forms part of $E_R$ - one should imagine red squares covering the black set; they were not drawn for simplicity). The small green squares are those whose skeleta form $E_{R, \rho}$.}
    \label{f:large-intersection}
\end{figure}
\begin{proof}
Because $\wt F \in \dF$, then $\wt F = \vp_1(E_{R,\rho} \cap V_Q^2)$, where $\{\vp_t\}$ satisfies \eqref{e:vpV2a}-\eqref{e:vpV2e}. We want to check that this specific one parameter family $\vp_t$ satisfies also the conditions for the deformations used for STC (see Definition \ref{d:TND}) relative to $B(x,r)=B(x_2,r_2)$. Note that  $\{\vp_t\}$ satisfies \eqref{e:vp1}, \eqref{e:vp2} and \eqref{e:vp3}, since from \eqref{e:Ta}, we have that $V_Q^2 \subset B(x,r)$. We want to check that  \eqref{e:TND3} hold, that is
\begin{align} \label{e:TND3b}
    \dist(\vp_t(y), E_R) \leq \frac{\alpha_1 \ell(Q)}{2} \mbox{ for } t \in [0,1] \mbox{ and } y \in E_{R,\rho} \cap B(x,r).
\end{align}
So, let $y \in E_{R,\rho} \cap B(x,r)$. If $y \notin V_Q^2$, then $\vp_t(y)=y $ by \eqref{e:vpV2c}; since $y \in E_{R, \rho}$, then by construction $\dist(y, E_R) < \sqrt{n} \rho \ll \alpha_1 \ell(Q)$ (by the constraint on $\rho$ in \eqref{e:rho}). Moreover, if $x \in E_R$ is the point closest to $y$, then $x \in \partial_d I$ for some $I \in \cC_R(x,r)$ and, by Lemma \ref{l:IandQ}, we have that $\ell(I)\lesssim \tau \ell(Q)$. All in all, this implies that $\dist(y, E)< \alpha_1 \ell(Q)$ whenever $y \in B(x,r)\cap E_{R,\rho}$ does not belong to   $V_Q^2$. 
\\
\\
\noindent
 If $y \in V_Q^2 \cap E_{R,\rho}$, then by \eqref{e:vpV2e}, $\vp_t(y)$ must lie in $V^1_Q$, and hence be at most $\sigma$ far away from $E_R$ (by construction); but $\sigma < \alpha_1 \ell(Q)$, and repeating the argument of the previous case we again see that  \eqref{e:TND3b} holds.
\\
\\
\noindent
Thus $\{\vp_t\}$ is a one parameter family of Lipschitz deformations which satisfies the requirements stated in the definition of STC (see Definition \ref{d:TND}). Recall that $E_{R,\rho}$  satisfies the STC (at scale $\ell(Q)$); we therefore have the lower bound \eqref{e:STC-E-rho}, i.e.
\begin{align*}
    \hd(\vp_1(E_{R,\rho} \cap B(x,r))) \geq \delta_1 \ell(Q)^d + \hd\ps{E_{R,\rho} \cap A_{\wt \eta_1 \ell(Q)}(x,r)}.
\end{align*}
Now, the family $\{\vp_t\}$ which we are considering, not only satisfies \eqref{e:vp3}, but also \eqref{e:vpV2c}, and so, in particular, 
\begin{align*}
    \vp_1\ps{E_{R,\rho} \cap B(x,r)} & = \vp_1\ps{(V_Q^2 \cap E_{R,\rho}) \cup (E_{R,\rho} \cap (B(x,r) \setminus V_Q^2))}\\
    & = \vp_1(E_{R,\rho} \cap V_Q^2) \cup \ps{E_{R,\rho} \cap (B(x,r) \setminus V_Q^2)},
\end{align*}
recalling that by definition (see \eqref{e:Ta}) $V_Q^2 \subset B(x,r)$. Also, note that 
\begin{align} \label{e:Ja}
    E_{R,\rho} \cap (B(x,r) \setminus V_Q^2) \subset A_{\wt\eta_1 \ell(Q)}(x,r) \cap E_{R,\rho};
\end{align}
indeed, using  $B_6 \cap E_{R,\rho} \subset V_Q^2$, we see that $E_{R,\rho} \cap (B(x,r) \setminus V_Q^2) \subset E_{R,\rho} \cap (B(x,r) \setminus B_6)$, and (recalling the definition of the balls $B_j$'s \eqref{e:Bj}), $B(x,r) \setminus B_6 \subset A_{\wt \eta_1\ell(Q)}(x,r)$. Thus we have
\begin{align*}
     \hd\ps{\vp_1(E_{R,\rho} \cap V_Q^2)} + \hd\ps{A_{\wt \eta_1 \ell(Q)}(x,r)\cap E_{R,\rho}} \geq \delta_1 \ell(Q)^d + \hd\ps{E_{R,\rho} \cap A_{\wt \eta_1 \ell(Q)}(x,r)}
\end{align*}
and so
\begin{align}\label{e:Jb}
    \hd\ps{\vp_1(E_{R,\rho} \cap V_Q^2)} \geq \delta_1\ell(Q)^d. 
\end{align}
In particular, from the definition of $\dF$, this inequality holds for any $F \in \dF$. 
Recall now that we decided that $\wt F$ was a minimiser of $J$ (as defined in \eqref{e:J}). Thus we have that
\begin{align}\label{e:Jc}
    & J(\wt F) \leqt{\text{Lemma } \ref{l:dF} } J(E_{R,\rho} \cap V_Q^2) = \hd\ps{E_{R,\rho} \cap V_Q^2} \notag\\
    & \leqt{\eqref{e:Ta}} \hd\ps{E_{R,\rho} \cap B(x,r)} \leqt{\text{Lemma } \ref{l:Erho-ADR}} \Cc_3 \ell(Q)^d. 
\end{align}
Moreover, by definition of $J$, 
\begin{align} \label{e:Jd}
    \hd\ps{\wt F \setminus E_{R,\rho}} \leq \frac{J(\wt F)}{M} \leqt{\eqref{e:Jc}} \frac{C_3}{M}\ell(Q)^d \eqt{\eqref{e:M}} c_2\,\delta_1 \ell(Q)^d.
\end{align}
But then we have that, with $c_2$ sifficiently small, 
\begin{align}
    \hd\ps{\wt F \cap E_{R,\rho}} = \hd(\wt F) - \hd(\wt F \setminus E_{R,\rho}) \geqt{\eqref{e:Jb}, \eqref{e:Jd}} \frac{\delta_1 \ell(Q)^d}{2}.
\end{align}
This proves the Lemma.
\end{proof}


\section{Almgren quasiminimality of $\wt F$.} \label{s:quasiminimality} Roughly speaking, a set $S$ in $\R^n$   is a (Almgren\footnote{The definition below was introduced by Almgren in \cite{almgren1976existence}.}) quasiminimiser of the $d$-dimensional Hausdorff measure $\hd$ if, whenever we deform $S$ in a suitable way, the $d$-measure of such deformations does not shrink too much. Quasiminimality is a form of stability: the set maintain its Hausdorff dimension under a suitable class of perturbations. Heuristically, this is the reason why we need to transfer the topological condition from $E$ to an Ahlfors regular set: in this case, modulo technicalities, quasimininality roughly coincides with our topological condition.

\subsubsection{The story so far.}
Let $E$ be the stable $d$-surface that we started with. From Lemma \ref{l:TP-LCR} we know that it is lower content ($d, \cc_1$)-regular. Hence we apply the coronisation as described in Subsection \ref{s:ER} (Lemma \ref{l:corona}), which is a collection of `top cubes' $\{R\}_{R \in \Top(k_0)}$. For each such an $R$ we construct an approximating set which we called $\wt E_R$ (see its definition in \eqref{e:wtER}, and see recall the definition of $\cC_R$ in \eqref{e:CR}). We subsequently modify it as in \eqref{e:newER}. Next we defined coherent Federer-Fleming projections (Lemma \ref{l:FF}), through which one can map $E$ into $E_R$ in a Lipschitz continuous way (without control on the Lipschitz constant). Via Federer-Fleming projections, we show that the fact that $E$ is a stable $d$-surface (with finite $d$-measure), implies that $E_R$, for each $R \in \Top(k_0)$, is also topologically stable, in the sense of the skeletal topological condition in Definition \ref{d:TND}; see Lemma \ref{l:TC-STC}. In Section \ref{s:E-rho-STC} we construct, for technical reasons, a further approximating set $E_{R, \rho}$ (for each $R \in \Top(k_0)$) which also satisfies the STC (Lemma \ref{l:TC-STC-Erho}). Then, it turns out that the minimiser $\wt F$ of the functional $J$, see \eqref{e:J}, has large intersection with $E_{R, \rho}$ (Lemma \ref{l:large-intersection}).

\subsection{Quasiminimizer sets for Hausdorff measure}
\noindent
We now recall from \cite{david2000uniform} the precise definitions to make this notion precise. Let $U$ be an open set in $\R^n$ and fix two constants 
\begin{align}\label{e:QMconst}
    1 \leq k < \infty \mbox{ and } 0 < \delta \leq + \infty. 
\end{align}
Let $S \subset U$ be so that
\begin{align} \label{e:QMa}
    S \neq \emptyset \mbox{ and } \overline{S} \setminus S \subset \R^n \setminus U.
\end{align}
Assume also that 
\begin{align} \label{e:QM-finite}
    \hd(S \cap B) < +\infty \mbox{ for all balls } B \subset U.     
\end{align}
Now, let us make precise what we mean by `deformations' or `perturbations'. Given a set $S$, deformations of $S$ will be sets of the form $\phi(S)$, where 
\begin{align}\label{e:QM-Lip}
    \phi: \R^n \to \R^n \mbox{ is Lipschitz }
\end{align}
and satisfies the following properties. 
\begin{align}
    & \diam\ps{ W \cup \phi(W)} \leq \delta \mbox{ where } W := \ck{ x \in \R^n \, |\, \phi(x) \neq x}; \label{e:QM-W} \\
    & \dist\ps{W \cup \phi(W) , \R^n \setminus U} > 0; \label{e:QM-dist}\\
    & \phi \mbox{ is Lipschitz-homotopic to the identity}. \label{e:QM-homo}
\end{align}
The last requirement means that there exists a continuous map 
\begin{align*}
    h : \R^n \times [0,1] \to \R^n
\end{align*}
such that $h(x, 0) = x$ and $h(x, 1) = \phi(x)$ for all $x \in \R^n$, such that $h(\cdot, t): \R^n \to \R^n$ is Lipschitz for all $t \in [0,1]$, and such that
\begin{align*}
    \diam(\widehat W) < \delta \mbox{ and } \dist( \widehat W, \R^n \setminus U ) >0,
\end{align*}
where 
\begin{align*}
    \widehat W := \bigcup_{t \in [0,1]} W_t \cup \phi_t(W_t), \enskip \phi_t(x) = h(x,t) \mbox{ and } W_t:= \ck{x \in \R^n \, |\, \phi_t(x) \neq x }.
\end{align*}

\begin{definition}
Let $0<d< n$; let $U \subset \R^n$ be an open set and fix two constant $k, \delta$ as in \eqref{e:QMconst}. We say that $S \subset U$ is a \textit{$(U, k, \delta)$-quasiminimizer} for $\hd$ if $S$ satisfies \eqref{e:QMa}, \eqref{e:QM-finite} and 
\begin{align}
    & \hd(S \cap W) \leq k \hd\ps{\phi(S \cap W)}\label{e:QMmain} 
\end{align}
for all Lipschitz mappings  $\phi$  which satisfy  \eqref{e:QM-W}, \eqref{e:QM-dist}  and  \eqref{e:QM-homo}.
\end{definition}

\subsection{The set $\wt F$ is a quasiminimizer}
\noindent This subsection is devoted to prove Lemma \ref{l:quasiS} below. Recall that $\wt F$ is a minimiser of the functional $J$ (as in \eqref{e:J}). Recall also the notation 
\begin{align*}
    B_j:= B\ps{x, r - \eta_1 \ell(Q) + \frac{j \eta_1 \ell(Q)}{10}} \mbox{ for } 0 \leq j \leq 10,
\end{align*}
where $x,r$ are as in \eqref{e:x-r}, $\eta_1$ as in \eqref{e:eta} (bearing in mind Remark \ref{rem:eta1-tilde}) and $Q \in \Tree(R)$ as in \eqref{e:TND-Q}.

\begin{lemma}\label{l:quasiS}
The set 
\begin{align}
    S:= \wt F \cap B_2
\end{align}
is a $(B_2, k, \delta)$-quasiminimizer for $\hd$, where
\begin{align}
    k = C\,4^{nd}M, 
\end{align}
(here $C$ is a geometric constant), and
\begin{align}
    \delta = c_3 \min\ck{\alpha_1, \eta_1} \ell(Q).
\end{align}
Here $0< c_3<1$ is a parameter bounded above by a universal constant.
\end{lemma}
\noindent
We will need the following lemma from \cite{david2004hausdorff}. We tailor it to our current notation.  
\begin{lemma}[{\cite{david2004hausdorff}, Lemma 5.8}]\label{l:retraction}
Let $V^1$ be a finite union of dyadic cubes belonging to $\Delta_\sigma$, where $\sigma$ should be thought of as in \eqref{e:sigma}. There exists a $4^n$-Lipschitz function $h$, defined on 
\begin{align} \label{e:h1}
    V^1_+ := \ck{ y \in \R^n \,|\, \dist(y, V^1) \leq \frac{\sigma}{4}}
\end{align}
and such that
\begin{align}
    & h(V^1_+) \subset V^1 \label{e:h2}\\
    & h(y) = y \mbox{ for } y \in V^1, \label{e:h3}
\end{align}
and 
\begin{align}
    |h(y) - y| \leq n^{1/2} \sigma.
\end{align}
\end{lemma}
\noindent
Recall that we want to show that $\wt F$ is a quasiminimal set for $\hd$.
Here is the idea to do so. We want to look at $\hd(\phi(\wt F \cap W))$; what we know about $\wt F$ which makes us hope that it may well be a quasiminimal set is that $\wt F$ is a minimiser of the functional $J$ as defined in \eqref{e:J}. We want to use this information. In other words, we would like to say that $\phi(\wt F)$ is a competitor of $\wt F$ belonging to the class $\dF$. Unfortunately, this is not true, in the sense that $\phi(\wt F)$ may lie outside $V_Q^1$, and this is not permitted (see \eqref{e:vpV2e}). What we can do however, is first to retract $\phi(\wt F)$ (which we will call $F_1$) back into $V_Q^1$ (using the map $h$ from Lemma \ref{l:retraction}); let us set $F_2:=h(F_1)$. Next, we want to project $F_2$ onto some $d$-dimensional skeleton so that it belongs to $\dF_0$ (as defined in \eqref{e:F_0}). This projection will happen in two steps, with two corresponding maps; we will denote the images so obtained by $F_3$ and then $F_4$; this latter one will be the needed competitor. The last step will be to show that these distortions of $\phi(\wt F)$ don't increase the size of $\phi(\wt F)$ too much. In this way, first by the minimising property of $\wt F$ we will obtain a bound like $\hd(\wt F) \leq M \hd(F_4) $ and then, by this last step, a bound similar to $\hd(F_4) \leq C \hd(\phi(\wt F))$ and thus establishing quasiminimality.

\subsubsection{Construction of $F_1$ and $F_2$}
Let us get started: we want to deform $\wt F$ with Lipshitz maps $\phi$ as in \eqref{e:QM-Lip}. Pick one such Lipshitz deformation $\phi$. We are interested in those points $y \in W \cap \wt F$, i.e. those points which are actually being moved by $\phi$. But by \eqref{e:QM-W}, we must have that $|\phi(y)-y| \leq \delta$. We put
\begin{align} \label{e:delta-choice}
    \delta= c_3 \min\ck{\alpha_1, \eta_1} \ell(Q), 
\end{align}
where $c_3$ is a small parameter to be chosen soon. The rationale to choose $c_3$ is that we want $ \phi(\wt F)$ to lie in $V^1_+$, so that we may apply Lemma \ref{l:retraction} and send $\phi(\wt F)$ back into $V_Q^1$. Recall that $V^1_+$ is the set of points lying at most $\sigma/4$ far away from $V_Q^1$; recall also that $\wt F= \vp_1(E_\rho \cap V^2) \subset V_Q^1$ by \eqref{e:dF} (the way $\dF$ was defined) and the property of $\{\vp_t\}$, \eqref{e:vpV2e}. Hence for an appropriate choice of $c_3$, say 
\begin{align}\label{e:c_3}
    c_3 = \frac{1}{300 \sqrt{n}},
\end{align}
(see \eqref{e:sigma}) we have that if $y \in W \cap \wt F$, then
\begin{align}
    \phi(y) \in V^1_+.
\end{align}
We set 
\begin{align}
    & F_1 := \phi(\wt F) \mbox{ and }\label{e:F1} \\
    & F_2 := h(F_1) = h(\phi(\wt F)). \label{e:F2} 
\end{align}
In particular, $F_2 \subset  V_Q^1$ by Lemma \ref{l:retraction}.

\subsubsection{Construction of $F_3$}
We want to project $F_2$ back into a $d$-dimensional skeleton, since this is a requirement to belong to $\dF_0$ (and so eventually to $\dF$). By definition of $\dF_0$, we will be projecting onto the $d$-skeleton of cubes coming from $\Delta_{\rho}$.   
We will use the following Lemma, which is taken from \cite{david2000uniform}.

\begin{lemma}[{Lemma 11.14, \cite{david2000uniform}}] \label{l:FF2}
Let $j \in \Z$ and let $A$ be a compact subset of $\R^n$ such that $\hd(A)<\infty$. Denote by $N(A)$ the union of all the cubes $I \in \Delta_j$ that touch a cube in $\Delta_j$ which intersects $A$. Then there is a Lipschitz mapping $f: \R^n \to \R^n$ with the following properties.
\begin{align}
    & f(x) = x \mbox{ for } x \in \R^n \setminus N(A); \label{e:FF2a}\\
    & f(x) = x \mbox{ for all } x \in \mathcal{S}_{j,d}; \label{e:FF2b} \\
    & f(A) \subset \mathcal{S}_{j,d}; \label{e:FF2c}\\
    & f(I) \subset I \mbox{ for all } I \in \Delta_j; \label{e:FF2d} \\
    & \hd\ps{f((A \cap I) \setminus \dS_{j,d})} \leq C \,\hd\ps{(A \cap I) \setminus \dS_{j,d}} \mbox{ for all } I \in \Delta_j. \label{e:FF2e}
\end{align}
\end{lemma}
\noindent
Recall the definition of $\dS_{j,d}$ in \eqref{e:Sjd}. \\
\\
\noindent
We now apply Lemma \ref{l:FF2} with \begin{align*}
    j=-j(\rho) \mbox{ and } A = F_2= h(\phi(\wt F)), 
\end{align*}
and thus we set
\begin{align} \label{e:F3}
    F_3= f(F_2)= (f \circ h \circ \phi)(\wt F). 
\end{align}

\begin{remark}\label{r:rem-A}
Let us note a couple of facts. First, we see that if $y \in \wt F \setminus W$, then $\phi(y)=y$ (by definition of $W$, as in \eqref{e:QM-W}); but, still with the same $y$, also $h(\phi(y))=h(y)=y$, since $y \in V_Q^1$ already, and $h$ does not move such points (as in \eqref{e:h3}), and further, $f(h(\phi(y)))= f(y) =y$ by \eqref{e:FF2b}, since $y \in \wt F$, and therefore it belongs to the $d$-face of some cubes from $\Delta_\rho$. 
\end{remark}
\begin{lemma}
With the notation as above,  we have
\begin{align} \label{e:dim}
    \dim \ps{F_3 \setminus \dS_{-j(\rho), d}} \leq d-1. 
\end{align}
\end{lemma}
\begin{proof}
By Remark \ref{r:rem-A}, we already know that $f(h(\phi(\wt F \setminus W)))= \wt F \setminus W \subset \dS_{-j(\rho), d}$. On the other hand, we must have that $f(h(\phi(\wt F \cap W))) \subset \dS_{-j(\rho), d}$ by \eqref{e:FF2c}. Thus the Lemma follows. 
\end{proof}

\subsubsection{Construction of $F_4$ and $F_4 \in \dF$}
Note that $F_3$ is not necessarily a union of full $d$-dimensional faces: the projection $f$ is into and not necessarily onto. 
\begin{lemma}
There exists a Lipschitz map $\pi: \R^n \to \R^n$ so that $\pi(F_3)$ is precisely the union of those $d$-dimensional faces which were contained in $F_3$. 
\end{lemma}
\begin{proof}
See \cite{david2004hausdorff}, pages 211-212. The idea of the proof is to consider those faces $T$ which intersect $F_3$ on a set of positive $d$-dimensional measure but that are not contained in $F_3$. On these faces, exactly because they are not contained in $F_3$, we can define a Lipschitz map $\pi$ which sends whatever lies of $F_3$ in one such face to its $(d-1)$-dimensional boundary. The result, $F_4$, will be a set of full $d$-faces plus a set of dimension smaller or equal to $d-1$.
\end{proof}
\noindent
We now set 
\begin{align}
    F_4 := \pi(F_3).
\end{align}
\begin{lemma}
With notation as above,
\begin{align}
    F_4 \in \dF.
\end{align}
\end{lemma}
\begin{proof}
Once again, see \cite{david2004hausdorff}, pages 212 to 215, from equation (5.32) to mid-page 215. 
\end{proof}
\noindent
Hence $F_4$ is a valid competitor in $\dF$ of $\wt F$. But $\wt F$ is a minimiser of the functional $J$ in this class, and therefore we have the inequality
\begin{align} \label{e:F4-F}
    J(\wt F) \leq J (F_4). 
\end{align}
We will use this inequalty in the following subsection to finally prove that $\wt F$ is also a quasiminimiser of $\hd$. 

\subsubsection{$\wt F$ is a quasiminimiser}
First, note that $F_4 \subset F_3$, except perhaps from a set of dimension smaller than, or equal to, $d-1$. Thus, using also \eqref{e:F4-F}, we have that
\begin{align*}
    J(\wt F) \leq J(F_3). 
\end{align*}
Recall the definition of $W$ in \eqref{e:QM-W} and that of $J$ in \eqref{e:J}. Writing $\wt F = \ps{\wt F \cap W} \cup \ps{\wt F \setminus W}$, and using the additivity of $\hd$, we have
\begin{align}
    J(\wt F) = J(\wt F \cap W) + J(\wt F \setminus W). 
\end{align}
Let us set 
\begin{align*}
    \Phi(y) = \ps{f \circ h \circ \phi} (y).
\end{align*}
With this notation we see that $F_3= \Phi(\wt F)$ (this is just \eqref{e:F3}). Moreover, recall from Remark \ref{r:rem-A}, that $\Phi(\wt F) = \Phi(\wt F \cap W) \cup \Phi(\wt F \setminus W)$, and in turn, that $\Phi(\wt F \setminus W) = \wt F \setminus W$, and thus $J(\Phi(\wt F \setminus W)) = J(\wt F \setminus W)$, which is immediate from the definition of $J$ as in \eqref{e:J}. In particular we get that
\begin{align*}
    J(F_3) & = J(\Phi(\wt F ))  \\
    & \leq J\ps{\Phi(\wt F \cap W)} + J\ps{\Phi(\wt F \setminus W)} \\
    & = J\ps{\Phi(\wt F \cap W)} + J\ps{\wt F \setminus W}. 
\end{align*}
We thus have
\begin{align*}
   J(\wt F)=  J(\wt F \cap W) + J(\wt F \setminus W) \leq J(\Phi(\wt F \cap W)) + J(\wt F \setminus W), 
\end{align*}
which, subtracting $J(\wt F \setminus W)$ from both sides, gives,
\begin{align*}
    J(\wt F \cap W) \leq J(\Phi(\wt F \cap W)).
\end{align*}
But note that, by the definition of the functional $J$ in \eqref{e:J}, \begin{align*}
    & \hd(\wt F \cap W) \leq J(\wt F \cap W) \leq J(\Phi(\wt F \cap W)) \\
    & = \hd(\Phi(\wt F \cap W)\cap E_\rho) + M \hd(\Phi(\wt F \cap W) \setminus E_\rho) \leq M \hd(\Phi(\wt F \cap W)).
\end{align*}
That is, 
\begin{align} \label{e:QMd}
    \hd(\wt F \cap W) \leq M \, \hd(\Phi(\wt F \cap W)).
\end{align}
Note that \eqref{e:QMd} resembles the comparison estimate \eqref{e:QMmain}: we need to swap $\Phi$ with $\phi$. To do so, we need to show that up constants, the maps $f$ and $h$ did not increase the mass of $\phi(\wt F \cap W)$. 
Let us worry about $f$ first. We write
\begin{align}
    & A_1 := h(\phi(\wt F \cap W)) \cap \dS_{-j(\rho),d}\,; \\
    & A_2 := h(\phi(\wt F \cap W)) \setminus \dS_{-j(\rho),d}\, .
\end{align}
Now, because $f(y) = y $ whenever $y \in \dS_{-j(\rho), d}$, we immediately have that
\begin{align} \label{e:QMi}
    \hd(f(A_1)) = \hd(A_1).
\end{align}
Let us look at $A_2$. Because of \eqref{e:FF2d} and the fact that dyadic cubes have bounded overlap, we can write
\begin{align*}
    \hd\ps{f(A_2)} \leq \sum_{I \in \Delta_{-j(\rho)}} \hd(f(A_2 \cap I)) \leq C \sum_{I \in \Delta_{-j(\rho)}} \hd(f(\overline{A_2} \cap I \setminus \dS_{-j(\rho),d} )).
\end{align*}
To estimate this last sum, we use \eqref{e:FF2e}:
\begin{align*}
   &  C \sum_{I \in \Delta_{-j(\rho)}} \hd(f( \overline{A_2} \cap I \setminus \dS_{-j(\rho),d} )) \nonumber \\
   &\leq C' \sum_{I \in \Delta_{-j(\rho)}} \hd(\overline{A_2} \cap I \setminus \dS_{-j(\rho),d})\nonumber\\
   & \leq C'\hd\ps{\overline{ h(\phi(\wt F \cap W))} \setminus \dS_{-j(\rho), d}}.
\end{align*}
Putting together these estimates, we see that
\begin{align} \label{e:QMf}
\hd(\Phi(\wt F \cap W)) \leq \hd(A_1) + C'\hd(\overline{ h(\phi(\wt F \cap W))} \setminus \dS_{-j(\rho), d}).
\end{align}
\begin{lemma}
With the notation as above, we have
\begin{align} \label{e:QMg}
    \dim\ps{\overline{h(\phi(\wt F \cap W))} \setminus (h(\phi(\wt F \cap W)) \cup \dS_{-j(\rho), d})} \leq d-1.
\end{align}
\end{lemma}
\begin{proof}
This is equation 5.60 in \cite{david2004hausdorff}. As the proof is brief, we add it for completeness. 
Set $X:= h(\phi(\wt F \cap W))$. Let $x \in \overline{X}\setminus X$. Since $W$ is compactly contained in $B_2$, and $V^1_+$ is closed, then there exists a 
 $y \in \wt F \cap \overline{W}$ so that $x=h(\phi(y))$. Also, $y \notin W$, since $x \notin X$. But then, by definition of $W$ in \eqref{e:QM-W}, we have that $\phi(y)=y$. Also, $y \in \wt F$, and $\wt F \subset V^1_Q \subset V^1_+$; hence $h(y)=y$. This implies that $x \in \wt F$. Now, if $x \in F^*$ (the coral part), then it belongs to $\dS_{-j(\rho), d}$. But recall also from \eqref{e:F_03} that $\wt F$ is constituted by a subset of $\dS_{-j(\rho), d}$ and a subset $L=\wt F \setminus \wt F^*$ of dimension lower than $d-1$. Thus we must have $\overline{X}\setminus X\cup (\dS_{-j(\rho), d}) \subset L$. This proves the lemma.
\end{proof}
\noindent
Using \eqref{e:QMg}, we then can write
\begin{align}
    & \hd(A_1) + C' \hd(\overline{h(\phi(\wt F \cap W))} \setminus \dS_{-j(\rho),d}) \nonumber\\
    &
    \leq \hd(A_1) + C' \hd(h(\phi(\wt F \cap W)) \setminus \dS_{-j(\rho),d}\nonumber \\
    & = \hd(A_1) + C' \hd(A_2) \nonumber\\
    & \leq C' \hd(h(\phi(\wt F \cap W))).  \label{e:QMh}
\end{align}
Hence, \eqref{e:QMf} and \eqref{e:QMh} tell us that
\begin{align} \label{e:QMm}
     \hd(\Phi(\wt F \cap W)) \leq C'\hd(h(\phi(\wt F \cap W))).
\end{align}
Now note that because $\pi$ is Lipschtz with constant $4^n$ as for Lemma \ref{l:retraction}, we immediately see that
\begin{align} \label{e:QMl}
    \hd(h \circ \phi(\wt F \cap W)) \leq 4^{nd}\hd(\phi(\wt F \cap W)).
\end{align}
The two estimates \eqref{e:QMm} and \eqref{e:QMl} together show that $\wt F \cap B_2$ is a $(B_2, k, \delta)$-quasiminimal set (with $B_2$ as defined in \eqref{e:Bj}, 
\begin{align} \label{e:QM-kappa}
    k=4^{nd} C' M
\end{align}
and 
\begin{align}\label{e:QM-delta-final}
    \delta= c_3 \min\ck{\alpha_1, \eta_1} \ell(Q).
\end{align}
This finishes the proof of Lemma \ref{l:quasiS}.

\section{A uniformly rectifiable set covering the minimising set $\wt F$.}
In this short section, we will use the main result of \cite{david2000uniform}, to show that $\wt F$ can be locally covered by a uniformly rectifiable set.

\subsubsection{The story so far.}
Let $E$ be the stable $d$-surface that we started with. From Lemma \ref{l:TP-LCR} we know that it is lower content ($d, \cc_1$)-regular. Hence we apply the coronisation as described in Subsection \ref{s:ER} (Lemma \ref{l:corona}), which is a collection of `top cubes' $\{R\}_{R \in \Top(k_0)}$. For each such an $R$ we construct an approximating set which we called $\wt E_R$ (see its definition in \eqref{e:wtER}, and see recall the definition of $\cC_R$ in \eqref{e:CR}). We subsequently modify it as in \eqref{e:newER}. Next we defined coherent Federer-Fleming projections (Lemma \ref{l:FF}), through which one can map $E$ into $E_R$ in a Lipschitz continuous way (without control on the Lipschitz constant). Via Federer-Fleming projections, we show that the fact that $E$ is a stable $d$-surface (with finite $d$-measure), implies that $E_R$, for each $R \in \Top(k_0)$, is also topologically stable, in the sense of the skeletal topological condition in Definition \ref{d:TND}; see Lemma \ref{l:TC-STC}. In Section \ref{s:E-rho-STC} we construct, for technical reasons, a further approximating set $E_{R, \rho}$ (for each $R \in \Top(k_0)$) which also satisfies the STC (Lemma \ref{l:TC-STC-Erho}). Then, it turns out that the minimiser $\wt F$ of the functional $J$, see \eqref{e:J}, has large intersection with $E_{R, \rho}$ (Lemma \ref{l:large-intersection}). Moreover, by virtue of being a minimiser of $J$, $\wt F$ is a quasiminimizer of $\hd$ (Lemma \ref{l:quasiS}). 
\\
\\
\noindent
Recall that we are dealing with a fixed Christ-David cube $Q \in \Tree(R)$ ($R \in \Top(k_0)$) and with $(x_2, r_2) \in E \times (\ell(Q)/3, 2\ell(Q)/3)$ (as they were chosen in \eqref{e:TND-x} and \eqref{e:TND-r_2}). For simplicity of notation we set $x_2=x$ and $r_2=r$ (see \eqref{e:x-r}) at the beginning of Section \ref{s:functional}; we retain this notation. 
\begin{lemma}\label{l:UR-cover}
With notation as above \textup{(}in particular recall the definition of $B_1$ in \eqref{e:Bj}\textup{)}, we have that
\begin{align*}
    \wt F \cap B_1 \subset Z, 
\end{align*}
where $Z$ is an Ahlfors $d$-regular set which is also uniformly $d$-rectifiable set and has {\rm BPLG}. The Ahlfors regularity constants, and those of uniform rectifiability and {\rm BPLG} only depend on $n, d, M$ \textup{(}where $M$ is as in \eqref{e:M}\textup{)}, $\alpha_1$ and $\eta_1$. In particular, they do not depend on $x, r, Q$ or $R$.  
\end{lemma}
\noindent
Let us recall the main result in \cite{david2000uniform}. 
\begin{theorem}[{\cite{david2000uniform}, Theorem 2.11}] \label{t:UR}
    Let $U$ be an open set in $\R^n$, and suppose that $S$ is a $(U, k, \delta)$-quasiminimizer for $\hd$. Let $S^*$ be the support in $U$ of the restriction of $\hd$ to $S$. Then for each $x \in S^*$ and radius $R_0$ which satisfy
    \begin{align}
        0<R_0< \delta \mbox{ and } B(x, 3R_0) \subset U,
    \end{align}
    there is a compact, Ahlfors $d$-regular set $Z$ such that
    \begin{align}
        S^* \cap B(x,R_0) \subset Z \subset S^* \cap B(x,2R_0)
    \end{align}
    and
    \begin{align}
        Z \mbox{ is uniformly rectifiable and contains big pieces of Lipschitz graphs.}
    \end{align}
    The Ahlfors regularity constants of $Z$, as well as those pertaining uniform rectifiability and {\rm BPLG} can be taken to depend only on $n$ and $k$.
\end{theorem}

\begin{proof}[Proof of Lemma \ref{l:UR-cover}]
Recall that $\wt F$ is a minimiser of the functional $J$ over $\dF$, see definitions \eqref{e:J} and \eqref{e:dF}.
By Lemma \ref{l:quasiS}, we know that $S= \wt F \cap B_2$ is a $(B_2, k, \delta)$-quasiminimizer for $\hd$. Also recall that
\begin{align*}
    B_2=B\ps{x, r- \frac{8}{10} \eta_1 \ell(Q) } \mbox{ (see \eqref{e:Bj})},
\end{align*}
that $k = C4^{nd}M$, and $\delta= c_3 \min\{\alpha_1, \eta_1\} \ell(Q)$. Then, making $c_3$ smaller if needed, we see that for all points $y \in \wt F \cap B_1$, we have 
\begin{align}
    B(y, 3\delta) \cap \wt F^* \subset B_2,
\end{align}
where recall that $\wt F^*$ is the coral part of $\wt F$ (see \eqref{e:F_03}).
By Theorem \ref{t:UR}, we see that there exists a uniformly rectifiable set $Z_y$ so that
\begin{align} \label{e:Zy}
    \wt F^* \cap B(y, \delta/2) \subset Z_y \subset \wt F^* \cap B(y, \delta),
\end{align}
since we can chose $R_0 \geq \delta/2$.
Now, clearly
\begin{align*}
    \bigcup_{y \in B_1\cap \wt F} B(y, \delta/2) \supset \wt F \cap B_1. 
\end{align*}
Moreover, we can find a finite subfamily of balls $\{B(y_j, \delta/2)\}_{j=0}^N$ such that
\begin{align}\label{e:fin-balls}
    \wt F \cap B_1 \subset \bigcup_{i=0}^N B(y_i, \delta/2) 
\end{align}
and 
\begin{align}\label{e:C-N}
    N \leq C = C(n, \eta_1, \alpha_1). 
\end{align}
To see this, recall that $r(B_1)= r-9/10\eta_1 \ell(Q)$. This, by the choice of $r=r_2$ in \eqref{e:TND-r_2}, and choosing the constant $C$ in \eqref{e:eta} appropriately, gives that $r(B_1) \leq \ell(Q)/2$. Hence, since $\delta=c_3\min\{\alpha_1, \eta_1\}\ell(Q)$, we need at most $C$ balls in \eqref{e:fin-balls}, where $C$ depends only on $n, \alpha_1$ and $ \eta_1$. Now, for this each one of these balls, we take the corresponding uniformly rectifiable set $Z_{y_i}$ as in \eqref{e:Zy}, and we set
\begin{align}\label{e:Zx}
    Z_x := \bigcup_{i=0}^N Z_{y_i}.
\end{align}
Then $Z_x$ is an Ahlfors $d$-regular, uniformly rectifiable set which has  BPLG and so that
\begin{align*}
    \wt F^* \cap B_1 \subset \bigcup_{i=0}^N B(y_i, \delta/2) \cap \wt F^* \subset \bigcup_{i=0}^N Z_{y_i} = Z_x. 
\end{align*}
A word on the constants involved: the Ahlfors regularity, uniform rectifiability and BPLG constants of $Z_{y_i}$ depend only on\footnote{In the current case $k = 4^{nd} C' M$ (see \eqref{e:QM-kappa}) where $C'$ is an absolute constant.} $n$ and $k$. In particular, they \textit{do not depend on } $y_i$, $x$, $r$, $Q$ or $R$.
When we take the (uniformly) finite union \eqref{e:Zx}, Ahlfors regularity, uniform rectifiability and BPLG constants come to depend on $C$ as well, which depend (see \eqref{e:C-N}) on $n, \alpha_1, \eta_1$.  
This proves the lemma. 
\end{proof}

\begin{remark}
A short summary of what has been done so far will be useful to the reader in the subsequent section. 
\\
\\
\noindent
We started off with a surface $E$ satisfying the topological condition \eqref{e:TC} with some prescribed parameters $r_0,\, \alpha_0, \, \eta_0$ and $\delta_0$. We took the corona construction from Lemma \ref{l:corona}, and we showed that the topological condition on $E$ implies a skeletal topological condition on all the approximating $E_R$ in the coronisation (Section \ref{s:TC-STC}, Lemma \ref{l:TC-STC}). Next, we constructed a further approximating Ahlfors regular set $E_{R, \rho}$ (see \eqref{e:E-rho}), to then show that \textit{for any point} $x \in R$ (see the choice of $x_2$ in \eqref{e:TND-x}), there is a $(B_2, k, \delta)$-quasiminimiser set $\wt F= \wt F(R, x)$ such that, first, 
\begin{align} \label{e:wtF-Erho}
    \hd(\wt F \setminus E_{R, \rho}) \leq c_2 \delta_1 \ell(Q)^d,
\end{align}
| this is equation \eqref{e:Jd}); and second, that there exists a uniformly $d$-rectifiable set $Z_x$ so that
\begin{align} \label{e:B1F-Zx}
    B_1 \cap \wt F^* \subset Z_x;
\end{align}
| this is Lemma \ref{l:UR-cover}.
\end{remark}
\begin{lemma} \label{l:UR-large-int}
Let $E \subset \R^n$ be a stable $d$-surface \textup{(}with parameters \textup{(}$r_0, \alpha_0, \delta_0, \eta_0$\textup{)}\textup{)} so that $\hd(E)<+\infty$. Let $\Top(k_0)$ be the decomposition from Lemma \ref{l:corona}. 
Let $R \in \Top(k_0)$ and consider the approximating Ahlfors $d$-regular sets $E_R$ and $E_{R, \rho}$ \textup{(}with $\rho$ as in \eqref{e:rho}\textup{)}. For each Christ-David cube $Q\in \Tree(R)$ \textup{(}$\Tree(R)$ as defined in the paragraph below \eqref{e:ST-cond}\textup{)}  there exists a uniformly $d$-rectifiable set $Z_Q$ \textup{(}which also has {\rm BPLG}\textup{)}\footnote{The constants involved in Ahlfors regularity, uniform recitfiability and BPLG have the same dependence as in Lemma \ref{l:UR-cover}.} and a set $\wt F_Q$ which is a the union of a uniformly finite family of quasiminimal sets so that, 
\begin{align} \label{e:UR-largint-1}
    \hd(\wt F_Q \cap E_{R, \rho}) \geq \delta_0 \ell(Q)^d,
\end{align}
and,
\begin{align} \label{e:UR-largint-2}
    \hd(\wt F_Q \setminus E_{R, \rho}) \leq c_2 \delta_1 \ell(Q)^d;
\end{align}
and,
\begin{align} \label{e:UR-largint-3}
     \wt F_Q \subset Z_Q.
\end{align}
\end{lemma}
\begin{proof}
Recall from \eqref{e:Bj}, that 
\begin{align*}
    B_1 = B\ps{x_2, r_2 - \frac{9}{10} \eta_1 \ell(Q)},
\end{align*}
and also recall from \eqref{e:r-R}, that $r_2 > \eta_1 \ell(Q)$. In particular we have that
\begin{align*}
    B(x, \eta_1\ell(Q)/10) \subset B_1.
\end{align*}
Note that to cover $Q$, we need at most $N'\approx_n \eta_1^{-n}$ balls centered on $R$ and with radius $\eta_1 \ell(Q)/10$. Pick one such collection that is also minimal, and thus of bounded overlap. Let it be
\begin{align*}
    B^j := B(x_j, \eta_1 \ell(Q)/10), \, x_j \in Q, \, 1 \leq j \leq N'. 
\end{align*}
For each $1\leq j \leq N'$, there correspond a quasiminimal set $\wt F=\wt F(Q, x_j)$ (and its coral part $\wt F^*$) satisfying \eqref{e:wtF-Erho} (Lemma \ref{l:large-intersection}), and a corresponding uniformly $d$-rectifiable set $Z_{x_j}$ satisfying \eqref{e:B1F-Zx} (from Lemma \ref{l:UR-cover}).
We now set 
\begin{align}\label{e:FQ}
    \wt F_Q := \bigcup_{j=1}^{N'} \wt F(Q, x_j),
\end{align}
and
\begin{align}\label{e:ZQ}
    Z_Q := \bigcup_{j=1}^{N'} Z_{x_j}
\end{align}
It is then easy to see that \eqref{e:UR-largint-1}, \eqref{e:UR-largint-2} and \eqref{e:UR-largint-3} hold.

\end{proof}
\begin{lemma}\label{l:BP2LG-1}
For each pair $(x,r) \in E_{R,\rho} \times (0, \ell(R))$, there exists a uniformly $d$-rectifiable set $Z_{x,r}$ such that 
\begin{align*}
    \hd(E_{R,\rho} \cap Z_{x,r} \cap B(x,r)) \geq \delta_0 r^d.
\end{align*}
The Ahlfors regularity, uniform rectifiability and {\rm BPLG} constants of the sets $Z_{x,r}$ depend on $n, d, M, \alpha_1, \eta_1$ \textup{(}as in Lemma \ref{l:UR-cover}\textup{)}; in particular they do not depend on $(x,r)$. 
\end{lemma}
\begin{proof}
If $x \in E_{R, \rho}$, then by the construction of $E_{R, \rho}$ (as in \eqref{e:cC-R-rho} and  \eqref{e:E-rho}), there exists a dyadic cube $I \in \cC_R$ such that $\dist(x, \partial_d I) < \ell(I)$. Recall also that for each $I \in \cC_R$, there exists a Christ-David cube $Q_I \in \Tree(R)$ such that $\dist(Q_I, I) \leq c \tau^{-1} \ell(I)$ and $\ell(I) \leq \ell(Q_I) \leq c' \tau^{-1} \ell(I)$. This cube is given by Lemma \ref{l:I-QI}.
\\
\\
\noindent
Now, take two constants $C$, $C'$ to be fixed below, depending on the constants $c, c'$.
\begin{enumerate}[leftmargin=0.8cm]
    \item Suppose first that \begin{align*}
        C\tau^{-1}\ell(I) \leq r \leq C' \ell(Q_I).
    \end{align*}
    Choosing $C$ appropriately, we can insure that
    \begin{align*}
        B(x,r) \supset B_{Q_I}.
    \end{align*}
    But from Lemma \ref{l:UR-large-int}, we know that $$\hd(E_{R, \rho} \cap Z_{Q_I}) = \hd(E_{R, \rho} \cap Z_{Q_I} \cap B(x,r)) \geq \delta_1 \ell(Q)^d.$$ 
    Since $r \approx \ell(Q_I)$, where the constants behind $\approx$ depend on $C,C', c, c'$, then we conclude that there is an absolute constant $C''$ so that
    \begin{align*}
        \hd(E_{R, \rho} \cap Z_{Q_I} \cap B(x,r) ) \geq C'' \delta_0 r^d.
    \end{align*}
    This give the Lemma for this case.
     \item Suppose now that
     \begin{align*}
         0< r < C \tau^{-1} \ell(I). 
     \end{align*}
     Let $T$ be a $d$-face of $\partial_d I$, and let $T(r)$ be tile of $T$ containing $x$ and with
     
     \begin{align*}
         \ell(T(r)) \approx \min\ck{\frac{1}{10}r, \ell(I)}.
     \end{align*}
     Then clearly, 
     \begin{align*}
         T(r) \subset E_{R, \rho} \cap B(x,r) \mbox{ for all } r>0,
     \end{align*}
     and $T(r)$ is a uniform rectifiable set with constants independent of $r$. Now, note that if $C \tau^{-1} > r>\ell(I)$, 
     then 
     \begin{align*}
         \hd(E_{R, \rho} \cap T(r) \cap B(x,r)) \geq \hd(T(r)) = c \ell(I)^d \approx_{\tau, C, C'} r^d.
     \end{align*}
     On the other hand, if $0<r< \ell(I)$, we have
     \begin{align*}
         \hd(E_{R, \rho} \cap B(x,r) \cap T(r)) \geq \hd(T(r)) = c r^d.
     \end{align*}
     In any case, we found a unifromly rectifiable set which intersects $E_{R, \rho}$ with measure bounded below uniformly. This gives the Lemma in this case.
     \item Now, if $C' \ell(Q_I) < r < \ell(R)$, we can repeat the arguments of point (1) for some parent of $Q_I$ appropriately chosen. By construction, this parent will be belong to $\Tree(R)$, and thus the same estimates apply.
\end{enumerate}
Note that in Cases (1) and (3) of the above list, the uniform rectifiable set $Z_{x,r}$ is in fact the set $Z_{Q_I}$ ($Q_I \in \Tree(R)$), which is described in \eqref{e:ZQ} and has Ahlfors regularity, uniform rectifiability and BPLG constants which depend only on $n, d, M, \alpha_1, \eta_1$; in particular they are independent of $x,r, Q, R$. As for Case (2), $Z_{x,r}$ is in fact a (piece of) affine $d$-plane. The constants of this set are trivially independent of $x,r, Q, R$. 
\end{proof}
\noindent
Now recall the following definitions from \cite{david-semmes93}.
\begin{definition}[\cite{david-semmes93}, Definition 1.26]\label{d:BP2LG}
We say that $E$ has BPLG (big pieces of Lipschitz graph) if $E$ is Ahlfors regular and if there exist constants $C$ and $b_1>0$ so that, for $x \in E$ and $R>0$, there is a $d$-dimensional Lipschitz graph $\Gamma$ with constant $\leq C$ such that
\begin{align*}
    \hd(E \cap \Gamma \cap B(x,R)) \geq b_1 R^d. 
\end{align*}
\end{definition}
\begin{definition}[\cite{david-semmes93}, Definition 1.32]
Let $E \subset \R^n$ has BP(BPLG) (big pieces of BPLG) if $E$ is Ahlfors $d$-regular and if there exist positive constants $C$, $b_1$ and $b_2$ so that, if $B$ is any ball centered on $E$, then there is an Ahlfors regular set $F \subset \R^n$ so that
\begin{align*}
    \hd(E \cap F \cap B) \geq b_2 \, \hd(E \cap B), 
\end{align*}
F is Ahlfors regular with constant $\leq C$, and $F$ satisfies the BPLG condition with constants $b_1$ and $C$. 
\end{definition}
\noindent
 Then we have the following (see also the paragraph below Definition 1.32 in \cite{david-semmes93}).
 \begin{proposition}[\cite{david-semmes93}, Proposition 1.28]\label{p:BP2LG-UR}
     If $E \subset \R^n$ has {\rm BP(BPLG)}, then $E$ is uniformly $d$-rectifiable. 
 \end{proposition}
 \noindent
We therefore obtain the following corollary. 
\begin{corollary}\label{c:UR-Erho}
Let $E_{R, \rho}$ be as in \eqref{e:E-rho} \textup{(}so $R \in \Top(k_0)$ with $\Top(k_0)$ as in Lemma \ref{l:corona}\textup{)}. 
The set  $E_{R, \rho}$ is an Ahlfors regular, uniformly $d$-rectifiable set.
\end{corollary}
\begin{proof}
We want to show that $E_{R,\rho}$ has BP(BPLG). We choose $C$ in Definition \ref{d:BP2LG} to be the maximum between the Ahlfors regularity and BPLG constants of the sets $Z_{x,r}$ from Lemma \ref{l:BP2LG-1}; note that neither of them depend on $x,r,R$. From Lemma \ref{l:BP2LG-1}, we know that for any ball $B=B(x,r)$ centered on $E_{R, \rho}$, there exists a set $Z_{x,r}=Z_B$ which has BPLG so that 
\begin{align*}
    \hd(E_{R,\rho} \cap Z_B \cap B) \geq \delta_0 r^d. 
\end{align*}
Thus $E_{R, \rho}$ satisfies  Definition \ref{d:BP2LG} with $E=E_{R, \rho}$, $F=Z_B$ (from Lemma \ref{l:BP2LG-1}), $C$ as above, $b_1=\delta_0$ and $b_2$ equal to the BPLG constant of $Z_B$, which is independent of $B$. From Proposition \ref{p:BP2LG-UR}, we see that $E_{R,\rho}$ is uniformly $d$-rectifiable. 
\end{proof}
\noindent
For future use, let us pin down an easy fact about the distance between $R \subset E$ and $ E_{R, \rho}$. 
\begin{lemma}\label{l:Erho-dist}
For each $x \in S$, with $S \in \Stop(R)$, we have
\begin{align*}
    \dist(x, E_{R, \rho}) \leq C \ell(S). 
\end{align*}
\begin{proof}
Using again Lemma \ref{l:meta1}, we see that if $S \in \Stop(R)$, then there exists a cube $I_S \in \cC_R$ such that $\zeta_S \in I_S$ and $\ell(S) \approx \tau^{-1} \ell(I_S)$. Thus in particular, if $x \in S$, then $\dist(x, E_R) \lesssim \ell(S) \approx \tau^{-1} \ell(I_S)$. Further, by construction of $E_{R, \rho}$, we have that $E_R \subset E_{R, \rho}$. This proves the Lemma. 
\end{proof}
\end{lemma}

\section{Estimates on the $\beta$ coefficients}\label{s:beta}
In this section we give the estimates on the $\beta$ coefficient $\beta_E^{p,d}$ (as defined in \eqref{e:beta-p-cont}) which will then give Theorem \ref{t:topo-main}. The main proposition of this section will be the following one\footnote{Recall that it suffices to prove Proposition \ref{p:main-one-dir} with $\aA=3$ and $p=2$. This was established in Lemma \ref{l:redux-Ap}. Hence we will prove Proposition \ref{p:main-one-dir-k0} for this choice of $\aA$ and $p$.}. 

\begin{proposition} \label{p:main-one-dir-k0}
Assume that $1 \leq d < n$. Let $E \subset \R^n$ be a stable $d$-surface with parameters $r_0, \alpha_0, \delta_0$ and $\eta_0$ and let $\dD$ be its Christ-David cubes. Let $Q_0 \in \dD$ be an arbitrary Christ-David cube\footnote{This is effectively the same Christ -David cube $Q_0$ as in the statement of Theorem \ref{t:topo-main}.} satisfying $\ell(Q_0)< r_0$.
 Let $k_0 \in \N$ be as in Lemma \ref{l:corona} and denote by $\dD(Q_0, k_0)$ the family of Christ-David cube $\cup_{k=0}^{k_0} \{ Q \in \dD_k \, |\, Q \subset Q_0\}$. 
 Then it holds that
\begin{align}\label{e:main-k_0}
    {\rm diameter}(Q_0)^d + {\rm BETA }(Q_0, k_0) \lesssim \hd(Q_0),
\end{align}
where 
\begin{align}\label{e:BETA-k_0}
    {\rm BETA}(Q_0,k_0) = {\rm BETA}_{E, 3, 2, d} (Q_0,k_0): = \sum_{Q \in \dD(Q_0,k_0)} \beta_E^{d,2} (3 B_Q)^2 \ell(Q)^d,
\end{align}
and where the constant behind the symbol $\lesssim$ depends on the parameters of \eqref{e:TC} \textup{(}$\alpha_0$,$\delta_0$, $r_0$, $\eta_0$, $\gamma_0$\textup{)}, and on the dimensional parameters \textup{(}$d$, $n$\textup{)}, but not on $k_0$.
\end{proposition}

\noindent
We will use the following lemma\footnote{The lemma below actually holds for a wider range of $p$ but we will just need it for $p=2$.}.
\begin{lemma}[{\cite{azzam2018analyst}, Lemma 2.21}] \label{lemma:azzamschul}
Let $1 \leq p < \infty$ and $E_1, E_2 \subset \R^n$. Let $x \in E_1$ and fix $r>0$. Take some $y \in E_2$ so that $B(x,r) \subset B(y, 2r)$. Assume that $E_1, E_2$ are both lower content \textup{(}$d, \cc_1$\textup{)}-regular. Then
\begin{align*}
    \beta_{E_1}^{p,2} (x,r) \lesssim_{\cc_1} \beta_{E_2}^{p,2} (y, 2r) + \left( \frac{1}{r^d} \int_{E_1 \cap B(x, 2r)} \left(\frac{\dist (y, E_2)}{r} \right)^2 \, d \hdc(y)\right)^{\frac{1}{2}}.
\end{align*}
\end{lemma}
\noindent
Let us now get started with the proof of Proposition \ref{p:main-one-dir-k0}. Recall the decomposition of $\dD(Q_0, k_0)$ from Lemma \ref{l:corona} and the collection of `top cubes' $R \in \Top(k_0)$. Recall also the definition of $\Forest(R)$ in \eqref{e:forest}.
\begin{lemma} \label{lemma:betatree}
Fix an integer $k_0>0$ sufficiently large as in Lemma \ref{l:corona} and let $R \in \Top(k_0)$.
We have
\begin{align} 
    \sum_{\substack{Q \in \Tree(R)\\ Q \in \dD(Q_0, k_0)}} \beta_{E}^{p,2}(3 B_Q)^2 \ell(Q)^d & \lesssim  \ell(R)^d,
\end{align}
where the constant behind the symbol $\lesssim$ depends on $n, d$, and the Ahlfors regularity constant of $E_R$ and $E_{R, \rho}$ \textup{(}which was called $\Cc_3$ and fixed in Section \ref{s:E-rho-STC}, Remark \ref{rem:C3}\textup{)}, but not on $k_0$. 
\end{lemma}

\begin{proof}
We want to apply Lemma \ref{lemma:azzamschul} with $E_1=E$ and $E_2 = E_{R, \rho}$. 
For $Q \in \dD$, recall that $z_Q$ denotes the center of $Q$. By the definition of $ \Tree(R)$, we see that if $Q \in \Tree(R)$, then there must exists a dyadic cube $I \in \cC_R$ which meets $Q$. The $d$-skeleton $\partial_d I$ of $I$ is part of $E_{R, \rho}$. We see that $\ell(I) \lesssim \tau \ell(Q)$. Hence 
\begin{align}\label{e:XQ}
    \mbox{there exists a point } \quad  x_Q \in E_{R, \rho} \quad \mbox{ such that } \quad |z_Q - x_Q| \leq 4 \tau \ell(Q),
\end{align}
and we obtain that
\begin{align*}
    B_Q := B(z_Q, \ell(Q)) \subset B(x_Q, 2 \ell(Q))=: B_Q'.
\end{align*}
 This implies that for each cube $Q \in  \Tree(R)$ the hypotheses of Lemma \ref{lemma:azzamschul} are satisfied (with $E_1 = E$ and $E_2 = E_{R, \rho}$); therefore we apply Lemma \ref{lemma:azzamschul} and write\footnote{Recall the notational convention for Choquet integrals \eqref{e:choquet}.}
\begin{align*}
    \sum_{\substack{P \in \Tree(R)\\ P \in \dD(Q_0, k_0)}} \beta_{E}^{p, 2} (3 B_P)^2 \ell(P)^d  \lesssim \sum_{\substack{P \in \Tree(R)\\ P \in \dD(Q_0,k_0)}} \beta_{E_{R, \rho}}^{d,2} (6 B'_P)^2 \, \ell(P)^d \\
    + \sum_{\substack{P \in  \Tree(R)\\ P \in \dD(Q_0,k_0)}} \ps{\frac{1}{\ell(P)^d} \int_{6 B_P\cap E} \ps{\frac{\dist(y, E_{R, \rho})}{\ell(P)}}^2 \, d \hdc(y) }\, \ell(P)^d \\
    & := I_1 + I_2.
\end{align*}
First, let us look at $I_1$. We apply Theorem \ref{t:Christ} to $E_{R,\rho}$; let us denote the cubes so obtained by $\dD(E_{R,\rho})$. Note that for each $P \in \Tree(R)$ with $P \in \dD(Q_0, k_0)$, $x_P$ (this is the point in $E_{R, \rho}$ found above in \eqref{e:XQ}) belongs to some cube $ P' \in \dD(E_{R, \rho})$ so that $\ell(P') \approx \ell(P)$; hence there exists a constant $C_4 \geq 1$ so that 
\begin{align*}
    6 B'_P \subset C_4 B_{P'}.
\end{align*}
This in turn implies that $\beta_{E_{R,\rho}}^{p,2}(6 B'_P) \lesssim_{n,d,C_4} \beta_{E_{R,\rho}}^{d,2}(C_4 B_{P'})$. Hence, 
\begin{align}
    \sum_{\substack{ P \in \Tree(R) \\ P \in \dD(Q_0,k_0)}} \beta_{E_{R,\rho}}^{d,2} (6 B'_P)^2 \ell(P)^d \lesssim_{n,d, C_4} \sum_{\substack{P' \in \dD(E_{R,\rho}) \\ \ell(P') \lesssim \ell(R) }} \beta_{E_{R,\rho}}^{d,2} (C_4 B_{P'})^2 \ell(P')^d. 
\end{align}
Since $E_{R,\rho}$ is uniformly rectifiable (by Corollary \ref{c:UR-Erho}), we immediately have that $I_1 \lesssim \ell(R)^d$ by Theorem \ref{t:DS-beta}. Let us remark that the content $\beta$ number we are using are comparable to the one introduced by David and Semmes when computed on Ahlfors regular sets.
\\
\\
\noindent
We now now estimate $I_2$. 
Let $y \in 6B_R$; by Lemma \ref{l:Erho-dist}, there exists a cube $S \in \Stop(R)$ such that 
\begin{align} \label{eq:SS_distyER}
    \dist(y, E_{R,\rho}) \lesssim \ell(S);
\end{align}
We can estimate the integral in $I_2$ with \eqref{eq:SS_distyER} as follows.
\begin{align*}
    \int_{6 B_P\cap E} \ps{\frac{\dist(y, E_{R,\rho})}{\ell(P)}}^2 \, d \hdc(y) & \leq \sum_{P' \in \dN(P)} \int_{P'} \ps{\frac{\dist(y, E_{R,\rho})}{\ell(P)}}^2 \, d \hdc(y)\\
   &  \lesssim \sum_{P' \in \dN(P)} \sum_{\substack{S \in \Stop(R) \\ S \subset P' }} \int_{S} \frac{\ell(S)^2}{\ell(P)^2}  \\
    & \lesssim \sum_{P' \in \dN(P)} \sum_{\substack{S \in \Stop(R) \\ S \subset P' }} \frac{ \ell(S)^{d+2}}{\ell(P)^{2}},
\end{align*}
and so 
\begin{align*}
    I_2 \lesssim \sum_{\substack{P \in \Tree(R)\\ P \in \dD(k_0)}}
    \sum_{P' \in \dN(P)} \sum_{\substack{S \in \Stop(R) \\ S \subset P' }} \frac{\ell(S)^{d + 2}}{\ell(P)^{ 2}}.
\end{align*}
We now swap the sums (which are all finite), to obtain that
\begin{align}
    I_2 & \lesssim \sum_{\substack{S \in \Stop(R)}}  \ell(S)^{d + 2} \sum_{\substack{P \in  \Tree(R) \\ \exists P' \in \dN(P):\, P' \supset S}}  \frac{1}{\ell(P)^{ 2}}\nonumber  \\
    & \lesssim_{d,n} \sum_{\substack{S \in \Stop(R)}} \ell(S)^{d + 2} \sum_{\substack{P \in  \Tree(R) \\ \exists P' \in \dN(P):\, P' \supset Q}} \frac{1}{\ell(P)^{2}}. \label{e:estBeta1}
\end{align}
We see that the number of cubes $P \in  \Tree(R)$ of a given generation so that there exists a sibling $P' \in \dN(P)$ for which $P' \supset S$ is bounded above by a universal constant depending on $n$. Thus 
 we can sum geometrically the interior series:
\begin{align*}
    \sum_{\substack{P \in  \Tree(R) \\ \exists P' \in \dN(P):\, P' \supset Q}} \frac{1}{\ell(P)^{ 2}} \lesssim_{n} \frac{1}{\ell(S)^{2}}. 
\end{align*}
Therefore we obtain 
\begin{align*}
    \eqref{e:estBeta1} \lesssim \sum_{\substack{S \in \Stop(R) }}  \frac{\ell(S)^{d +2 }}{\ell(S)^{2}} = \sum_{\substack{S \in \Stop(R)}} \ell(S)^d.
\end{align*}
This latter sum is bounded above by $C \ell(R)^d$. This concludes the proof of the lemma.
\end{proof}

\begin{proof}[Proof of Proposition \ref{p:main-one-dir-k0}]
Let $Q_0$ as in the statement of the proposition. Then
\begin{align}
    \sum_{Q \in \dD(Q_0, k_0)} \beta_E^{d,2}(3 B_Q)^2 \ell(Q)^d & \lesssim \sum_{\substack{R \in \Top(k_0) \\ R \subset Q_0}} \sum_{\substack{Q \in \Tree(R)\\ Q \in \dD(Q_0,k_0)}} \beta_{E}^{p,d}(3 B_Q)^2\ell(Q)^d. \label{e:betaest-10}
\end{align}
By Lemma \ref{lemma:betatree}, we see that 
\begin{align}
     \eqref{e:betaest-10} \lesssim \sum_{\substack{R \in \Top(k_0) \\ R \subset Q_0}} \ell(R)^d. 
\end{align}
Note that each $R \in \Top(k_0)$ is the child of some stopped cube $R'$. Recall: we stopped at a Christ-David cube $R' \in \mathcal{D}$ whenever it happened that $R' \cap I$ and $\ell(I) \approx \ell(Q)$ for some dyadic cube $I \in \Bad$. We can therefore associate to each $R \in \Top(k_0)$ a bad dyadic cube $I$, and thus, by \eqref{e:Bad-est}, we have that
\begin{align}
    \sum_{\substack{R \in \Top(k_0)  \\ R \subset Q_0}} \ell(R)^d \lesssim \sum_{\substack{I \in \Bad\\ I \subset B_{Q_0}}} \ell(I)^d \lesssim \hd(E\cap B_{Q_0}).
\end{align}
The estimate in \eqref{e:Bad-est} is independent of $k_0$, so is the one we obtained here.
All in all, we see that,
\begin{align*}
   \eqref{e:betaest-10}  \lesssim  \hd(E \cap B_{Q_0}).
\end{align*}
\noindent
This concludes the proof of Proposition \ref{p:main-one-dir-k0}
\end{proof}

\section{End of proof of Theorem \ref{t:topo-main}
}\label{s:end}
In this brief section we wrap up the proof of our main result. Recall that, starting with Section \ref{s:prel}, our main objective was to prove Proposition \ref{p:main-one-dir}, which we rewrite here for convenience. 
\begin{proposition} 
Assume that $1 \leq d < n$. Let $E \subset \R^n$ be a stable $d$-surface with parameters $r_0, \alpha_0, \delta_0$ and $\eta_0$ and let $\dD$ be its Christ-David cubes. Let $Q_0 \in \dD$ be an arbitrary cube satisfying $\ell(Q_0)< r_0$ and let $\aA>1$. Take $1 \leq p < p(d)$ with 
\begin{align*}
p(d):= \left\{ \begin{array}{cl} \frac{2d}{d-2} & \mbox{if } d>2 \\
 \infty & \mbox{if } d\leq 2\end{array}\right. 
 \end{align*}
 Then it holds that
\begin{align}
    {\rm diameter}(Q_0)^d + {\rm BETA }(Q_0) \lesssim \hd(Q_0),
\end{align}
where 
\begin{align}
    {\rm BETA}(Q_0) = {\rm BETA}_{E, \aA, p, d} (Q_0): = \sum_{Q \in \dD(Q_0)} \beta_E^{d,p} (\aA B_Q)^2 \ell(Q)^d,
\end{align}
and where the constant behind the symbol $\lesssim$ depends on the parameters of $\eqref{e:TC}$ \textup{(}$\alpha_0, \delta_0, r_0, \eta_0$\textup{)}, on $\aA$, on $p$, and on the dimensional parameters \textup{(}$d$, $n$\textup{)}.
\end{proposition}

\begin{proof}[Proof of Proposition \ref{p:main-one-dir}]
Using Proposition \ref{p:main-one-dir-k0}, the proposition follows immediately for $\aA = 3 $ and $p=2$ by taking $k_0 \to \infty$ and recalling that the estimate \eqref{e:main-k_0} is independent of $k_0$. For the remaining $p$'s and values of $\aA$, it suffices to use Lemma \ref{l:redux-Ap}.
\end{proof}

\begin{proof}[Proof of Theorem \ref{t:topo-main}]
It follows from \eqref{e:tst} in Theorem \ref{t:AS-real} (i.e. the main result of Azzam and Schul in \cite{azzam2018analyst}),  that 
\begin{align}
    \hd(Q_0) \lesssim {\rm diameter}(Q_0)^d+ {\rm BETA}_{E, \aA, p, d} (Q_0) 
\end{align}
whenever $\aA$ is sufficiently large.
This completes the proof of Theorem \ref{t:topo-main}.
\end{proof}
\section{Construction of a stable surface covering a lower content regular set}\label{s:sigma-construction}
In this section we show that any lower content $d$-regular set can be covered by a topologically stable $d$-surface. The $d$-LCR set $E$, the `top cube' $Q_0$ and the constants $\ve$ and $\Cc_0$ will be fixed throughout the section. We choose $\ve$ to be sufficiently small so that Theorem \ref{t:david-toro} below may be applied, and its value will be adjusted (i.e. made smaller so to absorb an absolute constant a finite number of times) without explicit mention. The construction below is inspired partly by one in \cite{azzam2018analyst} and partly by a well-known construction of David and Leger (see e.g. \cite{david-semmes91} or \cite{leger1999menger} - nice expositions of this might be found in Tolsa's book \cite{tolsa-book} or in \cite{jaye2019proof}).

\begin{remark}
In proving Theorem \ref{t:SIGMA-cor-in-sec} one may always assume that ${\rm BETA}_{E, \aA, p, d}(Q_0)< + \infty$, since otherwise one can simply cover $E$ by a union of $n$-dimensional dyadic cubes (say); this union will satisfy the topological condition trivially, but such a statement will be completely useless.  
\end{remark}

\begin{theorem}\label{t:SIGMA-cor-in-sec}
Let $\ve>0$ be sufficiently small, and $\Cc_0>1$. Let $E\subset \R^n$ be a lower content $(\cc_1, d)$-regular set and $Q_0 \in \dD(E)$. Then there is a choice of constants $(r_0$, $\alpha_0$, $\eta_0$, $\delta_0$, $\gamma_0)$ such that the following holds. We can construct a topologically stable $d$-surface (with parameters as above) $\Sigma=\Sigma(Q_0, \ve, \Cc_0)$ so that 
\begin{gather}
     Q_0 \subset \Sigma;\\
     {\rm BETA}_{E, \aA, p, d}(Q_0) + \ell(Q_0)^d \approx \hd(\Sigma).\label{e:SIGMA-const1}
\end{gather}
The constant behind the $\approx$ symbol depend on $d, n, \cc_1, \aA, p, \ve$ and the topological condition constants $(r_0, \alpha_0, \eta_0, \delta_0, \gamma_0)$ which can be chosen as follows: $r_0= \ell(Q_0)$, $\eta_0, \alpha_0, \gamma_0$ are chosen sufficiently small with respect to $\ve$ and $Q_0$,  and $\delta_0$ sufficiently small depending on $\eta_0, \gamma_0, d$. 
\end{theorem}
\noindent
Note in particular that the constant behind \eqref{e:SIGMA-const1} does not depend on $E$ (aside from its lower regularity constant). 

\subsection{A parameterisation theorem by David and Toro}
We will need the following Theorem by David and Toro \cite{david2012reifenberg}. In \cite{david2012reifenberg} it is not stated as it is below; however we provide references within \cite{david2012reifenberg} where the mentioned facts are proven. In fact, we copied this summary from \cite{azzam2018analyst}.

\begin{theorem}[\cite{david2012reifenberg}, Sections 1-9] \label{t:david-toro}
For each $m \in \N \cup \{0\}$, set $r_m := 10^{-m}$ and let $\{x^k_j\}_{J_k}$ be an $r_m$-separated net of points.
Suppose that 
\begin{enumerate}[leftmargin=0.5cm, label=\alph*)]
    \item we can find a $d$-plane $P_0 \subset \R^n$
\begin{align}\label{P_0}
    \{x_j^0\}_{j \in J_0} \subset P_0;
\end{align}
\item to each point $x_j^m$, $m \in \N$ and $j \in J_m$, we have associated a $d$-plane $P_j^m \subset \R^n$ so that
\begin{align*}
    P_j^m \ni x_j^m;
\end{align*}
\item for each $m \geq 1$, 
\begin{align}
    x_j^m \in V^2_{m-1},
\end{align}
where, for $c>0$, $V^c_m := \bigcup_{j \in J_m} B(x_j^m, c\, r_m)$. 
\end{enumerate}
Set
\begin{align*}
    B(x_j^m, r_m) =: B_j^m.
\end{align*}
For each $m \in \N$, define a function $\ve_m$ on $\R^n$ by setting
\begin{align*}
    \ve_m(x):= \sup \left\{ d_{x, 10^4 r_l}(P_j^m, P_i^l) \quad | \quad j \in J_m, \quad |l-m|<2, \quad i \in J_l, \quad x \in 100B_j^m \cap 100B_i^l\right\}.
\end{align*}

There exists an $\ve_0>0$ so that if we can find an $\ve \in(0, \ve_0)$ such that 
\begin{align}\label{e:DT4}
    \ve_m (x_j^m) < \ve \quad \mbox{ for all } \quad m >0 \quad \mbox{ and } \quad j \in J_m, 
\end{align}
then there is a bijection $g: \R^n \to \R^n$ such that the following holds. 
\begin{enumerate}[leftmargin=0.5cm, label=\textup{(}\arabic*\textup{)}]
    \item We have 
    \begin{align}
        E_\infty := \bigcap_{M=1}^\infty \overline{\bigcup_{m \geq M} \{x_j^m\}_{j \in J_m} } \subset \Sigma := g(P_0).
    \end{align}
    \item $g(z) = z $ when $\dist(z, P_0) >2$.
    \item For $x, y \in \R^n$, 
    \begin{align}
        \tfrac{1}{4}|x-y|^{1+\tau} \leq |g(x)-g(y)| \leq 10 |x-y|^{1-\tau}. 
    \end{align}
    \item $|g(z)-z| \lesssim \ve$ for $z \in \R^n$. 
    \item For $x \in P_0$, 
    \begin{align*}
        g(x) = \lim_{m \to \infty} \sigma_m \circ \sigma_{m-1} \circ \dots \circ \sigma_1(x),
    \end{align*}
    where 
    \begin{align*}
        \sigma_m (y) = y + \sum_{j \in J_m}\theta_j^m (y) \pi_{j}^n(y);
    \end{align*}
    Here $\{\theta_j^m\}_{j \in L_m \cup J_m}$ is a partition of unity so that $\1_{B_j^m} \leq \theta_j^m \leq \1_{10B_j^m}$ for all $m \in \N$ and $j \in L_m \cup J_m$, and for each $m$ $L_m$ is the index set of a $\tfrac{r_m}{2}$-separated net in $\R^n \setminus V^9_m$.
    \item \textup{(}\cite{david2012reifenberg}, Equation 4.5\textup{)} For $m \geq 0$, 
    \begin{align}
        \sigma_k(y) = y \mbox{ and } D\sigma_m(y) = I \mbox{ whenever } y \in \R^n \setminus V_m^{10}.
    \end{align}
    \item \textup{(}\cite{david2012reifenberg}, Proposition 5.1\textup{)}. Let $\Sigma_0=P_0$ and $\Sigma_m = \sigma_m(\Sigma_{m-1})$. There is a $C^2$ function $A_j^m: P_j^m \cap 49 B_j^m \to (P_j^m)^\perp$ so that $|A_j^m (x_j^m)| \lesssim \ve r_m$, $|DA_j^m| \lesssim \ve$ on $P_j^m \cap 49B_j^m$ and, if $\Gamma_j^m$ is its graph over $P_j^m$, then
    \begin{align*}
        \Sigma_m \cap D(x_j^m, P_j^m, 49r_m) = \Gamma_{j}^m \cap D(x_j^m, P_j^m,  49r_j^m),
    \end{align*}
    where 
    \begin{align*}
        D(x, P, r) := \{ x + w \, |\, z \in P \cap B(x,r), \, \, w \in P^\perp \cap B(0, r) \}.
    \end{align*}
    \textup{(}Here $P^\perp$ is the plane perpendicular to $P$ containing $0$\textup{)}. In particular, 
    \begin{align*}
        d_{x_j^m, 49r_m}(\Sigma_m, P_j^m) \lesssim \ve. 
    \end{align*}
    \item \textup{(}\cite{david2012reifenberg}, Proposition 6.3\textup{)} $\Sigma=g(P_0)$ is $C\ve$-Reifenberg flat, in the sense that for all $z \in \Sigma$ and $r \in (0,1)$ there is a $d$-plane $P=P(z, r)$ so that $d_{z, r}(\Sigma, P) \lesssim \ve$. 
    \item For all\footnote{This follows from eq. (7.13) in \cite{david2012reifenberg}.} $y \in \Sigma_m$, 
    \begin{align*}
        |\sigma_m(y)-y| \lesssim \ve_{m}(y) r_m.
    \end{align*}
    In particular, 
    \begin{align*}
        \dist(y, \Sigma) \lesssim \ve r_m \quad \mbox{ for } \quad y \in \Sigma_m. 
    \end{align*}
\end{enumerate}
\end{theorem}

\subsection{Construction of a good surface for a good cube}
The following proposition summarises what we do in the current subsection. 
\begin{proposition}
Let $E \subset \R^n$ be a lower content $d$-regular set, where $d \in \N$ and $d<n$. Let $R \in \dD$ will be a Christ-David cube with $\inf_{P: d\rm{-plane}} d_{\Cc_0B_R}(E, P)< \ve$. Fix two constant $\ve>0$ and $\Cc_0>1$. If $\ve>0$ is chosen sufficiently small, we can construct a Reifenberg $\ve$-flat surface $\Sigma^R$ satisfying
\begin{gather}
     \dist(x, \Sigma^R) \leq C \ve d(x) \mbox{ for all } x \in \Cc_0B_R \cap E \mbox{ and}\\
     \dist(x, E) \leq C \ve d(x) \mbox{ for all } x \in \Cc_0B_R \cap \Sigma^R. 
\end{gather}
\end{proposition}

\bigskip  
\subsubsection{Construction of separated nets}\label{sss:nets}
For a ball $B=B(x,r)$, $P_B$ or $P_{x,r}$ denotes a best approximating plane in terms of $d_B(\cdot,\cdot)$. Now let $x \in E$ and $0< r < 50 \ell(R)$. We call $B=B(x,r)$ \textit{good} and set $B \in \Gg$ if 
\begin{align*}
    d_{\Cc_0 B} (P_B, E) < \ve.
\end{align*}
We say that $B=B(x,r)$ is \textit{very good} and set $B \in \Vv\Gg$ if $B(x, s) \in \Gg$ for all $r \leq s \leq 50 \ell(R)$. For $x \in E$, we set
\begin{align}
    h(x) := \inf\{ r \leq 50 \ell(R) \, |\, B(x, r) \in \Vv \Gg\}. 
\end{align}
We will eventually need to control angles between planes to apply Theorem \ref{t:david-toro}. However, the fact that $h$ has no regularity would cause trouble. We follow a well-beaten path and introduce the following regularised version of $h$. For a point $x \in \R^n$, define
\begin{align}\label{e:dt-d}
    d(x) := \inf_{B \in \Vv\Gg} \left( r_B + \dist(x, B)\right).
\end{align}
Note that $d(\cdot)$ is an infimum over $1$-Lipschitz functions, and thus $1$-Lipschitz itself. For a ball $B=B(x,r)$, set 
\begin{align*}
    d(x,r) = d(B):= \inf_{z \in B} d(x). 
\end{align*}
The function $d$ is quite ubiquitous in stopping time arguments, one of its first appearance being Chapter 8 of \cite{david-semmes91}. Its properties are well known, but, for the sake of completeness, we will summarise them in a few lemmas below. 
We now inductively construct a sequence of maximal nets $\dN_m^R$, for $m=0, 1, 2,...$. At each step, we associate to all points in $\dN_m^R$ a $d$-plane. Put first
\begin{align*}
    \dN_0^R:= \{z_R\}, 
\end{align*}
where $z_R$ is the center of the ball $B_R$. Suppose we constructed $\dN_{m-1}^R$;
then we put
\begin{align}\label{e:dt-20}
    \wt{ \dN_m^R}:= \dN_{m-1}^Q \cup \{ x \in E \, |\, d(x, r_m) \leq 20 r_m \}.
\end{align}
For some index set $J_m^R$, we then set
\begin{align}\label{e:net-m}
    \dN_m^R= \{x_j^m\}_{j \in J_m^R} := \mbox{ a choice of maximal } r_m-\mbox{net for } \wt{\dN_m^R}. 
\end{align}

\begin{lemma}\label{l:dt-001}
With notation as above, let $x \in \Cc_0 B_R$ and suppose $d(x)>0$. Let $m=m_x$ be the minimal integer for which $d(x,r_m) > 20 r_m$. Then $d(x, r_m) \approx d(x) \approx r_m$. 
\end{lemma}
\begin{proof}
Let $y \in B(x, 10r_m)$; since $d$ is $1$-Lipschitz, we have that
\begin{align*}
    d(y) \geq d(x, r_m) - 10 r_m \geq 10 r_m. 
\end{align*}
On the other hand, since $m$ was chosen to be minimal with the property that $d(x, r_m) > 20 r_m$, we can find a $z \in B(x,r_{m-1})$ so that $d(y) \leq 20 r_m$, and thus, again because $d$ is $1$-Lipschitz, 
\begin{align*}
    d(y) \leq d(z) + 10r_m \leq 30 r_m. 
\end{align*}
We thus have $d(y) \approx r_m$ (up to a universal multiplicative constant) for any $y \in B(x, 10r_m)$, and the lemma follows. 
\end{proof}

\begin{lemma}\label{l:dt-80}
For $m \geq 0$, let $\dN_m^R$ be the sequence of nets defined in \eqref{e:net-m} satisfies conditions a) and c) of Theorem \ref{t:david-toro}. 
\end{lemma}

\begin{proof}
Condition a) holds by construction. Let $x \in \dN_m^R$. If $x \in \dN_m^R \cap \dN_{m-1}^R$, we are done. If $x \in\dN_m^R \setminus \dN_{m-1}^R$, we argue by contradiction. Suppose that for any $z \in \dN_{m-1}^R$, we have $|x-z| > 2r_{m-1}$. Note that 
\begin{align*}
    d(x, r_{m-1})= \inf_{y \in B(x, r_{m-1})} d(y) \leq \inf_{y \in B(x, r_m)} d(y) = d(x, r_m) \stackrel{\eqref{e:dt-20}}{\leq} r_m < r_{m-1}.
\end{align*}
This contradicts the fact that $\dN_{m-1}^R$ is a maximal $r_{m-1}$ separated net for $\dN_{m-2}^R$ $\cup$ $\{ d(z,r_{m-1}) \leq r_{m-1}\}$. Thus also condition b) holds. 
\end{proof}
\subsubsection{Choice of planes associated to nets}\label{sss:plane-choice}
 
We need an auxiliary lemma. 
\begin{lemma}\label{l:dt-30}
There exists a constant $c \geq 1$ such that the following holds. Let $m \geq 1$ and $x \in E$. Suppose that  $20 r_m \geq d(B)$, with $B=B(x, r_m)$. Then there exists a ball $B' \in \Vv\Gg$ such that
\begin{align}
     r_{B'} \approx_c r_m \quad \mbox{ and } \quad
     \dist(B, B') \lesssim_c r_m. 
\end{align}
\end{lemma} 
\begin{proof}
By definition of $d(\cdot)$, there is a ball $B' \in \Vv\Gg$ so that 
\begin{align*}
    r_{B'} + \dist(B, B') \leq 1.1 d(x, r_m) \stackrel{(\ref{e:dt-20})}{\leq} r_m.
\end{align*}
This $B'\in \Vv\Gg$ satisfies in particular
$\dist(B',B) \lesssim 1.1 r_m$ and $r_{B'} \leq 1.1 r_m$. If also $r_{B'} \geq \tfrac{1}{1.1}r_m$, then setting $c=1.1$ we are done. If $r_{B'} < \tfrac{1}{1.1}$, we choose the ball $B(x, \tfrac{1}{1.1}r_m) \in \Vv\Gg$. 
\end{proof}

\noindent
Let $m \geq 0$ and $x \in \dN_m^Q$. To each point $x_j^m$ we associate a plane $P=P_j^m$. We distinguish two cases: 
\begin{center}
\begin{gather}
    \mathrm{a)} \quad  x \in \dN_{m}^R \cap \dN_{m-1}^R,\label{e:dt-casea} \\
\mathrm{or} \nonumber \\
\mathrm{b)} \quad  20r_m \geq d(x, r_m). \label{e:dt-caseb} 
\end{gather}
\end{center}
Case a): let $i \in J_{m-1}^R$ be so that $x_j^m = x_i^{m-1}$. We set $P_j^m = P_i^{m-1}$. Case b): we let $P_j^m$ be the translate to $x_j^m$ of a best approximating plane of $E$ with respect to $d_{\Cc_0 B'} (\cdot, \cdot)$, where $B' \in \Vv\Gg$ is a ball that we find via Lemma \ref{l:dt-30}.

\bigskip

\noindent
\begin{corollary}\label{c:dt-10}
This choice immediately implies that condition b) of Theorem \ref{t:david-toro} holds. 
\end{corollary}

\subsubsection{Verification of \eqref{e:DT4}} To verify \eqref{e:DT4} we need to control the angles between the best approximating planes.

\begin{lemma}[\cite{villa2020tangent}, Lemma 4.12]\label{l:dt-lemma30}
Let $E \subset \R^n$ be lower content $(d, \cc_1)$-regular. Let $X_m=\{x_j^m\}$ be a sequence of maximal $r_m$-separated nets satisfying conditions a), b) and c) of Theorem \ref{t:david-toro}. Let $m \geq 0$ and $j \in J_m^R$, take either $\ell=m$ or $\ell=m-1$ and let $i \in J_\ell^R$ be so that $|x_j^m - x_i^\ell| < 100r_m$. Then, for $r_m/2 \leq s \leq 5000r_m$, we have
\begin{align}\label{e:dt-50}
    d_{x_j^m, s}(P_j^m, P_i^\ell) \lesssim \beta_{E}^{1, p} (x_j^m, 120r_m) + \beta_{E}^{1, p} (x_i^\ell, 120r_\ell),
\end{align}
where the constant behind $\lesssim$ depends on $d, \cc_1$.
\end{lemma}

\begin{lemma}\label{l:dt-lemma40}
Let $\ve>0$ the small constant in Subsection \ref{sss:nets} and let $\Cc_0 \geq 1$ (also appearing in \ref{sss:nets}) sufficiently large depending on $c \geq 1$ in Lemma \ref{l:dt-30}. There exists a universal constant $C \geq 1$ so that the following holds. For $m \geq 0$, let $x=x_j^m \in \dN_m^Q$. If $r_m \geq d(x, r_m)$, then 
\begin{align}\label{e:dt-60}
    d_{x, s} (P_j^m, E) < C \ve,
\end{align}
where $r_m \leq s \leq 0.99\Cc_0r_m$ and $P_j^m$ is the $d$-plane chosen in Subsection \ref{sss:plane-choice}.
\end{lemma}
\begin{proof}
Let $B'$ be the ball found in Lemma \ref{l:dt-30}, satisfying $B' \in \Vv\Gg$, $r_m \approx_c r_{B'}$ and $\dist(B', B(x,r_m)) \lesssim_c r_m$, for some numerical constant $c\geq 1$. Then, if we choose $\Cc_0 \geq 1$ sufficiently large with respect to $c \geq 1$, we have that
\begin{align*}
    B(x, s) \subset \Cc_0 B'.
\end{align*}
It is immediate to see that if two balls $B_1$ and $B_2$, both centered on a subset $E \subset \R^n$ and such that $B_2 \subset B_1$ and $r_{B_1} \approx_{C} r_{B_2}$ for some constant $C$, then 
\begin{align*}
    d_{B_2}(E, P) \lesssim_C d_{B_1}(E, P).
\end{align*}
(Here $P$ is a plane, say). We conclude therefore that 
\begin{align*}
    d_{x, s} (E, P_j^m) \lesssim d_{\Cc_0 B'} (E, P_j^m).
\end{align*}
Recall that $P_j^m$ is the translate to $x_j^m$ of a minimising plane $P'$ of $d_{\Cc_0 B'}(\cdot, E)$. It is easy to check that $$
d_{\Cc_0 B'}(E, P_j^m) \lesssim d_{\Cc_0 B'} (E, P') + d_{2\Cc_0 B'}(P', P_j^m).
$$
Because $B' \in \Vv\Gg$ the first term on the right hand side is $\leq C \ve$. For the very same reason, the second term, too, is $\leq C \ve$. This conclude the proof.
\end{proof}

\begin{lemma}\label{l:dt-lemma70}
Let $\ve>0$ and $\Cc_0 \geq 1$ be as above. Let $\dN_m^R$ be the sequence of nets constructed in \eqref{e:net-m}. For $m \geq 0$ and $i \in J_m^R$, and $\ell \in \N$ and $i \in J_\ell^R$ satisfying $\ell=m$ or $\ell = m-1$ and $x_j^m \in 100B_i^\ell$, we have
\begin{align*}
    d_{x_j^m, s} (P_j^m, P_i^\ell) \leq C \ve \quad \mbox{ for } \quad r_m \leq s \leq 0.99C_0r_m. 
\end{align*}
\end{lemma}
\begin{proof}\

\noindent
\textbf{Case 1.} Suppose that $r_m \geq d(x_j^m, r_m)$. 
\begin{itemize}[leftmargin=0.5cm]
    \item \textit{Case 1.1}: $r_\ell \geq d(x_i^\ell, r_\ell)$. From Lemma \ref{l:dt-lemma40}, there exists a universal constant $C$ such that $d_{x, r}(P, E) <  C\ve$ for $x=x_j^m, x_i^\ell$, $r=r_m, r_\ell$ and $P=P_j^m, P_i^\ell$. On the other hand, it is easy to check that $\beta_E^{d,1}(B) \lesssim_d \beta_{\infty, E}^d (B)$ for a lower content regular set $E$ and a ball $B$ centered on $E$. We refer the unconvinced reader to \cite{azzam2018analyst}, Lemma 2.11 (eq. 2.29); also, $\beta_{E, \infty}^d(B) \leq \inf_P d_B(E, P)$. Putting all this together (and choosing $\Cc_0$ larger than some absolute constant) we arrive at 
    \begin{align*}
        d_{x_j^m, s}(P_j^m, P_i^\ell) & \stackrel{\eqref{e:dt-50}}{\lesssim_{\cc_1, d}} \beta_{E}^{d, 1}(x_j^m, 120r_m) + \beta_E^{d, 1}(x_i^\ell, 120r_\ell)\\
        & \lesssim_{\cc_1, d} d_{x_j^m, 120r_m} (E, P_j^m) + d_{x_i^\ell, 120r_\ell} (E, P_i^\ell) \stackrel{\eqref{e:dt-60}}{\lesssim} \ve. 
    \end{align*}
    \item \textit{Case 1.2}. $r_\ell < d(x^\ell_i, r_\ell)$. Let $\ell'$ be the maximal integer for which $r_{\ell'} \geq d(x_i^\ell, r_{\ell'})$; in particular 
    \begin{align}\label{e:dt-70}
        d(x_i^\ell, r_{\ell'}) \approx r_{\ell'}. 
    \end{align}
    Since $x_j^m \in 100r_\ell$, it must hold that $x_j^m \in 10r_{\ell'}$. Observing that $d(x_i^\ell, r_{\ell'}) \approx r_{\ell'} \approx d(x_i^\ell, 10r_{\ell'})$, we conclude
    \begin{align*}
        r_\ell' \approx d(x_i^\ell, 10r_{\ell'}) = C \inf_{y \in B(x_i^\ell, 10r_{\ell'}) \cap E} d(y) \leq C d(x_j^m) \leq C r_m. 
    \end{align*}
    That is, there is a constant $C \geq 1$ so that $r_m \approx r_{\ell'}$. Moreover, 
    the choice of planes made in \ref{sss:plane-choice} implies that $P_i^\ell = P_{i'}^{\ell'}$ for some $i' \in J_{\ell'}^R$. We can therefore conclude as in Case 1.1.
\end{itemize}
\textbf{Case 2.} Suppose that $r_m < d(x_j^m, r_m)$. Cases 2.1 and 2.2 are similar to the ones above and we leave them to the reader. 
\end{proof}

\noindent
Lemma \ref{l:dt-80}, Corollary \ref{c:dt-10} and Lemma \ref{l:dt-lemma70} say that conditions a), b), c) and \eqref{e:DT4} of Theorem \ref{t:david-toro} are satisfied. We denote by $\left\{\Sigma_m^R\right\}_{m \geq 0}$ sequence of surfaces guaranteed by Theorem \ref{t:david-toro} (7), and by $\Sigma^R$ the surface from (8).  

\begin{lemma}\label{l:d(x)=0}
Notation as above. If $x \in \Cc_0 B_R $ is so that $d(x) =0$, then $x \in \Sigma^R\cap E$. 
\end{lemma}
\begin{proof}
That $d(x) =0$ implies the existence of a sequence $\{B_k\} \subset \Vv\Gg $ so that $r_{B_k} + \dist(x, B_k) \to 0$ as $k \to \infty$. By construction, there exists a sequence of points $x_j^{m_k}$ ($m_k \to \infty$ as $k \to \infty$), with $x_j^{m_k} \in \dN_{m_k}^R$ such that $|x_j^{m_k}- x| \to 0$ as $k \to \infty$. Thus for all $M \geq 1$, $x \in \overline{\bigcup_{m \geq M} \{x_j^m\}}$, and therefore $x \in E_\infty \subset \Sigma^R$ 
\end{proof}

\begin{lemma}
Let $R, \Sigma^R$ as above. Then there exists a constant $C \geq 1$ so that, for all $x \in \Cc_0B_R \cap E$, 
\begin{align}\label{e:dist-Sigma-R}
\dist(x, \Sigma^R) \leq C \ve d(x). 
\end{align}
\end{lemma}

\begin{proof}
If $d(x) = 0$, then \eqref{e:dist-Sigma-R} holds by Lemma \ref{l:d(x)=0}. Suppose $d(x) >0$; let $m \in\N$ be minimal with $d(x) > 20r_m$. Let $j \in J_m^R$ be so that $|x_j^m - x| < r_m$. For some $z_m \in \Sigma_m^R$, 
\begin{align*}
    \dist(x, \Sigma^R) \leq |x- z_m| + \dist(z_m, \Sigma^R) \lesssim |x-z_m| + \ve r_m, 
\end{align*}
by Theorem \ref{t:david-toro}, (9).  Let now $B \in \Vv\Gg$ be the ball from Lemma \ref{l:dt-30}. Precisely because $B \in \Vv\Gg$,  
\begin{align}\label{e:dt-101}
   d_B(E, P_B) < \ve.
\end{align}
From our choice of plane $P_j^m$, we have that \begin{align}\label{e:dt-102}
    d_{x_j^m, \Cc_0 r_m}(P_B, P_j^m) \lesssim \ve.
\end{align} 
Finally, it follows from Theorem \ref{t:david-toro}, (7), that 
\begin{align}\label{e:dt-103}
    d_{x_j^m, 49r_m} (\Sigma_m, P_j^m) \lesssim \ve. 
\end{align}
Putting together \eqref{e:dt-101}, \eqref{e:dt-102} and \eqref{e:dt-103}, we arrive at $|x-z_m| \lesssim r_m \ve$. Since $m$ was chosen minimal satisfying $20r_m \geq d(x)$, by Lemma \ref{l:dt-001}, we obtain $\dist(x, \Sigma^R) \lesssim \ve d(x)$. 
\end{proof}

\noindent
Set 
\begin{align*}
    E_0^R := \{x \in E \, |\, d(x)= 0 \}.
\end{align*}
We then consider the family of balls 
\begin{align}\label{e:dt-F_0}
   \dF_0 = \large\{ B \, |\, z_B \in \Cc_0 B_R \, \mbox{ and }\, r_B = r_m \mbox{ for } m \in \N \, \mbox{ minimal s.t. } \, d(z_B, r_m)  > 20 r_m \large\},
\end{align}
Note that
\begin{align*}
    \bigcup_{B \in \dF_0} B \cap E = E \setminus E_0.
\end{align*}
We choose a Vitali covering of $\dF_0$, denoting it $\dF$.  

\begin{lemma}\label{l:dt-300}
Let $B=B(x, r_m) \in \dF$ for some $m \in \N$. Then there exists a point $x_j^{m-1} \in \dN_{m-1}^R$ such that $|x-x_j^{m-1}| < r_{m-1}$ and moreover 
\begin{align}\label{e:dt-201}
    \left\{ y \in B(x_j^{m-1}, 2r_{m-1}) \supset B(x, r_m) \, |\, d(y,r_{m+1}) \geq 20 r_{m+1} \right\} = \emptyset. 
\end{align}
\end{lemma}
\begin{proof}
Since $m$ is minimal with $d(B) > 20r_m$, then $d(x, r_{m-1}) \leq 20r_{m-1}$. Recalling the definition of $\dN_{m-1}^R$, this implies that there exists a point $x_j^{m-1} \in \dN_{m-1}^R$ such that $|x-x_j^{m-1}| < r_{m-1}$. Moreover, using the fact that $d$ is $1$-Lipschitz and that $d(y) > 20 r_{m}$ for all $y \in B(x, r_m)$, one sees that $d(y) > 8r_m$ for all $y \in B(x_j^{m-1}, 3r_{m-1}) \supset B(x, r_m)$, and thus \eqref{e:dt-201}. 
\end{proof}

\begin{lemma}\label{l:dt-500}
Let $B=B(x, r_m) \in \dF$ as above. Then there exists a constant $C\geq 1$ so that for all integers $k> \floor{Cm}=: m'$, $\Sigma_{k}^R \cap B= \Sigma_{m'}^R \cap B$. In particular, $\Sigma^R \cap B= \Sigma^R_{m'}\cap B$ and thus, by Theorem \ref{t:david-toro} (7), $\Sigma^R \cap B$ coincides with a Lipschitz graph with small Lipschitz constant (over some plane).  
\end{lemma}
\begin{proof}
From Theorem \ref{t:david-toro} (9), we know that for all $y \in \Sigma_k$, $|\sigma_k(y)-y| \lesssim \ve_k(y)r_k$. Now, by Lemma \ref{l:dt-300}, there exists a point $z \in \dN_{m-1}^R$ so that $B(x, r_{m}) \subset B(z, 2 r_{m-1})$ and if $k > m$, then $\{d(x, r_k) \leq 20r_k\} \cap B(z, 2r_{m-1}) = \emptyset$. Thus for any point $x_i^k \in \dN_{k}^R \cap B(z, 2r_{m-1})$, there exists a $i' \in J_m$ so that $x_i^k = x_{i'}^m$, and therefore, by the choice of planes in Case a) \eqref{e:dt-casea}, $P_i^k = P_{i'}^m$. Then choosing $C$ larger than some absolute constant, we have that $\ve_{k}(y) = 0$ whenever $y \in B$ and $k \geq \floor{Cm}$. This proves the lemma.  
\end{proof}

\begin{lemma}\label{l:dt-400}
With $R$, $\Sigma^R$ as above, 
\begin{align}
    \dist(x, E) \lesssim \ve d(x)
\end{align}
for all $x \in \Cc_0 B_R \cap \Sigma^R$.
\end{lemma}
\begin{proof}
If $x \in \Cc_0 B_R$ is so that $d(x) = 0$, then $x \in \Sigma^R \cap E$. 
Let $m$ be minimal with $d(x) > r_m$. There exists a ball $B \in \Vv\Gg$ so that $r_B + \dist(x, B) \approx r_m$. In particular $x_B \in E$ is so that $| x-x_B | \lesssim r_m$. This implies that $d(x_B, r_m) \approx r_m$, and thus, by construction, there exists a point $x_j^{m-1} \in \dN_{m-1}^R$ so that $x \in B(x_j^{m-1}, r_{m-1})$. Now, since $x \in \Sigma^R$, there is a point $x' \in \Sigma_m^R$ so that $x = \lim_{K \to \infty} \sigma_{m+K} \circ \cdots \sigma_{m}(x')$. Then, by Theorem \ref{t:david-toro} (9), $|x-x'| \lesssim \ve r_m$. Moreover, by Theorem \ref{t:david-toro} (7), there is a point $p$ in $P_j^m$ so that $|x' - p| \lesssim \ve r_m$. Further, the plane was chosen so that $\dist(p, E) \lesssim \ve r_m$. Finally, because $m$ was chosen minimal satisfying $d(x)> r_m$, we also have $d(x) \approx r_m$. This proves the lemma.
\end{proof}

\noindent
\begin{lemma}
Let $R, \Sigma^R$ as above. Then
\begin{align}\label{e:dt-600}
    \hd(\Sigma^R) \lesssim \sum_{B \in \dF} r_B^d + \hd(E_0). 
\end{align}
\end{lemma}
\begin{proof}
Let $x \in \Sigma^R \cap \Cc_0B_R$. If $d(x)=0$, then $x \in E_0 \subset E$. If $d(x)>0$, let $m$ be minimal so that $d(x) >  20r_m$. Then by Lemma \ref{l:dt-400}, there exists a point $y \in E$ so that $|x-y|\lesssim \ve r_m$. Hence $m$ is minimal so that $d(y, r_m) > 20 r_m$, and thus $B(y, r_m) \in \dF_0$. This implies that $x \in B$ for some $B \in \dF$. All in all we see that if $ x \in \Sigma^R \cap \Cc_0 B_R$ then either $x \in E_0$ or $x \in B$ for some $B \in \dF$. 
Thus we write
\begin{align*}
    \hd(\Sigma^R \cap \Cc_0B_R) \leq \hd \left( \bigcup_{B \in \dF} \Sigma^R \cap B \right) + \hd(E_0).
\end{align*}
From Lemma \ref{l:dt-500}, $\Gamma^R \cap B$ is a Lipschitz graph with small constant for each $B$ and thus its $d$-dimensional measure is $\approx r_B^d$.
\end{proof}

\subsection{Construction of the covering surface $\Sigma$}

\noindent
Keep $Q_0, \ve, \Cc_0$ fixed. 
Set 
\begin{align*}
    \cG= \cG(Q_0, \ve, \Cc_0) := \{Q \in \dD(Q_0) \, |\, \inf_P d_{\Cc_0B_Q} (E, P) < \ve\}.
\end{align*}
For $R \in \cG$, define $\Tree(R)$ to be the stopping-time region constructed by inductively adding Christ-David cubes $Q$ to $\Tree(R)$ if either $Q=R$, or 
\begin{itemize}
    \item The parent $Q'$ of $Q$ satisfies $Q' \in \Tree(R)$.
    \item All siblings $Q'$ of $Q$ satisfy $\inf_{P} d_{\Cc_0 B_{Q'}}(E, P) < \ve$, where the infimum is taken over all $d$-dimensional planes. 
\end{itemize}
We denote by $\Stop(R)$ the family of maximal Christ-David cubes in $\Tree(R)$ satisfying $\ell(Q)< \tfrac{1}{20}d(Q)$, where $d(Q)=d_R(Q) = \inf_{x \in Q} d(x)$, with $d$ as in \eqref{e:dt-d}. Using a `cube' version of Lemma \ref{l:dt-30}, it is easy to see that, if $Q \in \Tree(R)$, then $\inf_P d_Q(E, P) \lesssim \ve$. Moreover, by the fact that $d$ is $1$-Lipschitz, the minimality of $m$ is in the construction of $\dF$ \eqref{e:dt-F_0}, and the maximality of the Christ-David cubes in $\Stop(R)$, for each $B \in \dF$, we can find a cube $Q_B$ in $\Stop(R)$ with  
\begin{align}\label{e:QB}
    \ell(Q_B) \approx r_B.
\end{align}
Now, let $\cF=\cF(Q_0, \ve, \Cc_0)$ be the family of maximal Christ-David cubes in $\cG=\cG(Q_0, \ve, \Cc_0)$. Set 
\begin{align*}
    \Top_0:=\cF.
\end{align*}
Suppose that $\Top_k$ has been defined. Then set
\begin{align*}
    \Top_{k+1} := \bigcup_{R \in \Top_k} \Stop(R). 
\end{align*}
Set $\Top := \cup_{k \geq 0} \Top_k$.
Note that 
\begin{align*}
    \sum_{R \in \BWGL(Q_0)} \ell(R)^d \approx \sum_{R \in \Top} \ell(R)^d. 
\end{align*}
Each Christ-David cube $R \in \Top$ satisfies $\inf_{P} d_{\Cc_0 B_R}(E, P) < \ve$. Then we can apply the construction above, and obtain a surface $\Sigma^R$. Set
\begin{align}
    \Sigma' := Q_0 \cup \left( \bigcup_{R \in \Top} \Sigma^R \cap 2 B_R \right). 
\end{align}
Note that if $R, S \in \Top$ then $E_0^R \cap E_0^S = \emptyset$. Thus
\begin{align}\label{e:measSIGMA1}
    \hd(\Sigma') \lesssim \hd(Q_0) + \sum_{R \in \Top} \hd(\Sigma^R) \stackrel{\eqref{e:dt-600}, \eqref{e:QB}}{\lesssim} \hd(Q_0) + \sum_{R \in \BWGL(Q_0)} \ell(Q)^d.
\end{align}
Finally, let $\cF(Q_0)$ to be the family of maximal Christ-David cubes $Q \in \dD(Q_0)$ so that $\inf_P d_{\Cc_0B_Q}(E, P) < \ve$. Then since $\sum_{Q \in \BWGL(Q_0, \ve)} \ell(Q)^d< + \infty$, we must have that 
\begin{align}\label{cheap-cover}
    \bigcup_{Q \in \cF} \Cc_0 B_Q \supset Q_0.
\end{align}
For each $Q \in \dD(Q_0)$ not contained in any cube belonging to $\cF(Q_0)$, we put
\begin{align*}
    \cC_Q^0 := \{I \in \Delta_k \, |\, I \cap \cC_0B_Q \neq \emptyset \mbox{ and } k=k(Q)\}.
\end{align*}
Then we set 
\begin{align}\label{e:SIGMA-0}
    \Sigma^0 := \bigcup_{\substack{Q: Q' \subsetneq Q \subset Q_0 \\ Q' \in \cF(Q_0)}} \bigcup_{I \in \cC_Q^0} \partial_d I. 
\end{align}
Finally, set
\begin{align}\label{e:SIGMA}
\Sigma= \Sigma(Q_0, \ve, \Cc_0) := \Sigma' \cup \Sigma^0. 
\end{align}
Note that if $Q$ is as in the first union, then $Q \in \BWGL(Q_0, \ve, \Cc_0)$. Thus, using \eqref{e:measSIGMA1}, 
\begin{align}\label{e:measSIGMA-2}
    \hd(\Sigma) \lesssim \hd(Q_0) + \sum_{Q \in \BWGL(Q_0, \ve, \Cc_0)} \ell(Q^d).
\end{align}
\begin{remark}\label{r:measSIGMA}
Note that the constant behind the symbol $\lesssim$ in \eqref{e:measSIGMA1} and \eqref{e:measSIGMA-2} depend only on $n, d, \ve$ and $\Cc_0$. 
\end{remark}

\subsection{The covering surface $\Sigma$ is topologically stable}
\begin{proposition}\label{p:SIGMA}
Let $Q_0 \in \dD$ and fix $\ve>0$. Let $\Sigma= \Sigma(Q_0, \ve)$ be defined as in \eqref{e:SIGMA}. Then we can find
\begin{itemize}
    \item $\alpha_0, \eta_0$ sufficiently small with respect to $\ve>0$;
    \item $\gamma_0$ smaller than some universal constant;
    \item $\delta_0$ sufficiently small depending on $\eta_0, \gamma_0, d$;
    \item $r_0= \ell(Q_0)$,
\end{itemize}
such that $\Sigma$ satisfies the topological condition (Definition \ref{d:TC}) with parameters $(r_0, \alpha_0, \eta_0, \delta_0, \gamma_0)$. 
\end{proposition}

\begin{lemma}\label{l:SIGMA-l1}
Let $\Sigma$ be the set defined in \eqref{e:SIGMA}. Let $\Sigma^R$ be a Reifenberg $\ve$-flat set corresponding to some Christ-David cube $R \in \Top$. There exist constants $\eta_0$ and $\alpha_0$ depending on $\ve$ so that the following holds true. Let $x \in \Sigma^R$ and $B(x,r) \subset 2 B_R$. Then 
\begin{align}\label{e:SIGMA-l1}
    \hd(B(x, (1-\eta_0)r) \cap \vp_1(\Sigma)) \gtrsim_{d, \eta_0} r^d, 
\end{align}
whenever $\{\vp_t\}$ is an $\alpha_0$-ALD for $B(x,r)$.  
\end{lemma} 

Recall that $\Sigma^R$ is a $d$-dimensional Reifenberg $\ve$-flat surface.
\begin{claim}\label{claim:reif-a}
Let $\{\psi_t\}$ be an $\alpha_0$-ALD with $\alpha_0$ to be fixed below depending on $\ve$. Then for all $x \in B(x,r) \cap \Sigma^R$, there exists $y \in \vp_1(\Sigma^R)$ such that
\begin{align*}
    |x-y| < 2\ve r.
\end{align*}
\end{claim}
\begin{proof}
Let $P$ be a $d$-plane for which $d_{x,r}(\Sigma^R, P) < 1.1\ve$ holds. Since $\vp_1$ is the continuous image of a $d$-plane, that is, $\vp_1(\Sigma^R) = \vp_1 \circ g (P_0)$ (where $P_0$ is the plane in \eqref{P_0}, which exists since $R \in \Top$), then there must exists a point $y \in \vp_1(\Sigma^R)$ which also satisfies $y \in \Pi_P^{-1}(B^d(\Pi_P(x),r))$, where $\Pi_P$ is the orthogonal projection $\R^n \to P$ and $B^d(\Pi_P(x),r)$ is $d$-dimensional ball in $P$. Now, suppose that the claim does not hold, and thus $B(x, 2\ve r) \cap \vp_1(\Sigma^R) = \emptyset$. But by the previous considerations, there exists a point $y \in \Pi_P^{-1}(B^d(\Pi_P(x),r)) \cap \vp_1(\Sigma^R)$. Since $\Sigma^R$ is Reifenberg $\ve$-flat, this points will have $\dist(y, \Sigma^R) > 2\ve r$. Choosing $\alpha_0 \leq \ve$, this contradicts the fact that $\{\vp_t\}$ is an $\alpha_0$-ALD for $B(x,r)$ (specifically \eqref{e:vp4}). 
\end{proof}
\noindent
\begin{claim}
Let $\eta_0>0$ be sufficiently small depending  on $\ve$. 
With the notation above, we have
\begin{align}\label{e:reif-b}
    \Pi_P(B(x, (1-\eta_0)r)\cap \vp_1(E)) \supset B^d(\Pi_P(x), (1-2\eta_0)r) \cap P.
\end{align}
\end{claim}
\begin{proof}
Let $\wt r= (1-\eta_0)r$.  Using Claim \ref{claim:reif-a}, we see that for any point $y \in \partial B(x,\wt r) \cap \vp_1(\Sigma^R)$, $|y-\Pi_P(y)|\leq 4 \ve$. Then by trigonometry, $|\Pi_P(x)- \Pi_P(y)| = |x- \Pi_P(y)| \geq (1-16\ve^2)\wt r$. From the continuity of $\vp_1$, it follows that for any point $p \in B(x, (1-16\ve^2)\wt r) \cap P$ there is a $y \in \vp_1(\Sigma^R)$ so that $\dist(y, P) < 4\ve r$. This proves the claim.
\end{proof}
\begin{proof}[Proof of Lemma \ref{l:SIGMA-l1}]
Since $\Pi_P$ is $1$-Lipschitz, we have 
\begin{align*}
    \hd(B(x, (1-\eta_0)r) \cap \vp_1(\Sigma)) & \stackrel{\eqref{e:SIGMA}}{\geq}  \hd(B(x, (1-\eta_0)r) \cap \vp_1(\Sigma^R))\\
    & \stackrel{\eqref{e:reif-b}}{\geq} \hd(B^d(\Pi_P(x), (1-2\eta_0)r)) = c(d,\eta_0)r^d. 
\end{align*}
\end{proof}

\begin{lemma}\label{l:SIGMA-l2}
Let $\Sigma$ be the set defined in \eqref{e:SIGMA}. Let $\Sigma^R$ be a Reifenberg $\ve$-flat set corresponding to some Christ-David cube $R \in \Top$. There exist constants $\eta_0$ and $\alpha_0$ depending on $\ve$, a constant $\gamma_0$ smaller than some absolute value and a constant $\delta_0< 1$ depending on $\eta_0, \gamma_0$ and $d$, so that the following holds true. Let $x \in \Sigma^R \cap 2 B_R$ and $r \in (0, \ell(R))$. Then we can find a ball $B$ satisfying \eqref{e:Bcent}, \eqref{e:r_B} and \eqref{e:Bin}, so that  
\begin{align}\label{e:SIGMA-l2}
    \hd(B(x_B, (1-\eta_0)r(B)) \cap \vp_1(\Sigma)) \geq\, \delta_0 r^d. 
\end{align}
\end{lemma} 
\begin{proof}
Using Reifenberg $\ve$-flatness of $\Sigma^R$ and Claim \ref{claim:reif-a}, we find a ball $B$ centered on $\Sigma^R$ with $r(B) \geq \frac{1}{10} r$ and such that $B \subset 2 B_R$ and $B \subset B(x, r)$. Then $B$ satisfies the hypotheses of Lemma \ref{l:SIGMA-l1}. Thus
\begin{align*}
    \hd(B(x_B, (1-\eta_0)r(B)) \cap \vp_1 \stackrel{\eqref{e:SIGMA-l1}}{\geq} c(d,\eta_0) r(B)^d \geq c(d, \eta_0, \gamma_0) r^d.
\end{align*}
Choosing $c(d, \eta_0, \gamma_0)= \delta_0$, proves the lemma. 
\end{proof}

\begin{lemma}\label{l:SIGMA-l3}
Let $\Sigma$ be the set defined in \eqref{e:SIGMA}. Let $\Sigma^R$ be a Reifenberg $\ve$-flat set corresponding to some Christ-David cube $R \in \Top$. There exist constants $\eta_0$ and $\alpha_0$ depending on $\ve$, a constant $\gamma_0$ smaller than some absolute value and a constant $\delta_0< 1$ depending on $\eta_0, \gamma_0$ and $d$, so that the following holds true. Let $x \in \Sigma^R \cap 2 B_R$ and $\Cc_0 \ell(R) < r \leq \ell(Q_0)$. Then we can find a ball $B$ satisfying \eqref{e:Bcent}, \eqref{e:r_B} and \eqref{e:Bin}, so that  
\begin{align}\label{e:SIGMA-l3}
    \hd(B(x_B, (1-\eta_0)r(B)) \cap \vp_1(\Sigma)) \geq\, \delta_0 r^d. 
\end{align}
\end{lemma} 
\begin{proof}
Recall from Lemma \ref{l:dt-400}, that there exists a point $y \in E$ such that $|x-y| \lesssim \ve d(x) \leq \ve \ell(R)$. There exists a universal constant $c>0$ and a Christ-David cubes $Q \in \dD(Q_0)$ satisfying $x \in 2B_Q \cap E$ and $cB_Q \subset B(x,r) \subset 2B_Q$. We consider two distinct cases. 

\noindent
\textbf{Case 1.} Suppose $Q \in \Tree(R')$ for some $R' \in \Top$. Then by Lemma \ref{l:dt-400}, we have $\dist(y, \Sigma^{R'}) \leq \ve \ell(Q)$. Let $z \in \Sigma^{R'}$ be such that $|z-y| \lesssim \ve \ell(Q)$. Then $|z-x| \lesssim \ve \ell(Q) + \ve\ell(R) \lesssim \ve \ell(Q)$. Then consider the ball $B$ with $x_B = y$ and $r_B= c' r$, where  $c'$ can be chosen smaller than some universal value and depending on $c$ so that $B\subset B(x,r)$. Now, if $B \subset \Cc_0B_{R'}$, then we argue as in Lemma \ref{l:SIGMA-l1}. Thus \eqref{e:SIGMA-l1} gives us 
\begin{align}
    \hd(B(x_B, (1-\eta_0)r_B) \cap \vp_1(\Sigma)) \gtrsim_{d, \eta_0} r_B \gtrsim_{d, \eta_0, c'} r^d. 
\end{align}
On the other hand, if $B$ is not contained in $\Cc_0 B_{R'}$, we use Lemma \ref{l:SIGMA-l2} to find a further ball $B' \subset B$ satisfying \eqref{e:SIGMA-l2}.

\noindent
\textbf{Case 2.}

Suppose now that $Q$ is so that $cB_Q \subset B(x,r) \subset 2B_Q$ but that  $Q \notin \Tree(R')$ for any $R' \in \Top$. 
Then there exists a Christ-David cube $Q' \in \cF$ such that $Q' \subsetneq Q$ and $x \in 1.5B_{Q'}$. But recall from the definition of $\Sigma^0$ \eqref{e:SIGMA-0}, that the dyadic cubes in $\cC_Q^0$ cover $B(x,r)$. 
It can be easily then shown that 
\begin{align*}
    \hd(B(x_B, (1-\eta_0)r_B) \cap \vp_1(\Sigma^{Q'})) \gtrsim_{d, \eta_0} r_B \gtrsim_{d, \eta_0} r^d,
\end{align*}
whenever $\vp_t$ is an $\alpha_0$-ALD for $B$ (with $\alpha_0$ chosen sufficiently small with respect to $\ve$). The ball $B=B(x_B, r_B)$ can be chosen so that $r_B \approx r$ up to a universal multiplicative constant. This implies at once \eqref{e:SIGMA-l3} and thus concludes the proof of Lemma \ref{l:SIGMA-l3}.
\end{proof}

\begin{lemma}
Keep the notation as above. Suppose that $x \in Q_0$ and $0< r < \ell(Q_0)$. Then if $\alpha_0$ and $\eta_0$ are chosen sufficiently small w.r.t $\ve$, if $\gamma_0$ is chosen smaller than a universal constant, and if $\delta_0$ is chosen appropriately (depending on $d, \gamma_0, \eta_0$), then we can find a ball $B$ centered on $\Sigma$, with $B \subset B(x,r)$ and $\gamma_0 r \leq r_B \leq r$ such that
\begin{align}
    \hd(B(x_B, (1-\eta_0)r_B) \cap \vp_1(\Sigma)) \geq \delta_0 r^d,
\end{align}
whenever $\vp_t$ is an $\alpha_0$-ALD for $B$.
\end{lemma}
\noindent
The proof of this Lemma follows the same arguments as Lemma \ref{l:SIGMA-l1}, \ref{l:SIGMA-l2} and \ref{l:SIGMA-l3} and is left to the reader. This concludes the proof of Proposition \ref{p:SIGMA}.

\subsection{Conclusion}
Let us finish this section by proving Theorem \ref{t:SIGMA-cor-in-sec}, and thus Corollary \ref{c:TST-lcr-stable}. 
Given $E, Q_0, \ve, \Cc_0$ as in the statement of Theorem \ref{t:SIGMA-cor-in-sec}, we constructed a stable $d$-surface $\Sigma$ with constants $(r_0, \alpha_0, \delta_0, \eta_0, \gamma_0)$ as in Proposition \ref{p:SIGMA}. It is immediate from the construction \eqref{e:SIGMA} that $Q_0 \subset \Sigma$. Moreover, we have
\begin{align}
    \hd(\Sigma) & \stackrel{\eqref{e:measSIGMA-2}, \,{\rm Remark } \ref{r:measSIGMA}}{\lesssim_{n,d,\ve,\Cc_0}} \hd(Q_0) + \sum_{R \in \BWGL(Q_0, \ve, \Cc_0)} \ell(Q_0)^d \\
    & \stackrel{{\rm Theorem} \,  \ref{t:AS-real}}{\lesssim_{n,d,p,\aA,p, \Cc_0,\ve}} {\rm BETA}_{E, d, p, \aA} (Q_0) \ell(Q)^d \\
    & \stackrel{{\rm Theorem} \, \ref{t:topo-main}}{\leq} C\, \hd(\Sigma),
\end{align}
where $C=C(n,d,p,\aA, \Cc_0, \alpha_0, \eta_0, \gamma_0, \delta_0)= C(n, d, p, \aA, \Cc_0)$, since $\alpha_0, \eta_0, \gamma_0, \delta_0$ depend on either a universal constant or $\ve$. In particular, $C$ does not depend on $E$. 
\section{An application to uniformly non-flat sets}\label{s:non-flat}
In \cite{david2004hausdorff}, David proved that if $E$ satisfies a stronger\footnote{David required that the ball $B$ equal to $B(x,r)$ in our Definition \ref{d:TC}.} version of the ($d$-dimensional) topological condition and it is uniformly non-flat (with respect to the $L^\infty$ coefficients $\beta_\infty$), then it must have dimension strictly larger that $d$. As mentioned in the introduction, David's result was in the spirit of a previous result by 
 Bishop and Jones about uniformly wiggly, or uniformly non-flat, sets. 
 \begin{definition}\label{d:nonflat}
 A set $E \subset \R^n$ is called \textit{uniformly wiggly} or \textit{uniformly non-flat} (with parameter $\beta_0$) if for all cubes $Q \in \dD_E$, we have that
 \begin{align*}
     \beta_\infty(Q) > \beta_0 >0.
 \end{align*}
 \end{definition}
 \begin{remark}
 Clearly, this definition can be recast in terms of different types of $\beta$ numbers, such as the content beta numbers which we have been using so far in this paper. 
 \end{remark}
 \noindent
 Let us now recall the result of Bishop and Jones. 
 \begin{theorem}[\cite{bishop1997wiggly}, Theorem 1.1] \label{t:nonflatBJ}
 Let $E \subset \R^2$ be a compact, connected subset which is uniformly wiggly with parameter $\beta_0$. Then $\dim(E) > 1+ C \beta_0^2$, where $C$ is an absolute constant.
 \end{theorem}
 \noindent
 Let us go back to David's result. His is
 a generalisation of Bishop and Jones's Theorem. However, it is of qualitative nature, and the dependence of the lower bound on the parameter $\beta_0$ is not explicit. It is also proven with respect to a slightly stronger topological condition and with respect to $\beta_\infty$. In this section we give a generalisation of Bishop and Jones Theorem where such a dependence is made explicit. This result is a fairly immediate application of Corollary \ref{t:topo-main} and of the scheme of proof from \cite{bishop1997wiggly}. 
 
\begin{theorem} \label{t:high-dim}
Let $E \subset \R^n$ be a topologically stable $d$-surface. Let $R \in \dD$ be such that, for any $Q \in \dD(R)$, we have that
\begin{align}
    \beta_E^{p,d}(C_0 Q)^2 > \beta_0 >0.
\end{align}
Then
\begin{align}
    \dim(R) > d + c \beta_0^2. 
\end{align}
\end{theorem}
\noindent
Here $\dim$ stands for Hausdorff dimension. Lemma 2.12 in \cite{azzam2018analyst} says that $\beta_\infty^d (0.5 B) \leq c(d) \beta^{1,d}_E(B)^{\frac{1}{d+1}}$ whenever $B$ is a ball centered on a lower content $d$-regular set $E$. (Here $\beta^d_\infty$ is the coefficient as in the Introduction, except that we infimize over $d$-planes). We immediately obtain the following corollary. 
\begin{corollary}
Let $E \subset \R^n$ be a stable $d$-surface. If $R \in \dD(E)$ is uniformly non-flat \textup{(}with respect to $\beta_\infty^d$ and parameter $\beta_0$\textup{)}, then $dim(R) > d + C \beta_0^{\frac{2}{d+1}}$. 
\end{corollary}
\noindent
To arrive at Corollary \ref{c:non-flat} from this is immediate. 
The scheme of the proof of Theorem \ref{t:high-dim} is the same as that of Bishop and Jones. We also used a clear summary of such proof to be found in Garnett and Marshall's book, \cite{garnett2005harmonic}, page 429. For this reason, we only sketch the proof. 
\begin{proof}
Given a stable $d$-surface, a cube $R \in \dD(E)$ and an integer $m \geq 0$, we put
\begin{align*}
    \beta_m(R) =  \sum_{Q \in \dD_m(R)} \beta_E^{p,d}(Q) \ell(Q)^d. 
\end{align*}
Next, we consider 
\begin{align} \label{e:Delta-ck}
    \Delta_{k,c}(R) :=  \ck{ I \in \Delta \, |\, I \cap R \neq \emptyset \mbox{ and } \ell(I) = c\,2^{-k}},
\end{align}
where $c<1$ is a constant which is a power of $2$ and will be fixed later (it will depend on the parameter $\lambda>0$ coming from Theorem \ref{t:Christ}). 
We then put
\begin{align*}
    E_{R, k} := \bigcup_{I \in \Delta_{k,c}(R)} \partial_d I. 
\end{align*}

\noindent
\textbf{Claim 1.} There exists a constant $C_5$ so that, if 
\begin{align}\label{e:N-0-scale-R}
    R \in \dD_{N_0}(E)
\end{align}
with $N_0 \leq k$, then
\begin{align} \label{e:BJ-a}
    C_5\ps{\ell(R)^d + \sum_{m=N_0}^k \beta_m (R)} \leq \hd(E_{R, k}).
\end{align}
To see this, note first that because $E$ satisfies the topological condition \eqref{e:TC} with parameters $r_0, \alpha_0, \eta_0, \delta_0, \gamma_0$, then $E_{R, k}$ must also be a stable $d$-surface with comparable parameters (up to constants). This can be seen with a proof akin to that of Lemma \ref{l:TC-STC}. Hence, we can apply Corollary \ref{t:topo-main} to see that
\begin{align*}
    \hd(E_{R,k}) \approx \beta_{E_{R,k}, C_0, p}(R),
\end{align*}
where the constants behind $\approx$ are as in the statement of Corollary \ref{t:topo-main}. 
We can now check \eqref{e:BJ-a}: we have that
\begin{align*}
    \hd(E_{R,k}) \approx \diam(E_{R,k})^d + \sum_{P \in \dD_{E_{R,k}}} \beta^{d,p}_{E_{R,k}}(C_0 P)^2 \,\ell(P)^d.
\end{align*}
By construction, we immediately see that $\diam(E_{R, k})^d \approx \ell(R)^d$. On the other hand, consider a cube $Q \in \dD_{E}$, such that $\ell(Q) > c 2^{-k}$, for $c<1$ as in \eqref{e:Delta-ck}. If we choose $c$ sufficiently small, we can apply Lemma \ref{lemma:azzamschul} with $E_1=E$ and $E_2 = E_{R,k}$, to obtain
\begin{align*}
    \beta_{E}^{d,p}(C_0P) \lesssim \beta^{d,p}_{E_{R,k}}(2C_0 P) + \ps{ \frac{1}{\ell(P)^d} \int_{2C_0 B_P} \ps{\frac{\dist(y, E_{R,k})}{\ell(P)}}^p \, d \mathcal{H}_\infty^d}^{\frac{1}{p}}.
\end{align*}
Thus we see that
\begin{align*}
    & \sum_{\substack{P \in \dD_{E}\\ \ell(P) > c 2^{-k}}} \beta_{E}^{d,p} (C_0 P) \\
    & \lesssim \sum_{\substack{P' \in \dD_{E_{R, k}} \\ \ell(P') \gtrsim c 2^{-k}}} \beta^{d,p}_{E_{R,k}}(2C_0 P') + \sum_{\substack{P \in \dD_{E} \\ \ell(P) > c 2^{-k}}} \ps{ \frac{1}{\ell(P)^d} \int_{2C_0 B_P} \ps{\frac{ \dist(y, E_{R,k})}{\ell(P)}}^{p} \, \hdc(y) }^{\frac{1}{p}}.
\end{align*}
With a calculation similar to that from \eqref{eq:SS_distyER} to \eqref{e:estBeta1}, we obtain that the second sum above is $ \lesssim \ell(R)^d$. This then gives 
\begin{align*}
    \hd(E_{R, k}) & \approx \ell(R)^d + \sum_{P \in \dD_{E_{R,k}}} \beta_{E_{R, k}}(C_0 P)^2 \ell(P)^d \\
    &   \gtrsim \ell(R)^d + \sum_{\substack{P' \in \dD_{E_{R,k}} \\ \ell(P') \geq c 2^{-k}}} \beta_{E_{R,k}} + \sum_{\substack{P \in \dD_{E} \\ \ell(P) > c 2^{-k}}} \ps{ \frac{1}{\ell(P)^d} \int_{2C_0 B_P} \ps{\frac{ \dist(y, E_{R,k})}{\ell(P)}}^{p} \, \hdc(y)}^{\frac{1}{p}}\\
    & \geq C_5\ps{ \ell(R)^d + \sum_{\substack{P' \in \dD_{E_{R,k}} \\ \ell(P') \geq c 2^{-k}}} \beta_{E_{R,k}}}. 
\end{align*}
This proves \eqref{e:BJ-a}.

\noindent
\textbf{Claim 2.}
Let $N$ an integer so that $N > N_0$ (recall that $N_0$ is roughly the scale of $R$, see \eqref{e:N-0-scale-R}). Consider a dyadic cube $I_N \in \Delta_N(\R^n)$ for which $\ell(I_N) < \ell(R)/10$ and such that $\frac{1}{3} I_N \cap E \neq \emptyset$. For $k > N$, we have
\begin{align*}
    \sum_{m=N}^k \beta_{m}(R \cap I_N)^2 \geq (k-N) \beta_0^2 2^{-dN}.
\end{align*}
By $\beta_{m}(R \cap I_N)$ here we mean that we sum over those cubes $Q \in \dD_m(R)$ such that $Q \cap I_N \neq \emptyset$.
To see this, note first that by lower regularity of $E$, there are at least $2^{d(m-N)}$ (up to a a constant depending on the lower regularity parameter) dyadic cubes $J$ of generation $m$ (with $m>N$) such that $J \subset I_N$ and $J \cap E \neq \emptyset$. Hence since $E$ is uniformly non-flat, we see that if $N \leq m \leq k$,  
\begin{align*}
    \beta_m(R \cap I_N)^2 & =  \sum_{\substack{Q \in \dD_m(R) \\ Q \cap I_N \neq \emptyset}} \beta_{E}(C_0 Q)^2 \ell(Q)^d \\
    & \geq  \beta_0^2 \sum_{\substack{Q \in \dD_m(R) \\ Q \cap I_N}} \ell(Q)^d \\
    & \approx \beta_0^2 \sum_{\substack{J \in \Delta_{m,c}(R)\\ J \subset I_N}} \ell( J)^d \\
    & \gtrsim_{c} \beta_0^2 2^{d(m-N)} 2^{-dm} \approx_c \beta_0^2 2^{-dN}.
\end{align*}
Hence, we have that
\begin{align*}
    \sum_{m=N}^k \beta_m(R \cap I_N)^2 \gtrsim_{c} (k-N) \beta_0^2 \, 2^{-dN},
\end{align*}
and so, using \eqref{e:BJ-a},
\begin{align*}
    \hd(E_{R,k}\cap I_N) \gtrsim_{C_5, c} (k-N) \, \beta_0^2 \, 2^{-dN}.
\end{align*}

Let now $\ck{z_j}$, $j$ in some index set $A$, be a maximal $2^{-k}$-separated net of $E_{R,k}\cap I_N$ such that 
\begin{align} \label{e:max-net-zj}
    \bigcup_{j \in A} B(z_j, 2^{-k+2}) \supset E_{R,k} \cap I_N.
\end{align}
Then there exists a constant $C_7$ (depending only on $n$) so that
\begin{align*}
    \hd(E_{R,k} \cap I_N) \leq C_7\,c^d 2^{-dk}\, \text{Card}(A).
\end{align*}
Thus we obtain 
\begin{align*}
    C_7 \, c^d 2^{-dk} \, \text{Card}(A) \gtrsim_{C_5, c} (k-N) \, \beta_0^2 2^{-dN}, 
\end{align*}
and therefore
\begin{align} \label{e:CardA}
    \text{Card}(A) \gtrsim_{C_5,C_7,c} (k-N) \, \beta_0^2\, 2^{d(k-N)}.
\end{align}
Since $k$ was an arbitrary integer with $k \geq N$, we can choose it so that
\begin{align*}
 \kappa:= k-N \approx \frac{1}{\beta_0^2}.
\end{align*}
Hence we see from \eqref{e:CardA} that
\begin{align} \label{e:CardA-b}
    \text{Card}(A) \geq 2^{(d+ c' \beta_0^2)\kappa},
\end{align}
where $c'=c'(C_7, C_5, c)$.
\noindent
We now apply this construction recursively for each $N > N_0$, as follows. For $N_0$, we put 
\begin{align*}
    \dS_0 := \ck{ I \in \Delta_{N_0+\kappa}(R) \, |\, \exists j \in A \mbox{ s.t. } z_j \in I}
\end{align*}
Then for each  $I \in \dS_0$, we find a maximal net $\{z_j\}_{j \in A}$ as in \eqref{e:max-net-zj}; the cardinality of this net will be again as in \eqref{e:CardA-b}. We put the relative cubes in the subfamily
\begin{align*}
    \dS(I):= \ck{ J \in \Delta_{N_0 + 2\kappa} \, |\, \exists j \in A \mbox{ s.t. }z_j \in J}.
\end{align*}
We then put
\begin{align*}
    \dS_1:= \bigcup_{I \in \dS_0} \dS(I).
\end{align*}
Having defined $\dS_{j-1}$, we set
\begin{align*}
    \dS_{j} := \bigcup_{I \in \dS_{j-1}} \dS(I), 
\end{align*}
where $\dS(I)=\{ j \in \Delta_{N_0 + j\kappa} \, |\, \exists j \in A \mbox{ s.t. } z_j \in J\}$.
Let us record that for each $j \in \N$, we have 
\begin{enumerate}
    \item Each $J \in \dS_j$, is a subset of some $I \in \dS_{j-1}$.
    \item Each $I \in \dS_{j-1}$ contains at least $2^{(d+ c' \beta_0^2) \kappa}$ cubes $I \in \dS_{j}$ (as in \eqref{e:CardA}).
    \item For each $j \in \N$, if $I \in \dS_j$, we have $I\cap R \neq \emptyset$.
\end{enumerate}

\noindent
\textbf{Claim 3.} If $R$ satisfies (1)-(3), then 
\begin{align*}
    \dim(R) > d+c' \beta_0^2.
\end{align*}
To prove this claim, we define the $\mu$ on the elements $I$ of $\dS_j$, for $j \geq 0$, by 
\begin{align*}
    \mu(I) = \text{Card}(A)^{-j} \leq 2^{-j\kappa(d+c'\beta_0^2)}.
\end{align*}
One can then check that $\spt(\mu)= E$ and that $\mu(R) = 1$. Then, by Frostman's Lemma (Theorem 8.8 in \cite{mattila}), we have that
\begin{align*}
    \mathcal{H}^{d+ c'\beta_0^2}(R) >0.
\end{align*}
This completes the proof of Theorem \ref{t:high-dim}.
\end{proof}
\section{Appendix}
Recall the statement of Lemma \ref{l:meta1}. 
\begin{lemman}
Let $S$ be a cube in $\wt \Stop(Q)$ for some $Q \in \Next(R)$, $R \in \Top(k_0)$. Then there exists a dyadic cube $I_S := I \in \cC_{Q}$ so that $I_S \subset \frac{1}{2}B_S$ and $\ell(I_S) \sim \tau \ell(S)$.
\end{lemman}
\begin{proof}[Proof of Lemma \ref{l:meta1}]
   Let $\zeta_S$ be the center of $S$. Then there exists a dyadic cube $I \in \cC_{Q}$ such that $\zeta_S \in I$; thus for $I$ we have $\dist(I, S)=0$, and therefore $d_{Q}(I) \leq \dist(I, S) + \ell(S) = \ell(S)$. In other words, when computing $d_{Q}(I)$, it suffices to minimise over all cubes $T$ such that 
   \begin{align*}
       \dist(I, T) + \ell(T) \leq \ell(Q).
   \end{align*}
   But note that since $S$ is a minimal cube in $\wt \Tree(Q)$, we must have that $T \subset S^c$. Recall also that, by Theorem \ref{t:Christ},  $E \cap B(\zeta_S, c_0 \ell(S)) \subset S$. If we let $\tau$ be small enough, we can insure that $I \subset B(\zeta_S, \frac{c_0 \ell(S)}{2})$; hence we see that \begin{align*}
        \dist(I, T) \gtrsim \ell(S),
   \end{align*}
  and therefore $\tau^{-1} \ell(I) \sim  d_{Q}(I) \gtrsim \ell(S) \geq \ell(I)$ 
\end{proof}

\begin{lemman}
Let $I \in \cC_Q$ for $Q \in \Next(R)$, $R \in \Top(k_0)$. Then there exists a cube $Q_I \in \wt \Tree(Q)$ so that 
\begin{align}
    & \ell(I) \leq \ell(Q_I) \leq c \tau^{-1} \ell(I);\label{e:IQ-I-a} \\
    & \dist(I, Q_I) \leq c \tau^{-1} \ell(I) \label{e:IQ-I-b}. 
\end{align}
\end{lemman}
\begin{proof}[Proof of Lemma \ref{l:I-QI}]
Recall that $d_Q(I) \tau \sim  \ell(I)$. Now, by definition of $d_Q(I)$, there exists a cube $Q' \in \wt \Stop(Q)$ such that $\dist(I, Q') + \ell(Q) \leq 1.5 d_Q(I) \sim \tau^{-1} \ell(I)$. This immediately implies \eqref{e:IQ-I-b} and the second inequality in \eqref{e:IQ-I-a}. As for the first one, if it doesn't hold, it suffices to take some ancestor of $Q'$ in $\wt \Tree(Q)$. We then let this ancestor to be $Q_I$.
\end{proof}

\bibliography{bibliography}
\bibliographystyle{halpha-abbrv}
\end{document}